\documentclass[11pt]{article}

\usepackage{pratik}

\title{Failures and Successes of Cross-Validation for\\Early-Stopped Gradient Descent} 

\newcommand{\footremember}[2]{%
    \footnote{#2}
    \newcounter{#1}
    \setcounter{#1}{\value{footnote}}%
}
\newcommand{\footrecall}[1]{%
    \footnotemark[\value{#1}]%
}

\author{
    Pratik Patil\footremember{equal}{Equal
      contribution.}\footremember{berkeleystats}
    {Department of Statistics, University of California, Berkeley, CA 94720,
      USA.} \\ {\small\texttt{pratikpatil@berkeley.edu}} 
    \and Yuchen Wu\footrecall{equal} \footremember{wharton}
    {Department of Statistics and Data Science, Wharton School, 
    University of Pennsylvania, PA 19104, USA.} \\ 
    {\small\texttt{wuyc14@wharton.upenn.edu}} 
    \and Ryan J.\ Tibshirani\footrecall{berkeleystats} \\ 
    {\small\texttt{ryantibs@berkeley.edu}} 
}

\date{\vspace{-20pt}}

\begin{document}
\maketitle

\begin{abstract}
We analyze the statistical properties of generalized cross-validation (GCV) and
leave-one-out cross-validation (LOOCV) applied to early-stopped gradient descent
(GD) in high-dimensional least squares regression. We prove that GCV is
generically inconsistent as an estimator of the prediction risk of early-stopped
GD, even for a well-specified linear model with isotropic features. In contrast,
we show that LOOCV converges uniformly along the GD trajectory to the prediction 
risk. Our theory requires only mild assumptions on the data distribution and
does not require the underlying regression function to be linear. Furthermore,
by leveraging the individual LOOCV errors, we construct consistent estimators
for the entire prediction error distribution along the GD trajectory and
consistent estimators for a wide class of error functionals. This in particular
enables the construction of pathwise prediction intervals based on GD iterates 
that have asymptotically correct nominal coverage conditional on the training
data.
\end{abstract}

\section{Introduction}

Cross-validation (CV) is a widely used tool for assessing and selecting models 
in various predictive applications of statistics and machine learning. It is
often used to tune the level of regularization strength in \emph{explicitly 
  regularized} methods, such as ridge regression and lasso. In general, CV error
is based on an iterative scheme that allows each data sample to play a role in
training and validation in different iterations. Minimizing CV error helps to
identify a trade-off between bias and variance that favors prediction accuracy 
\citep{hastie2009elements}.    

Meanwhile, especially in the modern era, techniques such as gradient descent
(GD) and its variants are central tools for optimizing the parameters of machine
learning models. Even when applied to models without explicit regularization,
these algorithms are known to induce what is called \emph{implicit
  regularization} in various settings \citep{bartlett2021deep, belkin2021fit,
  ji2019implicit, nacson2019stochastic}. For example, in the simplest case of
least squares regression, GD and stochastic GD iterates bear a close connection
to explicitly regularized ridge regression estimates
\citep{suggula2018connecting, neu2018iterate, ali2019continuous,
  ali2020implicit}.   

This naturally leads to the following question:   

\begin{center}
\emph{Can we reliably use CV to assess model performance along the trajectory of
  iterative algorithms?}  
\end{center}

An affirmative answer to this question would enable the use of cross-validation
to determine when to stop the GD training procedure, preventing overfitting and 
appropriately balancing the level of implicit regularization. Motivated by this,
we investigate the statistical properties of two popular CV procedures, namely
generalized cross-validation (GCV) and leave-one-out cross-validation (LOOCV),
along the gradient descent trajectory in high-dimensional linear regression. 

\begin{figure}
\centering
\includegraphics[width=0.495\textwidth]{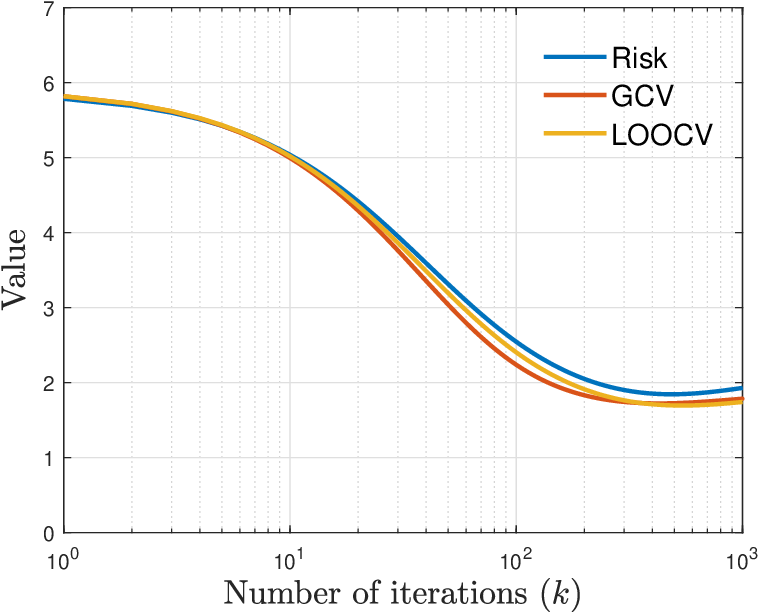}
\includegraphics[width=0.495\textwidth]{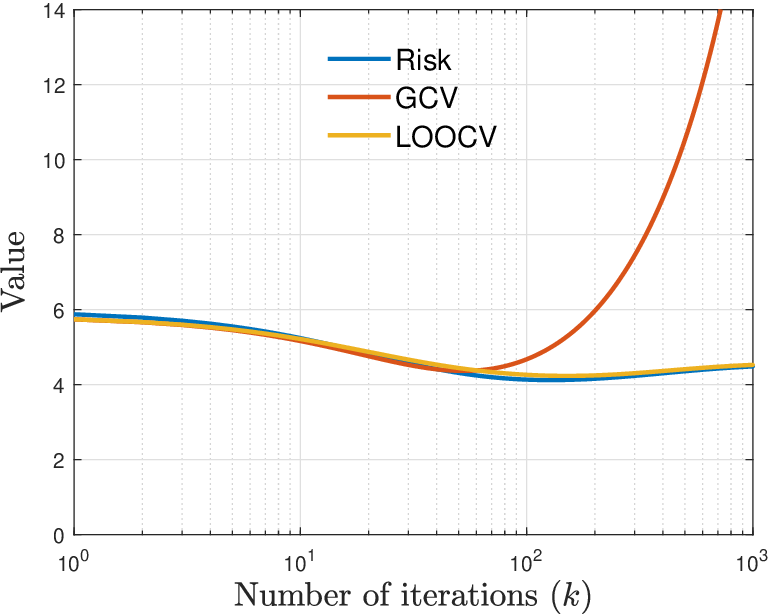}
\caption{\textbf{GCV can perform poorly in overparameterized problems, yet 
    LOOCV gives accurate risk estimates.} We investigate the risk of
  early-stopped gradient descent, applied to the least squares loss, as a
  function of iteration number. The \emph{left} panel shows an
  underparameterized experiment with $n = 3000$, $p = 1500$, and the
  \emph{right} panel an overparameterized experiment with $n = 3000$, $p =
  6000$. In both cases, the data is generated from a linear model with i.i.d.\
  standard normal features, a true signal vector with $\ell_2$ norm of $5$, and
  noise standard deviation of $1$. GD uses a constant step size of $0.01$. In
  the overparameterized case, we can see that the GCV risk estimate deviates
  wildly from the true risk, whereas LOOCV remains accurate throughout the
  entire path.}         
\label{fig:gcv-inconsistency-with-loocv-n2500-main}
\end{figure}

Previously, it has been noted that some common variants of CV: split-sample
validation and $K$-fold CV with small $K$ (such as $5$ or $10$), can suffer from
significant bias when the number of observations and features scale 
proportionally \citep{rad_maleki_2020, rad_zhou_maleki_2020}. Although LOOCV in   
most cases mitigates bias issues, it is typically computationally expensive to 
implement. Fortunately, for estimators that are linear smoothers (linear in the
response vector), GCV serves as an efficient approximation to LOOCV in classical
low-dimensional problems \citep{golub1979generalized,
  jansen1997generalized}. Furthermore, recent work has shown that both LOOCV and GCV
are consistent for estimating the out-of-sample prediction risk of ridge
regression in high-dimensional settings
\citep{patil_wei_rinaldo_tibshirani_2021, patil2022estimating,
  wei_hu_steinhardt_2022, han2023distribution}.   

Noting that for least squares loss the GD iterates are linear smoothers, and
recalling the connection between ridge regression and early-stopped GD in this
problem setting \citep{ali2019continuous}, a natural idea would then be to use
GCV to estimate the out-of-sample prediction risk of early-stopped GD
iterates. To our knowledge, the performance of GCV in this setting has not yet
been studied.       

In this work, we derive precise theory for both GCV and LOOCV applied to 
the GD iterates from high-dimensional least squares. Our first and somewhat 
surprising result establishes that GCV is generically inconsistent for the
out-of-sample prediction risk of early-stopped GD, even in the most idealized
setting of a well-specified linear model with isotropic features. This
inconsistency becomes particularly pronounced in the overparameterized regime,
where the number of features is greater than the number of observations. In such
a case, the gap between GCV and risk can be substantial, especially as the GD
iteration progresses. This is, of course, problematic for model tuning, as these
are precisely the scenarios in which the optimal stopping time for GD can occur
at a large iteration that allows for (near) interpolation (for the analogous 
theory for ridge regression, see \citet{kobak_lomond_sanchez_2020,
  wu2020optimal, richards2021asymptotics}).  
 
Our second result concerns LOOCV and establishes that it is consistent for
the out-of-sample prediction risk, in a uniform sense over the GD path. For this, we
make only weak assumptions on the feature distribution and do not assume a
well-specified model (i.e., allowing the true regression function to be
nonlinear). One interpretation is that this suggests that the failure of GCV lies
in its ability to approximate LOOCV, and not with LOOCV itself. 
\Cref{fig:gcv-inconsistency-with-loocv-n2500-main} showcases an empirical
illustration of our main results, which we summarize below.

\subsection{Summary of main results}

\begin{itemize}[leftmargin=5mm]
    \item[1.]
    \textbf{GCV inconsistency.}
    Under a proportional asymptotics model where the number of features $p$ and   
    observations $n$ scale proportionally, and assuming a well-specified linear
    model and isotropic features, we show that GCV is inconsistent for
    estimating the prediction risk throughout basically the entire GD
    path (\Cref{thm:gcv-inconsistency}). We prove this result by separately
    deriving the asymptotic limits for the GCV estimator and the true risk of
    early-stopped GD, and then showing that they do not match.      
    
    \item[2.]
    \textbf{LOOCV consistency.}
    Under a proportional asymptotics model again, we show that LOOCV is
    consistent for estimating the prediction risk of early-stopped GD, in a
    uniform sense over the GD iterations
    (\Cref{thm:uniform-consistency-squared}). Our analysis only requires the 
    distributions of the features and noise to satisfy a $T_2$-inequality, which
    is quite weak. In particular, we do not assume any specific model for the regression function. As a consequence of uniformity, we establish that
    the risk of the LOOCV-tuned iterate almost surely matches that of the 
    oracle-tuned iterate. Furthermore, we also propose an implementation of the
    LOOCV with lower computational complexity compared to the naive
    implementation (Proposition \ref{prop:efficient-loo-gd}).  

    \item[3.]
    \textbf{Functional consistency.}
    Beyond prediction risk, we propose a natural extension of LOOCV to estimate
    general functionals of the prediction error distribution for early-stopped GD, 
    which is a plug-in approach based on the empirical
    distribution of LOOCV errors (\Cref{thm:uniform-consistency-general}). As an
    application, we use this to consistently estimate the quantiles of the
    prediction error distribution for GD iterates, allowing the
    construction of prediction intervals with asymptotically correct nominal
    coverage conditional on the training data (\Cref{thm:coverage}).    
\end{itemize}
  
\subsection{Related work}
\label{sec:related_work}

GD and its variants are central tools for training modern machine learning
models. These methods, especially stochastic gradient methods, can be highly
scalable. But, somewhat surprisingly, overparameterized models trained with 
GD and variants also often generalize well, even in the absence of explicit
regularizers and with noisy labels
\citep{zhang_bengio_hardt_recht_vinyals_2016}. 
This behavior is often attributed to the fact that the GD iterates are subject
to a kind of implicit regularization
\citep{wilson2017marginal, gunasekar_lee_soudry_srebro_2018_p1,  
  gunasekar_lee_soudry_srebro_2018_p2}.   
Implicit regularization has a rich history in machine learning and 
has appeared in some of the first insights into the advantages of early stopping
in neural network training \citep{morgan_bourlard_1989}. A parallel idea in
numerical analysis is known as the Landweber iteration \citep{landweber_1951, 
  strand1974theory}. There  
is a rich literature on early stopping in the context of boosting
\citep{buhlmann_yu_2003, rosset_zhu_hastie_2004, zhang_yu_2005,
  yao_rosasco_caponnetto_2007, bauer_pereveraz_rosasco_2007,  
  raskutti_wainwright_yu_2014, wei_yang_wainwright_2017}. Furthermore, several
precise correspondences between GD and ridge penalized estimators have been
established by \citet{suggula2018connecting, neu2018iterate, ali2019continuous, 
  ali2020implicit}, among others.      

CV is a standard approach in statistics for parameter tuning and model
selection. For classic work on CV, see, e.g., \citet{allen_1974, stone_1974,  
  stone_1977, geisser_1975}. For practical surveys, see
\citet{arlot_celisse_2010, zhang2015cross}. More recently, there has been
renewed interest in developing a modern theory for CV, with contributions from  
\citet{kale_kumar_vassilvitskii_2011,
  kumar_lokshtanov_vassilviskii_vattani_2013, celisse_guedj_2016, 
  austern_zhou_2020, bayle_bayle_janson_mackey_2020,  
  lei2020cross, rad_zhou_maleki_2020}, among others. As LOOCV is, in general,
computationally expensive, there has also been recent work in designing and
analyzing approximate leave-one-out methods to address the computational
burden; see, e.g., \citet{wang2018approximate, stephenson_broderick_2020,  
  wilson_kasy_mackey_2020, rad_maleki_2020, auddy2023approximate}.  

GCV is an approximation to LOOCV and is closely connected to what is called
the ``shortcut'' leave-one-out formula for linear smoothers. The classic work on
GCV includes \citet{craven_wahba_1979, golub1979generalized, li_1985, li_1986,   
  li_1987}. Recently, GCV has garnered significant interest, as it has been found
to be consistent for out-of-sample prediction risk in various high-dimensional
settings; see, e.g., \citet{hastie2022surprises, adlam2020understanding, 
  patil_wei_rinaldo_tibshirani_2021, patil2022estimating,
  wei_hu_steinhardt_2022, du2023subsample, han2023distribution,
  patil2023asymptotically}. While originally defined for linear smoothers, the
idea of using similar degrees-of-freedom adjustments can be extended beyond this
original scope to nonlinear predictors; see, e.g., \citet{bayati2011dynamics,
  bayati2013estimating, miolane_montanari_2021, 
  bellec2022derivatives, bellec2023out}.  

Most of the aforementioned papers on CV have focused on estimators that are
defined as solutions to empirical risk minimization problems. There has been
little work that studies CV for iterates of optimization algorithms like GD,
which are commonly used to find solutions (train models) in practice.   
Very recently, \citet{luo2023iterative} consider approximating LOOCV for
iterative algorithms. They propose an algorithm that is more efficient than
the naive LOOCV when $p \ll n$. They also show that their method approximates LOOCV 
well.  In our work, we instead focus on analyzing LOOCV itself, along with GCV,
for least squares problems, which we view as complementary to their
work. Moreover, our analysis is in the proportional asymptotic regime, where $p
\asymp n$.    

\section{Preliminaries}
\label{sec:preliminaries}

In this section, we define the main object of study: early-stopped GD applied 
to the least squares loss. We then precisely define the risk metric of interest
and describe the risk estimators based on LOOCV and GCV.        

\subsection{Early-stopped gradient descent}
\label{sec:gradient_descent_setup}

Consider a standard regression setting, where we observe independent and 
identically distributed samples $\{(\bx_i, y_i)\} \in \RR^{p + 1} \times \RR$ 
for $i \in [n]$. Here, each $\bx_i \in \RR^{p + 1}$ denotes a feature vector
and $y_i \in \RR$ its response value. The last entry of each $\bx_i$ is set to
$1$ to accommodate an intercept term in the regression model. Let  
$\bX\in\RR^{n\times (p + 1)}$ denote the feature matrix whose $i$-th row
contains $\bx_i^\top$, and $\by\in\RR^n$ the response vector whose $i$-th entry
contains $y_i$.  

We focus on the ordinary least squares problem:
\begin{equation}
	\label{eq:ols-obj}
	\minimize_{\bbeta \in \RR^{p + 1}} \; \frac{1}{2n} \|\by-\bX \bbeta\|_2^2,
\end{equation}
and we study the sequence of estimates defined by applying gradient descent (GD) 
to the squared loss in \eqref{eq:ols-obj}. Specifically, given step sizes
$\bdelta = (\delta_0, \dots, \delta_{K-1}) \in \RR^K$, and initializing GD at
the origin, \smash{$\hat\bbeta_0 = 0$}, the GD iterates are defined recursively
as follows:    
\begin{equation}
  \label{eq:gd-iterate}
	\hat\bbeta_{k} = \hat\bbeta_{k - 1} + \frac{\delta_{k - 1}}{n}\bX^\top (\by -
  \bX \hat\bbeta_{k - 1}), \quad k \in [K]. 
\end{equation}
Let $(\bx_0, y_0) \in \RR^{p+1} \times \RR$ denote a test point drawn
independently from the same distribution as the training data. We are interested
in estimating the out-of-sample prediction risk along the GD path. More
precisely, we are interested in estimating the squared prediction error
\smash{$R(\hat{\bbeta}_k)$} achieved by the GD iterate at each step $k \in [K]$,
defined as:  
\begin{equation}
  \label{eq:pred-risk}
  R(\hat{\bbeta}_k) 
  = \E_{\bx_0, y_0}\big[ (y_0 - \bx_0^\top \hat{\bbeta}_k)^2 \mid \bX, \by
  \big].   
\end{equation}
Note that our notion of risk here is conditional on the training features and
responses, $X,y$. 

\subsection{GCV and LOOCV}
\label{sec:gcv_loocv}

Next, we present an overview of the LOOCV and GCV estimators associated with GD
iterates. First, we describe the estimators that correspond to the squared
prediction risk. The exact LOOCV estimator for the squared prediction risk of 
\smash{$\hat{\bbeta}_k$} is defined as: 
\begin{equation}
  \label{eq:loocv-risk}
  \hR^\loo(\hat{\bbeta}_k) 
  = \frac{1}{n} \sum_{i=1}^n (y_i - \bx_i^\top \hat{\bbeta}_{k,-i})^2,
\end{equation}
where \smash{$\hat{\bbeta}_{k,-i}$} denotes the GD iterate after $k$ iterations 
trained on the data \smash{$\bX_{-i}, \by_{-i}$}, which excludes the $i$-th
sample from the full data $\bX, \by$. To be explicit, \smash{$\bX_{-i}$} is the 
result of removing the $i$-th row of $X$, and \smash{$\by_{-i}$} is the result
of removing the $i$-th coordinate of $y$.  

Towards defining GCV, suppose that we have a predictor \smash{$\hat{f} \colon  
\mathbb{R}^{p + 1} \to \mathbb{R}$} which is a linear smoother, i.e.,
\smash{$\hat{f}(\mathbf{x}) = \mathbf{s}_{\mathbf{x}}^{\top}\mathbf{y}$} for 
some vector \smash{$\mathbf{s}_{\mathbf{x}} \in \RR^{n}$} which depends only on
the feature matrix $\mathbf{X}$ and the test point $\mathbf{x}$. The smoothing
matrix associated with the predictor \smash{$\hat{f}$} is denoted $\mathbf{S}
\in \mathbb{R}^{n \times n}$ and defined to have rows
\smash{$\mathbf{s}_{\mathbf{x}_1}^{\top}, \dots,
\mathbf{s}_{\mathbf{x}_n}^{\top}$}. The GCV estimator of the prediction risk of
\smash{$\hf$} is defined as: 
\[
  \hR^\gcv(\hat{f})
  = \frac{\| \mathbf{y} - \mathbf{S} \mathbf{y} \|_2^2 / n}{(1 -
    \tr[\mathbf{S}] / n)^2}. 
\]
The numerator here is the training error, which is of course typically biased
downward, meaning that it typically underestimates the prediction error. The 
denominator corrects for this downward bias, often referred to as the
``optimism'' of the training error, with $1 - \tr[\mathbf{S}] / n$ acting as a    
degrees-of-freedom correction, which is smaller the more complex the model 
(the larger the trace of $\mathbf{S}$).           

A short calculation shows that each GD iterate can be represented as a linear 
smoother, i.e., the in-sample predictions can be written as \smash{$\bX
  \hat\bbeta_k = \bH_k \by$}, where 
\[
  \bH_k = \sum_{j = 0}^{k - 1} \frac{\delta_{j}}{n} \bX \prod_{r = 1}^{k - j -
    1} \big( \id_{p + 1} - \delta_{k - r} \hat\bSigma \big) \bX^{\top},  
\]
and we denote by \smash{$\hat\bSigma = \bX^{\top} \bX / n$} the sample
covariance matrix. This motivates us to estimate its prediction risk using GCV:     
\begin{equation}
\label{eq:gcv-risk}
  \hR^{\gcv}(\hat\bbeta_k)
  = \frac{1}{n} \sum_{i=1}^{n} \frac{(y_i - \bx_i^\top \hat \bbeta_k)^{2}}{(1
    - \tr[\bH_{k}]/n)^2}. 
\end{equation}
Perhaps surprisingly, as we will see shortly in \Cref{sec:naive-inconsistency}, 
GCV does not consistently estimate the prediction risk for GD iterates, even if
we assume a well-specified linear model. On the other hand, we will show in 
\Cref{sec:new-consistency} that LOOCV is uniformly consistent along the GD path.
We also later propose a modified ``shortcut'' in \Cref{sec:computational} that
(1) exactly tracks the LOOCV estimates, and (2) is computationally more
efficient than the naive implementation of LOOCV.  

\section{GCV inconsistency}
\label{sec:naive-inconsistency}

In this section, we prove that GCV is generically inconsistent for estimating
the squared prediction risk, even under a well-specified linear model with
isotropic Gaussian features. For simplicity, in this section only, we consider
fixed step sizes $\delta_k = \delta$ and omit the intercept term. We impose the
following assumptions on the feature and response distributions.   

\begin{assumption}[Feature distribution]
    \label{asm:feature_dist}
    Each feature vector $\bx_i \in \RR^p$, for $i \in [n]$, contains i.i.d.\ 
    Gaussian entries with mean $0$ and variance $1$. 
\end{assumption}

\begin{assumption}[Response distribution]
    \label{asm:response_dist}
    Each response variable $y_i$, for $i \in [n]$, follows a well-specified
    linear model: $y_i = \bx_i^\top \bbeta_0 + \eps_i$. Here, $\bbeta_0 \in
    \RR^{p}$ is an unknown signal vector satisfying \smash{$\lim_{p \to \infty}
      \| \bbeta_0 \|_2^2 = r^2 < \infty$}, and $\eps_i$ is a noise variable,
    independent of $\bx_i$, drawn from a Gaussian distribution with mean     
    $0$ and variance $\sigma^2 < \infty$. 
\end{assumption}

The zero-mean condition for each $y_i$ is used only for
simplicity. (Accordingly, we do not include an additional intercept term in the
model, implying that $\bx_i \in \RR^p$.) Although one could establish the
inconsistency of GCV under more relaxed assumptions, we choose to work under 
\Cref{asm:feature_dist,asm:response_dist} to highlight that GCV fails even under
favorable conditions.   

We analyze the behavior of the estimator in the proportional asymptotics regime,
where both the number of samples $n$ and the number of features $p$ tend to
infinity, and their ratio $p/n$ converges to a constant \smash{$\zeta_{\ast}
  \in (0, \infty)$}. This regime has received considerable attention recently in
  high-dimensional statistics and machine learning theory.

The dynamics of GD are determined by both the step size $\delta$ and the iterate
number $k$. We study a regime in which $\delta \to 0$ and $k \to \infty$
as $n, p \to \infty$, which effectively reduces the GD iterates to a
continuous-time gradient flow, as studied in other work
\citep{ali2019continuous, celentano2021high, berthier2023learning}. Our main
negative result, on GCV, is given next. 

\begin{restatable}
    [Inconsistency of GCV]
    {theorem}
    {ThmGCVInconsistency}
    \label{thm:gcv-inconsistency}
    Suppose that $(x_i,y_i)$, $i \in [n]$ are i.i.d., and satisfy both
    \Cref{asm:feature_dist,asm:response_dist}, where either $r^2 > 0$ 
    or $\sigma^2 > 0$. As $n,p \to \infty$, assume \smash{$p / n \to
      \zeta_{\ast}$}, and $k \to \infty$, $\delta \to 0$ such that $k \delta \to
    T$, where \smash{$T, \zeta_{\ast} > 0$} are constants. Then, for every fixed 
    \smash{$\zeta_{\ast} > 0$}, it holds that for almost all $T > 0$ (i.e., all 
    $T > 0$ except for a set of Lebesgue measure zero),
    \begin{equation}
        \label{eq:gcv_inconsistency}
        \Big| \hR^{\mathrm{gcv}}(\hat\bbeta_k) - R(\hat\bbeta_k) \Big| \not\pto
        0, 
    \end{equation}
    where we recall that \smash{$\hat{R}^{\gcv}(\hat\bbeta_k)$} and
    \smash{$R(\hat\bbeta_k)$} are as defined in \eqref{eq:gcv-risk} and
    \eqref{eq:pred-risk}, respectively.   
\end{restatable}

In other words, the theorem says that GCV does not consistently track the true
prediction risk at basically any point along the GD path (in the sense that
GCV can only possibly be consistent at a Lebesgue measure zero set of times 
$T$). It is worth noting that the inconsistency here can be severe especially
in the overparameterized regime, when \smash{$\zeta_{\ast} > 1$}. In particular,
in this regime, it is easy to show that if $k \delta \to \infty$ (rather than $k
\delta \to T$ for a finite limit $T$), then \smash{$\lim_{k \to \infty}
  \hR^\gcv(\hat\bbeta_k) \to \infty$}, while \smash{$R(\hat \bbeta_K) \to
  r^2 + \sigma^2$}, under the assumptions of \Cref{thm:gcv-inconsistency}. This
is evident in \Cref{fig:gcv-inconsistency-with-loocv-n2500-main}.

\section{LOOCV consistency}
\label{sec:new-consistency}

Despite the inconsistency of GCV, LOOCV remains consistent for GD. This section  
establishes a uniform consistency result for LOOCV along the GD path. 

\subsection{Squared risk}

We begin by focusing on squared prediction risk. The technical crux of our
analysis revolves around establishing certain concentration properties of the
LOOCV estimator \smash{$\hR^\loo(\hat{\bbeta}_k)$}, and to do so, we leverage
Talagrand's $T_2$-inequality \citep{gozlan2009characterization}. Specifically, 
under the assumption that both the entries of the feature and noise
distributions satisfy the $T_2$-inequality, we show that
\smash{$\hR^\loo(\hat{\bbeta}_k)$} behaves approximately as a Lipschitz function
of these random variables. Together, these results enable us to leverage
powerful dimension-free concentration inequalities.

The inspiration for using $T_2$-inequality comes from the recent work of
\citet{avelin2022concentration}. They assume that the data distribution 
satisfies the logarithmic Sobolev inequality (LSI), which is a strictly stronger
condition than what we assume here. Furthermore, they only consider fixed $p$
and do not consider iterative algorithms. The extensions we pursue present
considerable technical challenges and require us to delicately upper bound the
norms of various gradients involved. Below we give a formal definition of what
it means for a distribution to satisfy the $T_2$-inequality.

\begin{definition}
    [$T_2$-inequality]
    \label{def:T2}
    We say a distribution $\mu$ satisfies the $T_2$-inequality if there exists a
    constant $\sigma(\mu) \geq 0$, such that for every distribution $\nu$, 
    \begin{equation}
        \label{eq:T2-inequality}
        W_2(\mu, \nu) \leq \sqrt{2 \sigma^2(\mu) D_{\KL}\big(\nu \big\|  
      \mu\big)}, 
    \end{equation}
    where $W_2(\cdot, \cdot)$ is the 2-Wasserstein distance, and
    \smash{$D_{\KL}(\cdot \| \cdot)$} the Kullback-Leibler divergence.  
\end{definition}

The $T_2$-inequality is, in some sense, a necessary and sufficient condition for
dimension-free concentration. We refer interested readers to Theorem 4.31 in 
\citet{van2014probability} for more details (see also \Cref{sec:T2} for further
facts related to the $T_2$-inequality).  

One prominent example of distributions that satisfy the $T_2$-inequality are
distributions that satisfy the log Sobolev inequality (LSI);
\Cref{sec:definitions} gives more details. We note that all distributions
that are strongly log-concave satisfy the LSI, as do many non-log-concave
distributions, such as Gaussian convolutions of distributions with bounded 
support \citep{chen2021dimension}. %
Next, we formally state our assumptions for this section, starting with the
feature distribution.

\begin{assumption}
    [Feature distribution]
    \label{asm:feat_dist_loocv} \, 
    \begin{enumerate}[leftmargin=5mm]
    \item Each feature vector $\bx_i \in \RR^{p + 1}$, for $i \in [n]$,
      decomposes as \smash{$\bx_i^{\top} = ((\bSigma^{1/2} \bz_i)^{\top},
        1)$}, where $\bz_i \in \RR^p$ has i.i.d.\ entries $z_{ij}$ drawn from
      $\mu_z$.  
    \item The distribution $\mu_z$ has mean $0$, variance $1$, and satisfies 
    the $T_2$-inequality with constant $\sigma_z$. 
    \item There covariance matrix satisfies $\|\bSigma\|_{\op} \leq
      \sigma_{\Sigma}$ for a constant $\sigma_{\Sigma}$.   
    \end{enumerate}
\end{assumption}

To be clear, in the above \smash{$\sigma_z, \sigma_{\Sigma}$} are constants that
are not allowed to change with $n,p$. It is worth emphasizing that we do not
require the smallest eigenvalue of $\bSigma$ in \Cref{asm:feat_dist_loocv} to
be bounded away from $0$. This is possible because the iterates along the GD
path are implicitly regularized. This is similar to not requiring a lower bound
on the smallest eigenvalue for ridge regression when $\lambda > 0$ (as opposed
to ridgeless regression, where we do need such an assumption); see
\citet{dobriban_wager_2018, patil_wei_rinaldo_tibshirani_2021}. We also impose
the following assumptions on the response distribution. 

\begin{assumption}[Response distribution]
\label{asm:response} \,
    \begin{enumerate}[leftmargin=5mm]
    \item Each $y_i = f (\bx_i) + \ep_i$, for $i \in [n]$,\footnotemark
      \hspace{0.5pt} where $\ep_i$ is independent of $\bx_i$ and drawn from
      $\mu_{\ep}$.   
      \item The distribution $\mu_{\ep}$ has mean $0$ and satisfies the
      $T_2$-inequality with constant $\sigma_{\ep}$.  
        \item The regression function $f$ is $L_{f}$-Lipschitz continuous, where
          without loss of generality, $L_{f} \le 1$.  
        \item Finally, $\E[y_i^8] \le m_8$, $\E[y_i^4] \le m_4$, and $\E[y_i^2]
          \le m_2$.  
    \end{enumerate}
\end{assumption}

\footnotetext{Our result holds under a more general setting where $y_i =
  f(\bx_i, \eps_i)$, with $f$ being $L_{f}$-Lipschitz continuous. In the
  appendix, we provide the proof under this more general condition.} 

In the above \smash{$\sigma_{\eps}, m_2, m_4, m_8$} are constants that
are not allowed to change with $n,p$. We note that the assumptions we impose 
in this section are strictly weaker than those in
\Cref{sec:naive-inconsistency}. In particular, it is notable that we do not
require $\E[y_i \,|\, \bx_i = x]$ to be linear in $x$. We are ready to give 
our first main positive result, on LOOCV for squared risk.    

\begin{restatable}    
    [Squared risk consistency of LOOCV]
    {theorem}
    {ThmUniformConsistencySquared}
    \label{thm:uniform-consistency-squared}
    Suppose that $(x_i,y_i)$, $i \in [n]$ are i.i.d., and satisfy both
    \Cref{asm:feat_dist_loocv,asm:response}.
    In addition, assume that there are constants \smash{$\Delta, B_0, \zeta_L,  
    \zeta_U$} (independent of $n,p$) such that:
    (1) \smash{$\sum_{k  = 1}^K \delta_{k - 1} \leq \Delta$}, 
    (2) \smash{$\|\hat \bbeta_0\|_2 \leq B_0$}, and
    (3) $0 < \zeta_L \leq p / n \leq \zeta_U < \infty$. 
    Furthermore, let \smash{$K = o(n \cdot (\log n)^{-3/2})$}. 
    Then, as $n, p \to \infty$,
    \begin{equation}
        \label{eq:loocv-consistency-squared}
        \max_{k \in [K]} \, \Big| \hat{R}^{\loo}(\hat\bbeta_k) - R(\hat\bbeta_k) 
      \Big| \asto 0,  
    \end{equation}
    where we recall that \smash{$\hat{R}^{\loo}(\hat\bbeta_k)$} and
    \smash{$R(\hat\bbeta_k)$}  are as defined in \eqref{eq:loocv-risk} and
    \eqref{eq:pred-risk}, respectively.   
\end{restatable}

The convergence guarantee in \Cref{thm:uniform-consistency-squared} is strong in
the sense that it is uniform across the entire GD path, and convergence occurs
conditional on the training data. Uniformity in particular allows us to argue
that tuning based on LOOCV guarantees asymptotically optimal risk. We cover this
next, where we also generalize our study from squared error to general error
functionals.

\subsection{General risk functionals}
\label{sec:extension-general-risk-functionals}

We now extend our theory from the last subsection to cover general risk
functionals, subject to only mild regularity conditions. Let $\psi \colon \RR^2
\to \RR$ be an error function, which takes as input the predictand (first 
argument) and prediction (second argument). We define a corresponding risk
functional as:      
\begin{equation}
    \label{eq:pred-functional}
    \Psi(\hat\bbeta_k)
    = \EE_{\bx_0, y_0}\big[ \psi(y_0, \bx_0^\top \hat{\bbeta}_k) \mid \bX, \by
    \big]. 
\end{equation}
One can naturally define an estimator for \smash{$\Psi(\hat\bbeta_k)$} based on
LOOCV using the ``plug-in'' principle:
\begin{equation}
    \label{eq:Tloocv}
    \hPsi^\loo(\hat{\bbeta}_k) 
    = \frac{1}{n} \sum_{i=1}^n \psi(y_i, \bx_i^\top \hat{\bbeta}_{k,-i}).
\end{equation}
Our second main positive result shows that this LOOCV plug-in estimator is  
uniformly consistent along the GD path.  

\begin{restatable}
    [Functional consistency of LOOCV]
    {theorem}
    {ThmUniformConsistencyGeneral}
    \label{thm:uniform-consistency-general}
    Under the conditions of \Cref{thm:uniform-consistency-squared}, suppose
    that $\psi: \R^2 \to \R$ is differentiable and satisfies \smash{$\|\nabla
      \psi(u)\|_2 \leq C_{\psi}\|u\|_2 + \bar C_{\psi}$} for all $u \in \R^2$
    and for constants \smash{$C_{\psi}, \bar C_{\psi} \geq 0$}. Then, as $n, p 
    \to \infty$,   
    \begin{equation}
        \label{eq:loocv-consistency-functional}
        \max_{k \in [K]} \, \big| \hat{\Psi}^{\loo}(\hat\bbeta_k) -
      \Psi(\hat\bbeta_k) \big|   \asto 0. 
    \end{equation}
  where we recall that \smash{$\hat{R}^{\loo}(\hat\bbeta_k)$} and
    \smash{$R(\hat\bbeta_k)$}  are as defined in \eqref{eq:Tloocv} and
    \eqref{eq:pred-functional}, respectively. 
    
    As consequence of \eqref{eq:loocv-consistency-functional}, LOOCV can be used
    to tune early stopping. Specifically, if we define \smash{$k_{\ast} =
      \argmin_{k \in  [K]} \hat \Psi^{\loo}(\hat \bbeta_k)$}, then as $n, p
    \to \infty$, 
    \begin{equation}
        \label{eq:loocv-consistency-functional-tuned}
        \big| \Psi(\hat \bbeta_{k_{\ast}})
        - \min_{k \in [K]} \Psi(\hat \bbeta_k)
        \big| \asto 0. 
    \end{equation}
\end{restatable}

Thanks to \Cref{thm:uniform-consistency-general}, we can consistently estimate
the quantiles of the prediction error distribution using the empirical quantiles
of the distribution that puts $1/n$ mass at each LOOCV residual.

\begin{restatable}
    [Coverage guarantee]
    {theorem}
    {ThmCoverage}
    \label{thm:coverage}
    Under the conditions of \Cref{thm:uniform-consistency-general}, assume
    further that the distribution of the noise $\ep_i$ is continuous with
    density bounded by \smash{$\kappa_{\pdf}$}. Denote by
    \smash{$\hat\alpha_{k}(q)$} the $q$-quantile of \smash{$\{y_i - \bx_i^{\top} 
      \hat\bbeta_{k, -i}: i \in [n]\}$}. Then, for any quantile levels $0 \leq
    q_1 \leq q_2 \leq 1$, letting \smash{$\cI_k = [\hat\alpha_{k}(q_1), 
      \hat\alpha_{k}(q_2)]$}, we have as $n,p \to \infty$,    
    \begin{equation}
        \label{eq:loocv-coverage}
      \max_{k \in [K]} \,
        \P_{(\bx_0, y_0)}\big( y_{\new} - \bx_{\new}^{\top} \hat\bbeta_k \in
        \cI_k \mid \bX, \by \big) \asto q_2 - q_1. 
    \end{equation}
\end{restatable}

Note that \Cref{thm:coverage} provides \emph{conditional} rather than
\emph{marginal} coverage guarantees for the specific data $X,y$ that we observe.
\Cref{fig:pred-intervals} provides an example. Finally, we remark that
the empirical distribution of the LOOCV errors can be shown to weakly converge 
to the true error distribution, almost surely. This is illustrated in 
\Cref{fig:test-loo-dist-comparison}, with
\Cref{fig:test-loo-dist-comparison-ridges} providing an additional
visualization. 

\begin{figure*}[p]
    \centering
    \includegraphics[width=0.495\textwidth]{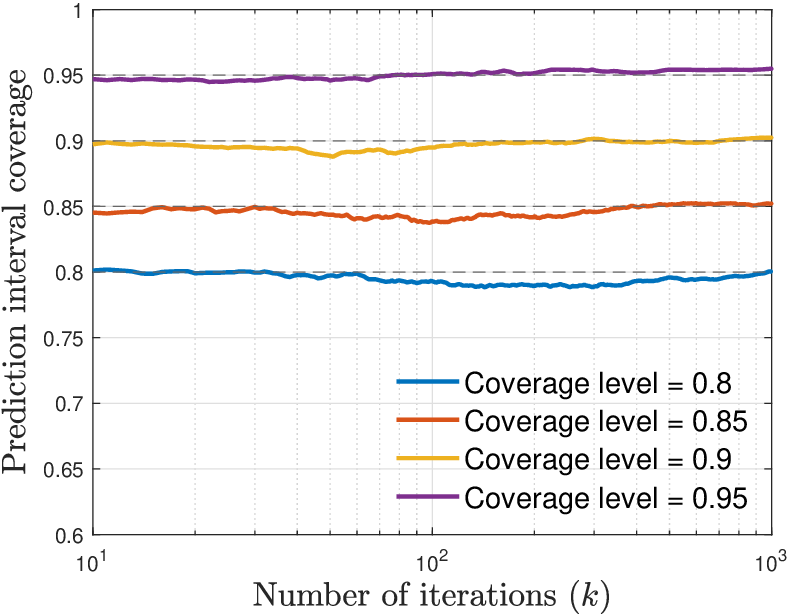}
    \includegraphics[width=0.485\textwidth]{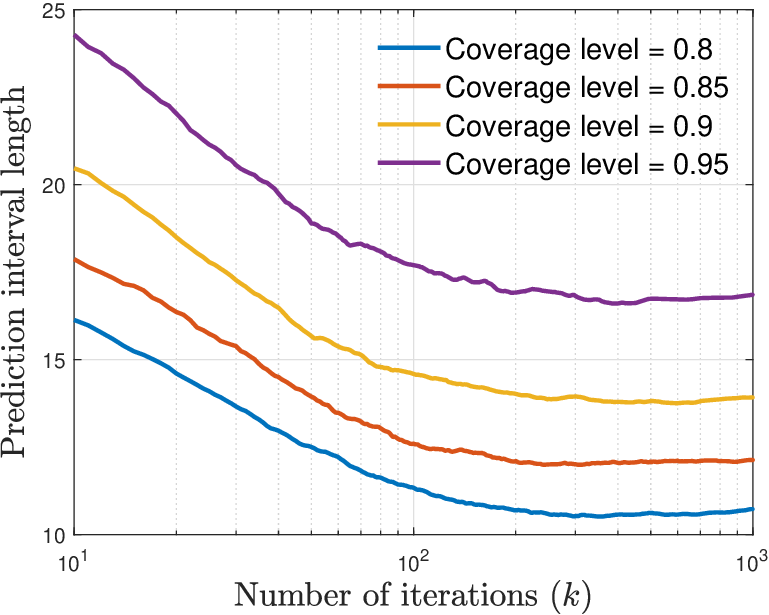}
    \caption{\textbf{LOOCV provides (asymptotically) valid prediction intervals,
        for various nominal coverage levels.} We investigate the empirical
      coverage and length of LOOCV prediction intervals along the GD path, at
      varying coverage levels. We consider an overparameterized regime with
      $n=2500$ and $p=5000$. The features are drawn from a Gaussian distribution
      with a covariance structure: $\bSigma_{ij} = \rho^{|i - j |}$ for all
      $i,j$ and $\rho=0.25$. The response is generated from a nonlinear model
      with heavy-tailed noise: $t$-distribution with 5 degrees of freedom. The 
      linear component of $\E[y_i \,|\, \bx_i = x]$ is aligned with the top 
      eigenvector of $\bSigma$. GD is run with a constant step size of
      $0.01$. (See \Cref{sec:additional-numerical-illustrations} for further
      details on the experimental setup.) We can see that the prediction
      intervals generally have excellent finite-sample coverage along the entire
      path (\emph{left}), and the smallest prediction length is typically
      obtained at a large iteration of GD (\emph{right}).} 
    \label{fig:pred-intervals}

\bigskip\medskip
    \includegraphics[width=0.315\textwidth]{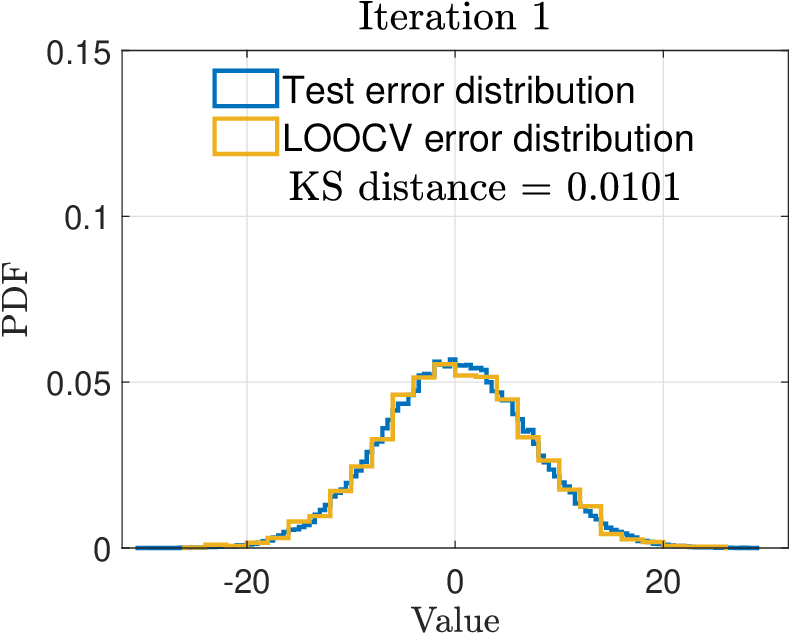}
    ~
    \includegraphics[width=0.315\textwidth]{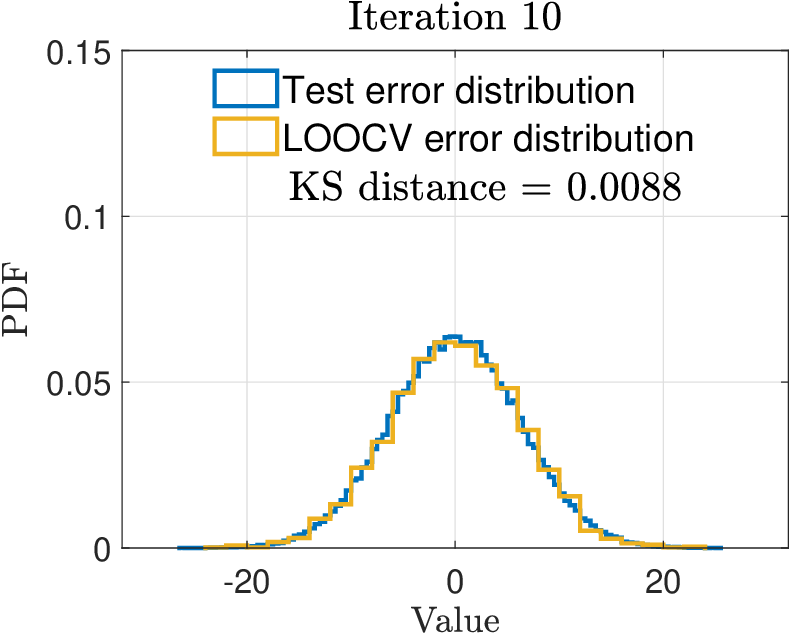}
    ~
    \includegraphics[width=0.315\textwidth]{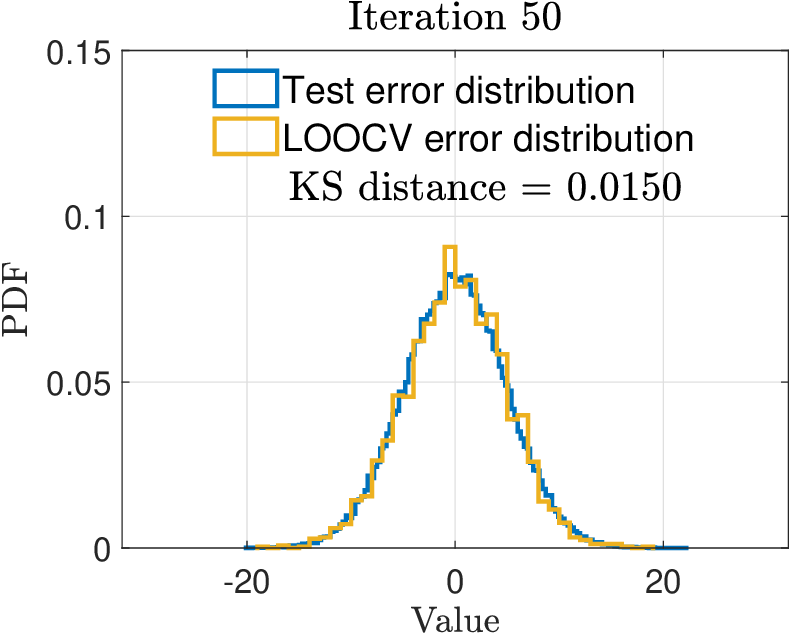}
    \\
    \medskip
    \includegraphics[width=0.315\textwidth]{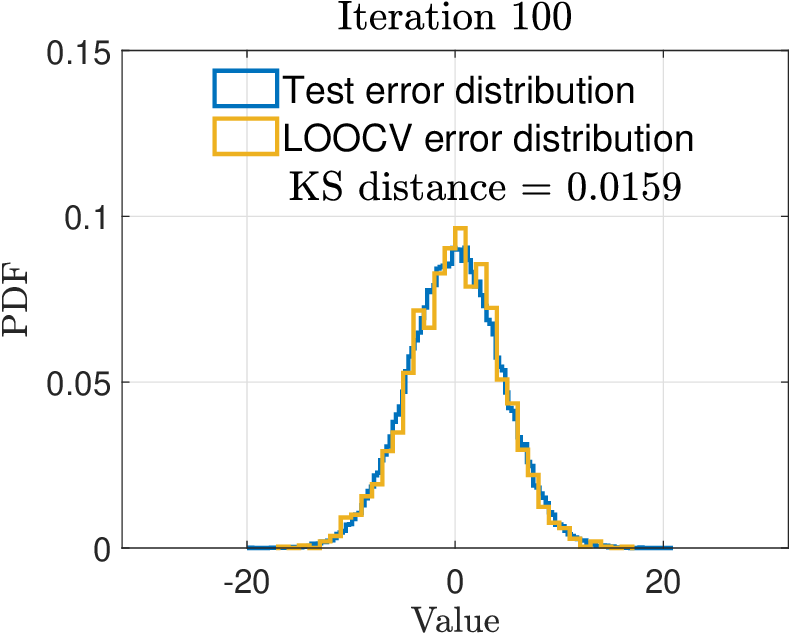}
    ~
    \includegraphics[width=0.315\textwidth]{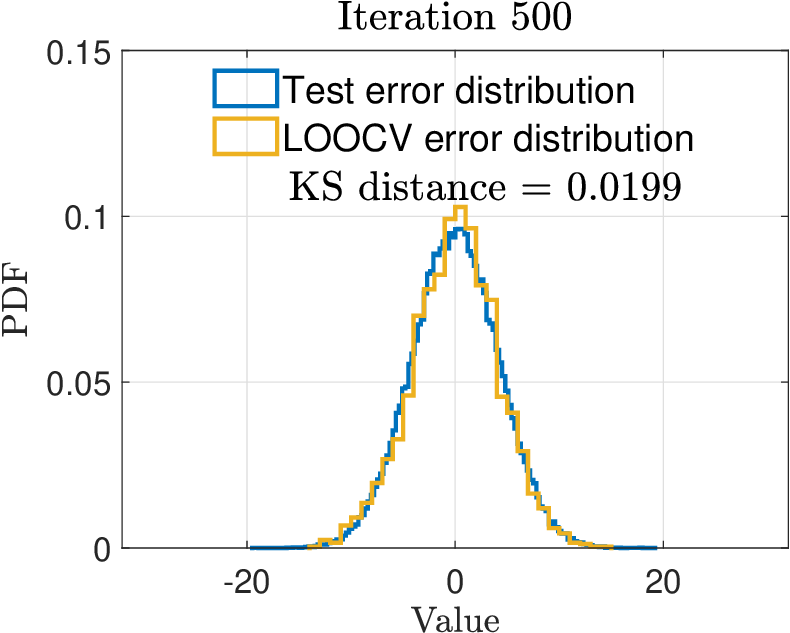}
    ~
    \includegraphics[width=0.315\textwidth]{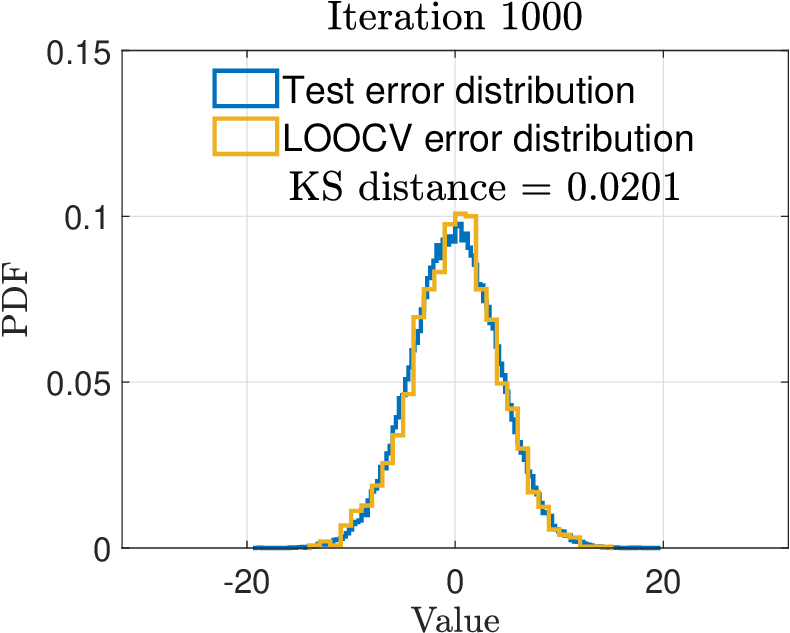}
    \caption{
    \textbf{Empirical distribution of LOOCV errors tracks the true test error
      distribution along the entire GD trajectory.} We consider the same setup
    as in \Cref{fig:pred-intervals} with an overparameterized regime of 
    $n=2500$ and $p=5000$. (See \Cref{sec:additional-numerical-illustrations}
    for further details on the experimental setup.) The blue curve in each
    panel represents a histogram of true prediction error errors (computed via
    Monte Carlo), while the yellow curve represents a histogram of the LOOCV 
    errors. Each panel represents a given GD iteration, and we see strong
    agreement in the histograms throughout. Furthermore, due to the structure of
    the simulation setup, the test error distribution begins to exhibit lower
    variance as the iterations proceed.} %
    \label{fig:test-loo-dist-comparison}
\end{figure*}

\section{Discussion}
\label{sec:computational}

In the paper, we establish a significant discrepancy between LOOCV and GCV when 
it comes to estimating the prediction risk of early-stopped GD for least squares 
regression in high dimensions. While LOOCV is consistent in a strong uniform
sense, GCV fails along essentially the entire path. This is especially curious
considering that both LOOCV and GCV are uniformly consistent for the risk of
explicitly regularized estimators such as ridge regression
\citep{patil_wei_rinaldo_tibshirani_2021, patil2022estimating}. Therefore, this
discrepancy also highlights a difference between GD and ridge regression, which
is interesting in light of all of the existing work that establishes similarities
between the two \citep{suggula2018connecting, neu2018iterate,
  ali2019continuous}.  

Recall that GCV is generally tied to the ``shortcut'' formula for the
leave-one-out (LOO) predictions in linear smoothers, where we adjust the
training error for the $i$-th sample by $1 - \tr[S] / n$ (GCV), in place of $1 - 
S_{ii} / n$ (shortcut formula). A key part of the failure of GCV for GD is that
its LOO predictions behave differently than those in ridge regression, as we
discuss in what follows.    

\subsection{LOO predictions in ridge regression versus gradient descent}   
\label{sec:gcv-failure}

For ridge regression, the LOO predictions, and hence LOOCV residuals, can be 
computed directly from the residuals of the full model (the model fit on the
full data $X,y$) using a shortcut formula \citep{golub1979generalized,
  hastie_2020}. This is computationally important since it means we can compute
LOOCV without any refitting. 

An elegant way to verify this shortcut formula involves creating an augmented
system that allows us to identify the LOO prediction, which we briefly describe
here. (We omit the intercept in the model, for simplicity.) For a given data
point $(\bx_i, y_i)$ that is to be left out, we seek to solve the problem:   
\begin{equation}
    \label{eq:ridge-loo}
    \minimize_{\bbeta \in \RR^{p}} \; \| \by_{-i} -  \bX_{-i} \bbeta \|_2^2 + 
    \lambda \| \bbeta \|_2^2.  
\end{equation}
Denoting its solution by \smash{$\hat{\bbeta}_{\lambda, -i}$}, the corresponding
LOO prediction is therefore \smash{$\bx_i^\top \hat{\bbeta}_{\lambda, -i}$}. Let
us now imagine that we ``augment'' the data \smash{$\bX_{-i}, \by_{-i}$} set by 
adding the pair \smash{$(\bx_i, \bx_i^\top \hat{\bbeta}_{\lambda, -i})$} in
place of the $i$-th sample. Denote by \smash{$\widetilde{\by}_{-i} \in \RR^{n}$}
the response vector in the augmented data set, and $\bX$ the feature matrix (it
is unchanged from the original data set). Denote by
\smash{$\widetilde{\bbeta}_{\lambda, -i}$} the ridge estimator fit on the
augmented data set \smash{$\bX, \widetilde{\by}_{-i}$}, which solves:  
\begin{equation}
    \label{eq:ridge-augmented}
    \minimize_{\bbeta \in \RR^{p}} \; \| \widetilde{\by}_{-i} - \bX \bbeta
    \|_2^2 + \lambda \| \beta \|_2^2. 
\end{equation}
Problems \eqref{eq:ridge-loo} and \eqref{eq:ridge-augmented} admit the same 
solution, because in the latter we have only added a single sample
\smash{$\bx_i^\top \hat{\bbeta}_{\lambda, -i}$} and this attains zero loss at
the solution in the former. Thus, we have \smash{$\widetilde{\bbeta}_{\lambda,
    -i} = \hat{\bbeta}_{\lambda, -i}$}, and we can write the predicted value for
the $i$-th observation as follows: 
\[
    \bx_i^\top \hat{\bbeta}_{\lambda, -i}
    =
    \sum_{j \neq i} [\bH_{\lambda}]_{ij} y_{j} + [\bH_{\lambda}]_{ii} 
    (\bx_i^\top \hat{\bbeta}_{\lambda, -i}),
\]
where \smash{$\bH_{\lambda} \in \RR^{n \times n}$} is the ridge smoothing matrix
associated with full feature matrix $\bX$ at regularization level $\lambda$. 
Rearranging, we have: 
\[
    \bx_i^\top \hat{\bbeta}_{\lambda, -i}
    = \frac{\sum_{j \neq i} [\bH_{\lambda}]_{ij} y_j}{1 - [\bH_{\lambda}]_{ii}},
\]
or equivalently, in terms of residuals:
\begin{equation}
    \label{eq:ridge-loo-shortcut}
    y_i - \bx_i^\top \hat{\bbeta}_{\lambda, -i}
    = \frac{y_i - \sum_{j} [\bH_{\lambda}]_{ij} y_j}{1 - [\bH_{\lambda}]_{ii}}
    = \frac{y_i - \bx_i^\top \hat{\bbeta}_\lambda}{1 - [\bH_{\lambda}]_{ii}}.
\end{equation}

Meanwhile, for the GD path, the analogous construction does not reproduce
the LOO predictions. More precisely, let \smash{$\hat{\bbeta}_{k,-i}$} be the GD 
iterate at step $k$, run on the LOO data set \smash{$\bX_{-i}, \by_{-i}$}. As
before, imagine that we augment this data set with the pair \smash{$(\bx_i,
  \bx_i^\top \hat{\bbeta}_{k, -i})$}. Denote again by $\bX$ the feature matrix 
and \smash{$\widetilde{\by}_{-i} \in \RR^{n}$} the response vector in the
augmented data set, and denote by \smash{$\widetilde{\bbeta}_{k, -i}$} the
result of running $k$ iterations of GD on \smash{$\bX,
  \widetilde{\by}_{-i}$}. In general, we will have \smash{$\hat{\bbeta}_{k, -i}
  \neq \widetilde{\bbeta}_{k, -i}$}. 

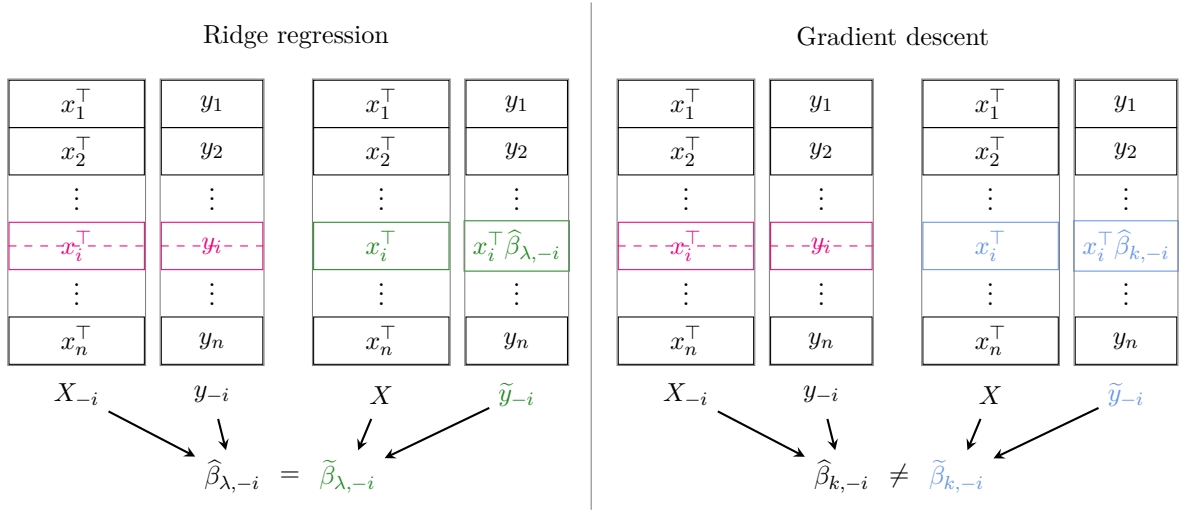
\begin{figure}[t]
    \centering
    \begin{tikzpicture}[every node/.style={minimum size=0.7cm},on grid,scale=0.9,transform shape]
    
    \draw [gray] (7.6,1.5) -- (7.6,-6);
    \node[anchor=east] at (4.75,1) {\scalebox{1}{{Ridge regression}}};
    \node[anchor=west] at (10.5,1) {\scalebox{1}{{Gradient descent}}};
    
    \node[draw,minimum width=2cm, minimum height=0.7cm] (x1) at (0,0) {\scalebox{1}{\(\bx_1^\top\)}};
    \node[draw,minimum width=2cm, minimum height=0.7cm] (x2) at (0,-0.7) {\scalebox{1}{\(\bx_2^\top\)}};
    \node (xdots1) at (0,-1.3) {\(\vdots\)};
    \node[draw,RubineRed,minimum width=2cm, minimum height=0.7cm] (xi) at (0,-2.1) {\scalebox{1}{\(\bx_i^\top\)}};
    \node (xdots2) at (0,-2.7) {\(\vdots\)};
    \node[draw,minimum width=2cm, minimum height=0.7cm] (xn) at (0,-3.5) {\scalebox{1}{\(\bx_n^\top\)}};
    \node[draw,minimum width=1.5cm, minimum height=0.7cm] (y1) at (2,0) {\scalebox{1}{\(y_1\)}};
    \node[draw,minimum width=1.5cm, minimum height=0.7cm] (y2) at (2,-0.7) {\scalebox{1}{\(y_2\)}};
    \node (ydots1) at (2,-1.3) {\(\vdots\)};
    \node[draw,RubineRed,minimum width=1.5cm, minimum height=0.7cm] (yi) at (2,-2.1) {\scalebox{1}{\(y_i\)}};
    \node (ydots2) at (2,-2.7) {\(\vdots\)};
    \node[draw,minimum width=1.5cm, minimum height=0.7cm] (yn) at (2,-3.5) {\scalebox{1}{\(y_n\)}};
    
    \draw[gray] ([xshift=-0.01cm, yshift=0.01cm]x1.north west) rectangle ([xshift=0.01cm, yshift=-0.01cm]xn.south east);
    
    \draw[gray] ([xshift=-0.01cm, yshift=0.01cm]y1.north west) rectangle ([xshift=0.01cm, yshift=-0.01cm]yn.south east);
    
    \node (feat-ridge-orig) at (0, -4.3) {\scalebox{1}{\(\mathbf{X}_{-i}\)}};
    \node (resp-ridge-orig) at (2, -4.3) {\scalebox{1}{\(\mathbf{y}_{-i}\)}};
    
    \node (est1) at (2.3,-5.5) {\scalebox{1}{\(\hat{\bbeta}_{\lambda, -i}\)}};
    
    \draw[arrow] (feat-ridge-orig) -- (est1);
    \draw[arrow] (resp-ridge-orig) -- (est1);
    
    \draw[RubineRed,dashed] (xi.west)--(xi.east);
    \draw[RubineRed,dashed] (yi.west)--(yi.east);
    
    \node[draw,minimum width=2cm, minimum height=0.7cm] (ax1) at (4.5,0) {\scalebox{1}{\(\bx_1^\top\)}};
    \node[draw,minimum width=2cm, minimum height=0.7cm] (ax2) at (4.5,-0.7) {\scalebox{1}{\(\bx_2^\top\)}};
    \node (axdots1) at (4.5,-1.3) {\(\vdots\)};
    \node[draw,ForestGreen,minimum width=2cm, minimum height=0.7cm] (axi) at (4.5,-2.1) {\scalebox{1}{\(\bx_i^\top\)}};
    \node (axdots2) at (4.5,-2.7) {\(\vdots\)};
    \node[draw,minimum width=2cm, minimum height=0.7cm] (axn) at (4.5,-3.5) {\scalebox{1}{\(\bx_n^\top\)}};
    \node[draw,minimum width=1.5cm, minimum height=0.7cm] (ay1) at (6.5,0) {\scalebox{1}{\(y_1\)}};
    \node[draw,minimum width=1.5cm, minimum height=0.7cm] (ay2) at (6.5,-0.7) {\scalebox{1}{\(y_2\)}};
    \node (aydots1) at (6.5,-1.3) {\(\vdots\)};
    \node[draw,ForestGreen,minimum width=1.5cm, minimum height=0.7cm] (ayi) at (6.5,-2.1) {\scalebox{1}{\(\bx_i^\top \hat{\bbeta}_{\lambda, -i}\)}};
    \node (aydots2) at (6.5,-2.7) {\(\vdots\)};
    \node[draw,minimum width=1.5cm, minimum height=0.7cm] (ayn) at (6.5,-3.5) {\scalebox{1}{\(y_n\)}};
    
    \draw[gray] ([xshift=-0.01cm, yshift=0.01cm]ax1.north west) rectangle ([xshift=0.01cm, yshift=-0.01cm]axn.south east);
    
    \draw[gray] ([xshift=-0.01cm, yshift=0.01cm]ay1.north west) rectangle ([xshift=0.01cm, yshift=-0.01cm]ayn.south east);
    
    \node (feat-ridge-aug) at  (4.5, -4.3) {\scalebox{1}{\(\mathbf{X}\)}};
    \node [ForestGreen] (resp-ridge-aug) at (6.5, -4.3) {\scalebox{1}{\(\widetilde{\mathbf{y}}_{-i}\)}};
    
    \node [ForestGreen] (est2) at (4,-5.5) {\scalebox{1}{\(\widetilde{\bbeta}_{\lambda, -i}\)}};
    
    \draw[arrow] (feat-ridge-aug) -- (est2);
    \draw[arrow] (resp-ridge-aug) -- (est2);
    
    \node at (3.15,-5.5) {\scalebox{1}{=}};
     
    \node[draw,minimum width=2cm, minimum height=0.7cm] (kx1) at (9,0) {\scalebox{1}{\(\bx_1^\top\)}};
    \node[draw,minimum width=2cm, minimum height=0.7cm] (kx2) at (9,-0.7) {\scalebox{1}{\(\bx_2^\top\)}};
    \node (kxdots1) at (9,-1.3) {\(\vdots\)};
    \node[draw,RubineRed,minimum width=2cm, minimum height=0.7cm] (kxi) at (9,-2.1) {\scalebox{1}{\(\bx_i^\top\)}};
    \node (kxdots2) at (9,-2.7) {\(\vdots\)};
    \node[draw,minimum width=2cm, minimum height=0.7cm] (kxn) at (9,-3.5) {\scalebox{1}{\(\bx_n^\top\)}};
    \node[draw,minimum width=1.5cm, minimum height=0.7cm] (ky1) at (11,0) {\scalebox{1}{\(y_1\)}};
    \node[draw,minimum width=1.5cm, minimum height=0.7cm] (ky2) at (11,-0.7) {\scalebox{1}{\(y_2\)}};
    \node (kydots1) at (11,-1.3) {\(\vdots\)};
    \node[draw,RubineRed,minimum width=1.5cm, minimum height=0.7cm] (kyi) at (11,-2.1) {\scalebox{1}{\(y_i\)}};
    \node (kydots2) at (11,-2.7) {\(\vdots\)};
    \node[draw,minimum width=1.5cm, minimum height=0.7cm] (kyn) at (11,-3.5) {\scalebox{1}{\(y_n\)}};
    
    \draw[gray] ([xshift=-0.01cm, yshift=0.01cm]kx1.north west) rectangle ([xshift=0.01cm, yshift=-0.01cm]kxn.south east);
    
    \draw[gray] ([xshift=-0.01cm, yshift=0.01cm]ky1.north west) rectangle ([xshift=0.01cm, yshift=-0.01cm]kyn.south east);
    
    \node (feat-gd-orig) at (9, -4.3) {\scalebox{1}{\(\mathbf{X}_{-i}\)}};
    \node (resp-gd-orig) at (11, -4.3) {\scalebox{1}{\(\mathbf{y}_{-i}\)}};
    
    \draw[RubineRed,dashed] (kxi.west)--(kxi.east);
    \draw[RubineRed,dashed] (kyi.west)--(kyi.east);
    
    \node (kest1) at (11.3,-5.5) {\scalebox{1}{\(\hat{\bbeta}_{k, -i}\)}};
    
    \draw[arrow] (feat-gd-orig) -- (kest1);
    \draw[arrow] (resp-gd-orig) -- (kest1);
    
    \node[draw,minimum width=2cm, minimum height=0.7cm] (kax1) at (13.5,0) {\scalebox{1}{\(\bx_1^\top\)}};
    \node[draw,minimum width=2cm, minimum height=0.7cm] (kax2) at (13.5,-0.7) {\scalebox{1}{\(\bx_2^\top\)}};
    \node (kaxdots1) at (13.5,-1.3) {\(\vdots\)};
    \node[draw,CornflowerBlue,minimum width=2cm, minimum height=0.7cm] (kaxi) at (13.5,-2.1) {\scalebox{1}{\(\bx_i^\top\)}};
    \node (kaxdots2) at (13.5,-2.7) {\(\vdots\)};
    \node[draw,minimum width=2cm, minimum height=0.7cm] (kaxn) at (13.5,-3.5) {\scalebox{1}{\(\bx_n^\top\)}};
    \node[draw,minimum width=1.5cm, minimum height=0.7cm] (kay1) at (15.5,0) {\scalebox{1}{\(y_1\)}};
    \node[draw,minimum width=1.5cm, minimum height=0.7cm] (kay2) at (15.5,-0.7) {\scalebox{1}{\(y_2\)}};
    \node (kaydots1) at (15.5,-1.3) {\(\vdots\)};
    \node[draw,CornflowerBlue,minimum width=1.5cm, minimum height=0.7cm] (kayi) at (15.5,-2.1) {\scalebox{1}{\(\bx_i^\top \hat{\bbeta}_{k, -i}\)}};
    \node (kaydots2) at (15.5,-2.7) {\(\vdots\)};
    \node[draw,minimum width=1.5cm, minimum height=0.7cm] (kayn) at (15.5,-3.5) {\scalebox{1}{\(y_n\)}};
    
    \draw[gray] ([xshift=-0.01cm, yshift=0.01cm]kax1.north west) rectangle ([xshift=0.01cm, yshift=-0.01cm]kaxn.south east);
    
    \draw[gray] ([xshift=-0.01cm, yshift=0.01cm]kay1.north west) rectangle ([xshift=0.01cm, yshift=-0.01cm]kayn.south east);
    
    \node (feat-gd-aug) at (13.5, -4.3) {\scalebox{1}{\(\mathbf{X}\)}};
    \node [CornflowerBlue] (resp-gd-aug) at (15.5, -4.3) {\scalebox{1}{\(\widetilde{\mathbf{y}}_{-i}\)}};
    
    \node [CornflowerBlue] (kest2) at (13,-5.5) {\scalebox{1}{\(\widetilde{\bbeta}_{k, -i}\)}};
    
    \draw[arrow] (feat-gd-aug) -- (kest2);
    \draw[arrow] (resp-gd-aug) -- (kest2);
    
    \node at (12.15,-5.5) {\scalebox{1}{\(\neq\)}};
    
    \end{tikzpicture}
    \caption{Illustrations of the differences between the LOO systems for ridge 
      regression (\emph{left}) and GD (\emph{right}).}  
    \label{fig:loo-ridge-vs-gd}
\end{figure}

The underlying reason for this is that, even though the GD iterates
can be written as a solution to a regularized least squares problem, the
regularizer in this problem depends on the data (which is not true in
ridge). For constant step sizes all equal to $\delta$, the GD iterate
\eqref{eq:gd-iterate} can be shown to solve:     
\[
\minimize_{\bbeta \in \R^p} \; \| \by - \bX \bbeta \|_2^2 /2n + \bbeta^\top Q_k
\bbeta,
\]
where \smash{$Q_k = \bX^\top \bX / n ((\bI_p - \delta \bX^\top \bX / n)^k -
  \bI_p)^{-1}$}. The regularization term is not only a function of $\delta$ and
$k$, but also of $\bX$. This complicates the LOO predictions. 

\subsection{Modified augmented system for LOO in GD}
\label{sec:modified-augmentation-system}

Identifying this failure of GD, as summarized in \Cref{fig:loo-ridge-vs-gd},
also helps us modify the augmentation trick so that we can recover the LOO
predictions. Precisely, for $k \in [K]$ and $i,j \in [n]$, let   
\[
	(\tilde y_{k,-i})_j = \begin{cases}
		y_j, & j \neq i \vspace{0.2cm} \\
		\bx_i^{\top} \hat{\bbeta}_{k, -i}, & j = i.
    \end{cases}
\]
Define the vector \smash{$\tilde y_{k,-i} = (\tilde y_{k,-i})_{j \leq n}$}, and
let \smash{$\tilde \bbeta_{k, -i}$} be the iterate obtained by running GD for
$k$ steps where at each step $\ell \leq k$, the augmented response vector
\smash{$\tilde y_{\ell,-i} $} is used in the gradient update. See 
\Cref{fig:modified-augmentation} for an illustration. Next, we show that this
scheme recovers the LOO coefficients along the GD path.  

\begin{restatable}
    [Correctness of the modified augmented system] 
    {proposition}
    {LemmaBetaTildeBeta}
    \label{lemma:beta_tilde=beta}
    For all $k \in [K]$ and $i \in [n]$, it holds that \smash{$\tilde \bbeta_{k,
        -i} = \hat{\bbeta}_{k, -i}$}.
\end{restatable}

In other words, to recreate LOO coefficients from $k$-step GD, we must use
an augmented response vector not only at step $k$ but \emph{at every iteration
  before $k$} as well. With this insight, we can represent the LOO predictions in
a certain linear smoother form.  

\begin{restatable}
    [Smoother representation for the modified augmented system]  
    {proposition}
    {LemmaHatb}
    \label{lemma:hat-b}
    For all $k \in [K]$ and $i \in [n]$, there is a vector
    \smash{$(h_{ij}^{(k)})_{j \leq n}$} and scalar \smash{$b_i^{(k)}$} 
    depending $\bdelta = (\bdelta_0,\dots\bdelta_{k-1})$ and $\bX$  
    such that:   
    \[
		\bx_i^\top \hat{\bbeta}_{k,-i} 
    = \bx_i^\top \tilde{\bbeta}_{k,-i} 
    = \sum_{j = 1}^n h_{ij}^{(k)} y_j + b_i^{(k)}. 
	\]
\end{restatable}

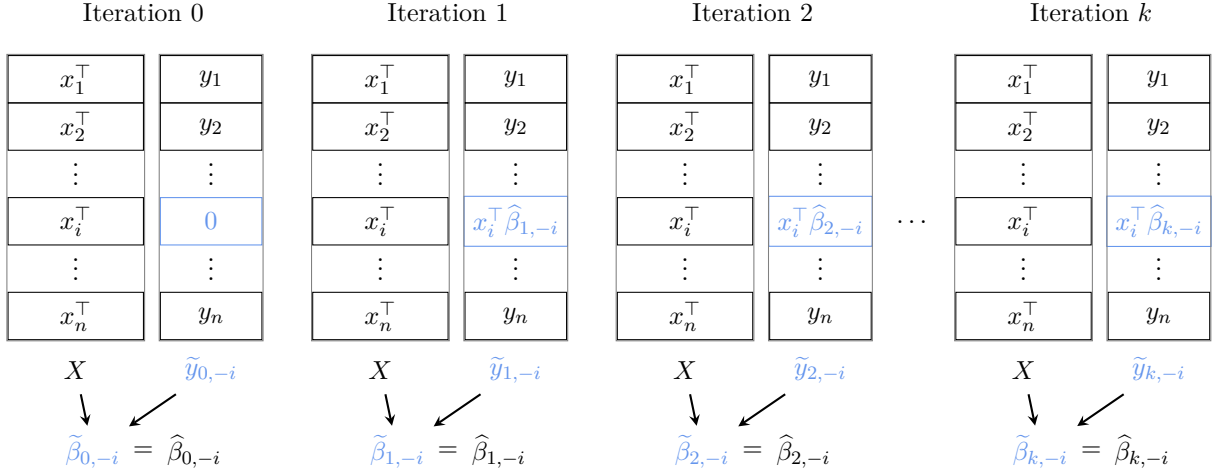
\begin{figure}
    \centering
    \begin{tikzpicture}[every node/.style={minimum size=0.7cm},on grid,scale=0.9,transform shape]
    
    \node at (1, 1) {\scalebox{1}{Iteration 0}};
    
    \node[draw,minimum width=2cm, minimum height=0.7cm] (kax1) at (0,0) {\scalebox{1}{\(\bx_1^\top\)}};
    \node[draw,minimum width=2cm, minimum height=0.7cm] (kax2) at (0,-0.7) {\scalebox{1}{\(\bx_2^\top\)}};
    \node (kaxdots1) at (0,-1.3) {\(\vdots\)};
    \node[draw,minimum width=2cm, minimum height=0.7cm] (kaxi) at (0,-2.1) {\scalebox{1}{\(\bx_i^\top\)}};
    \node (kaxdots2) at (0,-2.7) {\(\vdots\)};
    \node[draw,minimum width=2cm, minimum height=0.7cm] (kaxn) at (0,-3.5) {\scalebox{1}{\(\bx_n^\top\)}};
    \node[draw,minimum width=1.5cm, minimum height=0.7cm] (kay1) at (2,0) {\scalebox{1}{\(y_1\)}};
    \node[draw,minimum width=1.5cm, minimum height=0.7cm] (kay2) at (2,-0.7) {\scalebox{1}{\(y_2\)}};
    \node (kaydots1) at (2,-1.3) {\(\vdots\)};
    \node[draw,CornflowerBlue,minimum width=1.5cm, minimum height=0.7cm] (kayi) at (2,-2.1) {\scalebox{1}{\( 0 \)}};
    \node (kaydots2) at (2,-2.7) {\(\vdots\)};
    \node[draw,minimum width=1.5cm, minimum height=0.7cm] (kayn) at (2,-3.5) {\scalebox{1}{\(y_n\)}};
    
    \draw[gray] ([xshift=-0.01cm, yshift=0.01cm]kax1.north west) rectangle ([xshift=0.01cm, yshift=-0.01cm]kaxn.south east);
    
    \draw[gray] ([xshift=-0.01cm, yshift=0.01cm]kay1.north west) rectangle ([xshift=0.01cm, yshift=-0.01cm]kayn.south east);
    
    \node (feat-gd-aug) at (0, -4.3) {\scalebox{1}{\(\mathbf{X}\)}};
    \node [CornflowerBlue] (resp-gd-aug) at (2, -4.3) {\scalebox{1}{\(\widetilde{\mathbf{y}}_{0,-i}\)}};
    
    \node [CornflowerBlue] (aug0) at (0.25,-5.5) {\scalebox{1}{\(\widetilde{\bbeta}_{0, -i}\)}};
    
    \draw[arrow] (feat-gd-aug) -- (aug0);
    \draw[arrow] (resp-gd-aug) -- (aug0);
    
    \node at (1,-5.5) {\scalebox{1}{\(=\)}};

    \node (loo0) at (1.75,-5.5) {\scalebox{1}{\(\hat{\bbeta}_{0, -i}\)}};

    \node at (5.5, 1) {\scalebox{1}{Iteration 1}};
    
    \node[draw,minimum width=2cm, minimum height=0.7cm] (kax1) at (4.5,0) {\scalebox{1}{\(\bx_1^\top\)}};
    \node[draw,minimum width=2cm, minimum height=0.7cm] (kax2) at (4.5,-0.7) {\scalebox{1}{\(\bx_2^\top\)}};
    \node (kaxdots1) at (4.5,-1.3) {\(\vdots\)};
    \node[draw,minimum width=2cm, minimum height=0.7cm] (kaxi) at (4.5,-2.1) {\scalebox{1}{\(\bx_i^\top\)}};
    \node (kaxdots2) at (4.5,-2.7) {\(\vdots\)};
    \node[draw,minimum width=2cm, minimum height=0.7cm] (kaxn) at (4.5,-3.5) {\scalebox{1}{\(\bx_n^\top\)}};
    \node[draw,minimum width=1.5cm, minimum height=0.7cm] (kay1) at (6.5,0) {\scalebox{1}{\(y_1\)}};
    \node[draw,minimum width=1.5cm, minimum height=0.7cm] (kay2) at (6.5,-0.7) {\scalebox{1}{\(y_2\)}};
    \node (kaydots1) at (6.5,-1.3) {\(\vdots\)};
    \node[draw,CornflowerBlue,minimum width=1.5cm, minimum height=0.7cm] (kayi) at (6.5,-2.1) {\( \bx_i^\top \hat{\bbeta}_{1, -i} \)};
    \node (kaydots2) at (6.5,-2.7) {\(\vdots\)};
    \node[draw,minimum width=1.5cm, minimum height=0.7cm] (kayn) at (6.5,-3.5) {\scalebox{1}{\(y_n\)}};
    
    \draw[gray] ([xshift=-0.01cm, yshift=0.01cm]kax1.north west) rectangle ([xshift=0.01cm, yshift=-0.01cm]kaxn.south east);
    
    \draw[gray] ([xshift=-0.01cm, yshift=0.01cm]kay1.north west) rectangle ([xshift=0.01cm, yshift=-0.01cm]kayn.south east);
    
    \node (feat-gd-aug) at (4.5, -4.3) {\scalebox{1}{\(\mathbf{X}\)}};
    \node [CornflowerBlue] (resp-gd-aug) at (6.5, -4.3) {\scalebox{1}{\(\widetilde{\mathbf{y}}_{1,-i}\)}};
    
    \node [CornflowerBlue] (aug1) at (4.75,-5.5) {\scalebox{1}{\(\widetilde{\bbeta}_{1, -i}\)}};
    
    \draw[arrow] (feat-gd-aug) -- (aug1);
    \draw[arrow] (resp-gd-aug) -- (aug1);
    
    \node at (5.5,-5.5) {\scalebox{1}{\(=\)}};

    \node (loo1) at (6.25,-5.5) {\scalebox{1}{\(\hat{\bbeta}_{1, -i}\)}};

    \node at (10, 1) {\scalebox{1}{Iteration 2}};
    
    \node[draw,minimum width=2cm, minimum height=0.7cm] (kax1) at (9,0) {\scalebox{1}{\(\bx_1^\top\)}};
    \node[draw,minimum width=2cm, minimum height=0.7cm] (kax2) at (9,-0.7) {\scalebox{1}{\(\bx_2^\top\)}};
    \node (kaxdots1) at (9,-1.3) {\(\vdots\)};
    \node[draw,minimum width=2cm, minimum height=0.7cm] (kaxi) at (9,-2.1) {\scalebox{1}{\(\bx_i^\top\)}};
    \node (kaxdots2) at (9,-2.7) {\(\vdots\)};
    \node[draw,minimum width=2cm, minimum height=0.7cm] (kaxn) at (9,-3.5) {\scalebox{1}{\(\bx_n^\top\)}};
    \node[draw,minimum width=1.5cm, minimum height=0.7cm] (kay1) at (11,0) {\scalebox{1}{\(y_1\)}};
    \node[draw,minimum width=1.5cm, minimum height=0.7cm] (kay2) at (11,-0.7) {\scalebox{1}{\(y_2\)}};
    \node (kaydots1) at (11,-1.3) {\(\vdots\)};
    \node[draw,CornflowerBlue,minimum width=1.5cm, minimum height=0.7cm] (kayi) at (11,-2.1) {\( \bx_i^\top \hat{\bbeta}_{2, -i} \)};
    \node (kaydots2) at (11,-2.7) {\(\vdots\)};
    \node[draw,minimum width=1.5cm, minimum height=0.7cm] (kayn) at (11,-3.5) {\scalebox{1}{\(y_n\)}};
    
    \draw[gray] ([xshift=-0.01cm, yshift=0.01cm]kax1.north west) rectangle ([xshift=0.01cm, yshift=-0.01cm]kaxn.south east);
    
    \draw[gray] ([xshift=-0.01cm, yshift=0.01cm]kay1.north west) rectangle ([xshift=0.01cm, yshift=-0.01cm]kayn.south east);
    
    \node (feat-gd-aug) at (9, -4.3) {\scalebox{1}{\(\mathbf{X}\)}};
    \node [CornflowerBlue] (resp-gd-aug) at (11, -4.3) {\scalebox{1}{\(\widetilde{\mathbf{y}}_{2,-i}\)}};
    
    \node [CornflowerBlue] (aug2) at (9.25,-5.5) {\scalebox{1}{\(\widetilde{\bbeta}_{2, -i}\)}};
    
    \draw[arrow] (feat-gd-aug) -- (aug2);
    \draw[arrow] (resp-gd-aug) -- (aug2);
    
    \node at (10,-5.5) {\scalebox{1}{\(=\)}};

    \node (loo2) at (10.75,-5.5) {\scalebox{1}{\(\hat{\bbeta}_{2, -i}\)}};

    \node (kaxdots1) at (12.35,-2.1) {\(\hdots\)};

    \node at (15, 1) {\scalebox{1}{Iteration $k$}};
    
    \node[draw,minimum width=2cm, minimum height=0.7cm] (kax1) at (14,0) {\scalebox{1}{\(\bx_1^\top\)}};
    \node[draw,minimum width=2cm, minimum height=0.7cm] (kax2) at (14,-0.7) {\scalebox{1}{\(\bx_2^\top\)}};
    \node (kaxdots1) at (14,-1.3) {\(\vdots\)};
    \node[draw,minimum width=2cm, minimum height=0.7cm] (kaxi) at (14,-2.1) {\scalebox{1}{\(\bx_i^\top\)}};
    \node (kaxdots2) at (14,-2.7) {\(\vdots\)};
    \node[draw,minimum width=2cm, minimum height=0.7cm] (kaxn) at (14,-3.5) {\scalebox{1}{\(\bx_n^\top\)}};
    \node[draw,minimum width=1.5cm, minimum height=0.7cm] (kay1) at (16,0) {\scalebox{1}{\(y_1\)}};
    \node[draw,minimum width=1.5cm, minimum height=0.7cm] (kay2) at (16,-0.7) {\scalebox{1}{\(y_2\)}};
    \node (kaydots1) at (16,-1.3) {\(\vdots\)};
    \node[draw,CornflowerBlue,minimum width=1.5cm, minimum height=0.7cm] (kayi) at (16,-2.1) {\( \bx_i^\top \hat{\bbeta}_{k, -i} \)};
    \node (kaydots2) at (16,-2.7) {\(\vdots\)};
    \node[draw,minimum width=1.5cm, minimum height=0.7cm] (kayn) at (16,-3.5) {\scalebox{1}{\(y_n\)}};
    
    \draw[gray] ([xshift=-0.01cm, yshift=0.01cm]kax1.north west) rectangle ([xshift=0.01cm, yshift=-0.01cm]kaxn.south east);
    
    \draw[gray] ([xshift=-0.01cm, yshift=0.01cm]kay1.north west) rectangle ([xshift=0.01cm, yshift=-0.01cm]kayn.south east);
    
    \node (feat-gd-aug) at (14, -4.3) {\scalebox{1}{\(\mathbf{X}\)}};
    \node [CornflowerBlue] (resp-gd-aug) at (16, -4.3) {\scalebox{1}{\(\widetilde{\mathbf{y}}_{k,-i}\)}};
    
    \node [CornflowerBlue] (kestiter1) at (14.25,-5.5) {\scalebox{1}{\(\widetilde{\bbeta}_{k, -i}\)}};
    
    \draw[arrow] (feat-gd-aug) -- (kestiter1);
    \draw[arrow] (resp-gd-aug) -- (kestiter1);
    
    \node at (15,-5.5) {\scalebox{1}{\(=\)}};

    \node (looiter1) at (15.75,-5.5) {\scalebox{1}{\(\hat{\bbeta}_{k, -i}\)}};

    \end{tikzpicture}
    \caption{Illustration of the modified augmented system for LOO in GD.} 
    \label{fig:modified-augmentation}
\end{figure}

\subsection{Towards exact and efficient LOOCV for GD?} 
\label{sec:efficient-loocv}

We can unravel the relationships in LOO coefficients between iterations of the
modified augmented system to arrive at explicit recursive forms for
\smash{$(h_{ij}^{(k)})_{j \leq n}$} and \smash{$b_i^{(k)}$} in
\Cref{lemma:hat-b}, given next.     

\begin{restatable}
    [Recursive shortcut formula for LOO predictions in GD]
    {proposition}
    {PropEfficientLooGd}
    \label{prop:efficient-loo-gd}
    For all $k \in [K]$ and $i \in [n]$, 
    \[
      \bx_{i}^\top \hat{\bbeta}_{k, -i}
        = \bx_i^{\top} \hat\bbeta_k + A_{i, k} \|\bx_i\|_2^2 + \sum_{j = 1}^{k -
          1} B_{i,k}^{(j)} \bx_i^{\top} (\bX^{\top} \bX)^j \bx_i,
    \]
    where
\begin{align*}
	A_{i, k + 1} &= A_{i, k} + \frac{2\delta_kA_{i,k} \|\bx_i\|_2^2}{n} + 
                 \sum_{j = 1}^{k - 1}\frac{2\delta_k B_{i,k}^{(j)}
                 \bx_i^{\top}(\bX^{\top} \bX)^j \bx_i}{n} +
                 \frac{2\delta_{k+1}(\bx_i^{\top}\hat\bbeta_k - y_i)}{n}, \\  
	B_{i, k + 1}^{(1)} &= B_{i, k}^{(1)} - \frac{2\delta_k A_{i, k}}{n}, \\
	B_{i, k + 1}^{(j)} &= B_{i, k}^{(j)} - \frac{2\delta_k B_{i,k}^{(j-1)}}{n},
                       \quad 2 \leq j \leq k,  
\end{align*}
and we make the convention that \smash{$B_{i,k}^{(k)} = 0$}. 
\end{restatable}

Using this proposition, we can estimate generic prediction risk functionals as
follows. Abbreviating \smash{$\cH_{ij} = \bx_i^{\top} (\bX^{\top} \bX)^j 
  \bx_i$}, to estimate the risk functional \eqref{eq:pred-functional}, we use:
\begin{equation}
    \label{eq:Tloocv-shortcut}
    \Psi^{\loo}(\hat\bbeta_k) =
    \frac{1}{n}
    \sum_{i=1}^{n}
    \psi
    \bigg(
    y_i, \, x_i^\top \hat \bbeta_k
    + A_{i,k} \| \bx_i \|_2^2
    + \sum_{j=1}^{k-1} B_{i,k}^{(j)}
    \cH_{ij}
    \bigg).
\end{equation}
To be clear, \eqref{eq:Tloocv-shortcut} is an \emph{exact} shortcut formula for
\eqref{eq:Tloocv}. 

In the $p \asymp n$ regime, the computational cost of a naive implementation of
LOOCV for $k$-step GD is $O(n^3 k)$. (Each GD step costs $O(n^2)$, as we
must compute $p$ inner products, each of length $n$; then multiply this by $k$
steps and $n$ LOO predictions). In comparison, the shortcut formula given above
can be shown to require $O(n^3 + n^2 k + nk^2)$ operations using a spectral 
decomposition of $X$. If $k$ is large, say, itself proportional to $n$, then we
can see that the shortcut formula is more efficient.  

This is certainly not meant to be the final word on efficient LOOCV along the
GD path. For one, a spectral decomposition is prohibitive for large
problems (more expensive than solving the original least squares problem  
\eqref{eq:ols-obj}), and there may be alternative perspectives on the shortcut 
formula given in \Cref{prop:efficient-loo-gd} that lead to faster
implementation. Further, if $n$ is large enough, then stochastic variants of GD  
would be preferred in place of batch GD. All that said, the above analysis
should be seen as a demonstration that exact shortcuts for LOO predictions in GD
are \emph{possible}, and may inspire others to develop more practical exact or 
approximate LOO methods.   

\section*{Acknowledgments}

We thank Alnur Ali, Arun Kumar Kuchibhotla, Arian Maleki, Alessandro Rinaldo,
and Yuting Wei for enjoyable discussions related to this project, and for
lending ears to parts of these results a while back.
We also thank the anonymous reviewers for their constructive feedback, which has 
helped improve the clarity of the manuscript. 
We thank Evan Chen for inspiring parts of our color palette. 
The idea of providing proof blueprints is inspired
by the \texttt{leanblueprint} plugin used in the Lean theorem prover. 
PP and RJT were supported by ONR grant N00014-20-1-2787.

\bibliographystyle{plainnat}
\bibliography{pratik.bib}
 
\clearpage
\appendix

\newgeometry{left=0.5in,top=0.5in,right=0.5in,bottom=0.25in,head=0.05in,foot=0.05in}

\begin{center}
\Large
{
\bf
\framebox{Supplement}
}
\end{center}

\bigskip

This document serves as a supplement to the paper ``Failures and Successes of
Cross-Validation for Early-Stopped Gradient Descent.'' The structure of the
supplement is outlined below, followed by a summary of the notation and
conventions used in both the main paper and this supplement. The section and
figure numbers in this supplement begin with the letter ``S'' and the equation
numbers begin with the letter ``E'' to differentiate them from those appearing
in the main paper.

\setcounter{section}{0}
\setcounter{equation}{0}
\setcounter{figure}{0}
\renewcommand{\thesection}{S.\arabic{section}}
\renewcommand{\theequation}{E.\arabic{equation}}
\renewcommand{\thefigure}{S.\arabic{figure}}

\section*{Organization}

\begin{itemize}[leftmargin=5mm]
    \item \Cref{sec:gcv-inconsistency-proof-ingredients} provides the main steps involved in the proofs of \Cref{thm:gcv-inconsistency}.

    \begin{table}[!ht]
    \centering
    \begin{tabular}{l  l  l}
        \toprule
        \textbf{Section} & \textbf{Content} & \textbf{Purpose} \\
        \midrule
        \Cref{sec:gcv-inconsistency-proof-ingredients-step1} & \Cref{lem:gd-gf-equi-risk,lem:gd-gf-equi-gcv} & 
        Equivalences between gradient descent and flow for risk and GCV \\
        \Cref{sec:gcv-inconsistency-proof-ingredients-step2} & \Cref{lem:risk-asymptotics-gf,lem:gcv-asymptotics-gf} & Asymptotics of risk and GCV for gradient flow \\
        \Cref{sec:gcv-inconsistency-proof-ingredients-step3} & \Cref{lem:risk-gcv-mismatch-gf} & Mismatch of risk and GCV asymptotics for gradient flow \\
        \addlinespace[0.0ex] \arrayrulecolor{black}
        \bottomrule
    \end{tabular}
    \end{table}

    \item \Cref{sec:supporting-lemmas-proof-thm:gcv-inconsistency} contains supporting lemmas that are used in the proof of \Cref{thm:gcv-inconsistency}.

    \begin{table}[!ht]
    \centering
    \begin{tabular}{l  l  l}
        \toprule
        \textbf{Section} & \textbf{Content} & \textbf{Purpose} \\
        \midrule
        \Cref{sec:gd-gf-equi} & \Cref{lemma:scalar-uniform-approximation} & Closeness between gradient descent and flow \\
        \Cref{sec:useful-concen-results-gcv-inconsistency} & \Cref{lem:concen-linform,lem:concen-quadform} & Statements of concentration results for linear and quadratic forms  \\
        \addlinespace[0.0ex] \arrayrulecolor{black}
        \bottomrule
    \end{tabular}
    \label{tab:sec:supporting-lemmas-proof-thm:gcv-inconsistency}
    \end{table}
 
    \item \Cref{sec:proof-inconsistency} contains the proof of \Cref{thm:gcv-inconsistency}.

    \begin{table}[!ht]
    \centering
    \begin{tabular}{l  l  l}
        \toprule
        \textbf{Section} & \textbf{Content} & \textbf{Purpose} \\
        \midrule
        \Cref{sec:outline-thm:gcv-inconsistency} & \cellcolor{lightgray!25} & Proof schematic \\
        \Cref{sec:proof-lem:gd-gf-equi-risk} & \cellcolor{lightgray!25} & 
        Proof of \Cref{lem:gd-gf-equi-risk} \\
        
        \Cref{sec:proof-lem:gd-gf-equi-gcv} & \cellcolor{lightgray!25} & 
        Proof of \Cref{lem:gd-gf-equi-gcv} \\
        
        \Cref{sec:proof-lem:risk-asymptotics-gf} & \cellcolor{lightgray!25} & 
        Proof of \Cref{lem:risk-asymptotics-gf} \\
        
        \Cref{sec:proof-lem:gcv-asymptotics-gf} & \cellcolor{lightgray!25} & 
        Proof of \Cref{lem:gcv-asymptotics-gf} \\
        \Cref{sec:proof-lem:risk-gcv-mismatch-gf} & \cellcolor{lightgray!25} & 
        Proof of \Cref{lem:risk-gcv-mismatch-gf} \\
        \addlinespace[0.0ex] \arrayrulecolor{black}
        \Cref{sec:mp-moments} & \cellcolor{lightgray!25} & A helper lemma related to the Marchenko-Pastur law \\
        \bottomrule
    \end{tabular}
    \label{tab:sec:proof-inconsistency}
    \end{table}

    \item \Cref{sec:thm:uniform-consistency-squared-proof-ingredients} provides the main steps involved in the proofs of \Cref{thm:uniform-consistency-squared}.
    
    \begin{table}[!ht]
    \centering
    \begin{tabular}{l  l  l}
        \toprule
        \textbf{Section} & \textbf{Content} & \textbf{Purpose} \\
        \midrule
        \Cref{sec:thm:uniform-consistency-squared-proof-ingredients-step1} & \Cref{lemma:gradient-upper-bound,lemma:Rloo} & Concentration of the LOOCV estimator \\
        \Cref{sec:thm:uniform-consistency-squared-proof-ingredients-step2} & \Cref{lemma:cRk} & Concentration of the risk  \\
        \Cref{sec:thm:uniform-consistency-squared-proof-ingredients-step3} & \Cref{lemma:projection-effects,lemma:expectation-close} & LOOCV bias analysis \\
        \addlinespace[0.0ex] \arrayrulecolor{black}
        \bottomrule
    \end{tabular}
    \end{table}
    
    \item \Cref{sec:helper-lemmas-uniform-consistency} contains supporting lemmas that are used in the proofs of \Cref{thm:uniform-consistency-squared,thm:uniform-consistency-general,thm:coverage}. 

    \begin{table}[!ht]
        \centering
        \begin{tabular}{l  l  l}
            \toprule
            \textbf{Section} & \textbf{Content} & \textbf{Purpose} \\
            \midrule
            \Cref{sec:definitions} & \Cref{def:lsi} & Technical preliminaries \\
            \Cref{sec:T2} & \Cref{prop:Gozlan} & Useful property of the \(T_2\)-inequality \\
            \Cref{sec:dimension-free} & \Cref{lemma:concentration} & Dimension-free concentration inequality \\
            \Cref{sec:op-and-energy} & \Cref{lemma:op-Sigma,lemma:norm-y} & Upper bounds on operator norm of \(\hat\bSigma\) and \(\|\by\|_2\) \\
            \Cref{sec:other} & \Cref{lemma:norm-theta,lemma:op-subGaussian} & Upper bounds on \(\|\E[y_0 \bx_0]\|\) and sub-exponential of \(\|\hat\bSigma\|_{\op}\) \\
            \Cref{eq:upper-bound-two-beta} & \Cref{lemma:upper-bound-beta} and \Cref{cor:y-Xbeta} & Upper bounds on \(\|\hat\bbeta_k\|_2\) and \(\|\hat\bbeta_{k, -i}\|_2\) \\
            \Cref{sec:upper-bound-loocv-res} & \Cref{lemma:y-xbeta} & Upper bounds on LOOCV residuals \(\{|y_i - \bx_i^{\top} \hat{\bbeta}_{k, -i}|\}_{i \in [n]}\) \\
            \addlinespace[0.0ex] \arrayrulecolor{black}
            \bottomrule
        \end{tabular}
        \label{tab:sec:helper-lemmas-uniform-consistency}
    \end{table}

    \newpage
    \item \Cref{sec:uniform-consistency-proof-squared}
    contains the proof of \Cref{thm:uniform-consistency-squared}.
        
    \begin{table}[!ht]
    \centering
    \begin{tabular}{l  l  l}
        \toprule
        \textbf{Section} & \textbf{Content} & \textbf{Purpose} \\
        \midrule
        \Cref{sec:proof-thm:uniform-consistency} & \cellcolor{lightgray!25} & Proof schematic \\
        \Cref{sec:proof-lemma:Rloo} & \cellcolor{lightgray!25} & 
        Proof of \Cref{lemma:Rloo} \\
        \Cref{sec:proof-lemma:cRk} & \cellcolor{lightgray!25} & 
        Proof of \Cref{lemma:cRk} \\
        \Cref{sec:proof-lemma:projection-effects} & \cellcolor{lightgray!25} & 
        Proof of \Cref{lemma:projection-effects} \\
        \Cref{sec:proof-lemma:expectation-close} & \cellcolor{lightgray!25} & 
        Proof of \Cref{lemma:projection-effects} \\
        \addlinespace[0.0ex] \arrayrulecolor{black}
        \bottomrule
    \end{tabular}
    \label{tab:sec:uniform-consistency-proof-squared}
    \end{table}

    \item \Cref{sec:proof-lemma:gradient-upper-bound} contains the proof of \Cref{lemma:gradient-upper-bound} that forms a key component in the proof of \Cref{thm:uniform-consistency-squared}.

    \begin{table}[!ht]
        \centering
        \begin{tabular}{l  l  l}
            \toprule
            \textbf{Section} & \textbf{Contents} & \textbf{Purpose} \\
            \midrule
            \Cref{sec:outline-proof-lemma:gradient-upper-bound} & \cellcolor{lightgray!25} & Proof schematic \\
            \Cref{sec:proof-lem:bound-norm-gradient-wrt-features} & \Cref{lem:bound-norm-gradient-wrt-features,lemma:gradx,lemma:V} & Upper bounding norm of the gradient with respect to the features \\
            \Cref{sec:proof-lem:bound-norm-gradient-wrt-response} & \Cref{lem:bound-norm-gradient-wrt-response} & Upper bounding norm of the gradient with respect to the response \\
            \addlinespace[0.0ex] \arrayrulecolor{black}
            \bottomrule
        \end{tabular}
        \label{tab:sec:proof-lemma:gradient-upper-bound}
    \end{table}

    \item \Cref{sec:proof-thm:uniform-consistency-general} contains the proof of \Cref{thm:uniform-consistency-general} for general risk functionals.

    \begin{table}[!ht]
        \centering
        \begin{tabular}{l  l  l}
            \toprule
            \textbf{Section} & \textbf{Contents} & \textbf{Purpose} \\
            \midrule
            \Cref{sec:outline-thm:uniform-consistency-general} & \cellcolor{lightgray!25} & Proof schematic \\
            \Cref{sec:concentration-analysis-general} & \Cref{lem:loo-risk-concentration-general} & Concentration analysis for LOOCV estimator and prediction risk \\ 
            \Cref{sec:uniform-consistency-general} & \Cref{lem:projection-effects-general} & Demonstrating that projection has little effect on quantities of interest \\
            \addlinespace[0.0ex] \arrayrulecolor{black}
            \bottomrule
        \end{tabular}
        \label{tab:sec:proof-thm:uniform-consistency-general}
    \end{table}

    \item \Cref{sec:proof-thm:coverage} contains the proof of \Cref{thm:coverage}.
    The proof uses the component \Cref{lemma:Lipschitz-F}.

    \item 
    \Cref{sec:compare} provides 
    proofs of results related to the naive and modified augmentation systems (\Cref{lemma:beta_tilde=beta,prop:efficient-loo-gd} and \Cref{lemma:hat-b}) for LOOCV along the gradient path in \Cref{sec:computational}.

    \begin{table}[!ht]
    \centering
    \begin{tabular}{l  l  l}
        \toprule
        \textbf{Section} & \textbf{Content} & \textbf{Purpose} \\
        \midrule
        \Cref{sec:proof-lemma:beta_tilde=beta} & \cellcolor{lightgray!25} & Proof of \Cref{lemma:beta_tilde=beta} \\
        \Cref{sec:proof-lemma:hat-b} & \cellcolor{lightgray!25} & Proof of \Cref{lemma:hat-b} \\
        \Cref{sec:additional-details-computational} & \cellcolor{lightgray!25} & Proof of \Cref{prop:efficient-loo-gd} \\
        \addlinespace[0.0ex] \arrayrulecolor{black}
        \bottomrule
    \end{tabular}
    \label{tab:sec:compare}
    \end{table}

    \item \Cref{sec:additional-numerical-illustrations} provides an additional numerical illustration and details of the setups for \Cref{fig:gcv-inconsistency-with-loocv-n2500-main,fig:pred-intervals}.

    \begin{table}[!ht]
    \centering
    \begin{tabular}{l  l  l}
        \toprule
        \textbf{Section} & \textbf{Content} & \textbf{Purpose} \\
        \midrule
        \Cref{sec:mismatch_illustration_varying_snr} & \Crefrange{fig:gf_limit_mismatch_surface_sum_moderate_signal_energy}{fig:gf_limit_mismatch_surface_sum_verylow_noise_energy} & Additional illustrations in \Cref{sec:combined_sum_mismatch} \\
        \Cref{sec:setup-details} & \cellcolor{lightgray!25} & Setup details for \Cref{fig:gcv-inconsistency-with-loocv-n2500-main,fig:pred-intervals} \\
        \Cref{sec:prediction-intervals-linear-model} & \Cref{fig:pred-intervals-linear-model} & Additional illustrations related to \Cref{fig:pred-intervals} \\
        \Cref{sec:additional-illustrations-distributional-closeness} & \Crefrange{fig:test-loo-dist-comparison-sq-vs-abs}{fig:test-loo-dist-comparison-ridges} & Additional illustrations related to \Cref{fig:test-loo-dist-comparison}  \\
        \addlinespace[0.0ex] \arrayrulecolor{black}
        \bottomrule
    \end{tabular}
    \label{tab:sec:additional-numerical-illustrations}
    \end{table}
\end{itemize}

\restoregeometry

\section*{Notation}

\begin{itemize}[leftmargin=5mm]
    \item \textbf{General notation.}
    We denote vectors in non-bold lowercase (e.g., $x$) and matrices in non-bold uppercase (e.g., $\bX$).
    We use blackboard letters to denote some special sets: $\NN$ denotes the set of positive integers, and $\RR$ denotes the set of real numbers.
    We use calligraphic font letters to denote sets or certain limiting functions (e.g., $\cX$).
    For a positive integer $n$, we use the shorthand $[n]$ to denote the set $\{1,\dots, n\}$. 
    For a pair of real numbers $x$ and $y$, we use $x \wedge y$ to denote $\min\{x, y\}$, and $x \vee y$ to denote $\max\{x, y\}$.
    For an event or set $A$, $\ind_A$ denotes the indicator random variable associated with $A$.

    \item \textbf{Vector and matrix notation.}
    For a vector $\bx$, $\| \bx \|_2$ denotes its $\ell_2$ norm.
    For $\bv \in \R^n$ and $k \in \NN_+$, we let $\bv_{1:k} \in \R^k$ be the vector that contains the first $k$ coordinates of $\bv$. 
    For a matrix $\bX \in \RR^{n \times p}$, $\bX^\top \in \RR^{p \times n}$ denotes its transpose, and $\bX^{\dagger} \in \RR^{p \times n}$ denotes its Moore-Penrose inverse.
    For a square matrix $\bA \in \RR^{p \times p}$, $\tr[\bA]$ denotes its trace, and $\bA^{-1} \in \RR^{p \times p}$ denotes its inverse, provided that it is invertible.
    For a positive semidefinite matrix $\bSigma$, $\bSigma^{1/2}$ denotes its principal square root.
    A $p \times p$ identity matrix is denoted $\bI_p$, or simply by $\bI$ when it is clear from the context.
    For a matrix $\bX$, we denote its operator norm with respect to $\ell_2$ vector norm by $\| \bX \|_{\mathrm{op}}$ and its Frobenius norm by $\| \bX \|_F$.
    For a matrix $\bM$, $\| \bX \|_{\tr}$ denotes the trace norm of $\bM$, which is the sum of all its singular values.

    \item \textbf{Asymptotics notation.}
    For a nonnegative quantity $Y$, we use $X = O_\alpha(Y)$ to denote the deterministic big-O notation that indicates the bound $| X | \le C_\alpha Y$, where $C_\alpha$ is some numerical constant that can depend on the ambient parameter $\alpha$ but otherwise does not depend on other parameters in the context.
    We denote the probabilistic big-O notation by $O_p$.
    We denote convergence in probability by \smash{``$\pto$''} and almost sure convergence by \smash{``$\asto$''}.
\end{itemize}

\section*{Conventions}

\begin{itemize}[leftmargin=5mm]
    \item 
    Throughout, $C$ and $C'$ (not to be confused with derivative) denote positive absolute constants.
    \item
    If no subscript is specified for the norm $\| \bx \|$ of a vector $\bx$, then it is assumed to be the $\ell_2$ norm.
    \item
    We use the following color scheme for various mathematical environments:
    \begin{itemize}
        \item \tikz[baseline=(X.base)] \node[draw=none, fill=ForestGreen!10,
          text=ForestGreen, inner sep=4pt, rounded corners=0pt] (X) {{\bfseries
              \textcolor{ForestGreen!70!black}{Assumption}}: ...};
        
        \item \tikz[baseline=(X.base)] \node[draw=CornflowerBlue,
          fill=CornflowerBlue!10, inner sep=4pt, rounded corners=2pt] (X)
          {{\bfseries \textcolor{CornflowerBlue!40!black}{Theorem}}: ...};      

        \item \tikz[baseline=(X.base)] \node[draw=none, fill=CornflowerBlue!10,
          inner sep=4pt, rounded corners=0pt] (X) {{\bfseries
              \textcolor{CornflowerBlue!40!black}{Proposition}}: ...};

        \item \tikz[baseline=(X.base)] \node[draw=none, fill=RedViolet!8!gray!8, inner sep=4pt, rounded corners=0pt] (X) {{\bfseries \textcolor{black}{Lemma/Corollary}}: ...};
    \end{itemize}
    \item
    If a proof of a statement is separated from the statement, the statement is restated (while keeping the original numbering) along with the proof for the reader's convenience.
\end{itemize}

\clearpage
\section{Proof sketch for \Cref{thm:gcv-inconsistency}}
\label{sec:gcv-inconsistency-proof-ingredients}

In this section, we outline the idea behind the proof of \Cref{thm:gcv-inconsistency}.
The detailed proof can be found in \Cref{sec:proof-inconsistency}.

\subsection{Step 1: Closeness between gradient descent and gradient flow}
\label{sec:gcv-inconsistency-proof-ingredients-step1}

This step involves establishing equivalences between gradient descent and gradient flow, specifically for the downstream analysis of risk and generalized cross-validation. 

\paragraph{Smoothers for gradient descent and flow.\hspace{-0.5em}}
We start by rearranging the terms in \eqref{eq:gd-iterate} in the form of a first-order difference equation:
\begin{equation}
	\label{eq:gd_diffeq}
	\frac{\hat\bbeta_k - \hat\bbeta_{k-1}}{\delta} =  \frac{1}{n}
	\bX^\top (\by - \bX \hat\bbeta_{k-1}).
\end{equation}
(Recall we are considering a fixed step size of $\delta$ and initialization at the origin $\hat\bbeta_0 = 0$.)
To consider a continuous time analog of \eqref{eq:gd_diffeq},
we imagine the interval $(0, t)$ is divided into $k$ pieces 
each of size $\delta$.
Letting $\hat\bbeta_t^{\gf} = \hat\bbeta_{k}$ at time $t = k \delta$ 
and taking the limit $\delta \to 0$,
we arrive at an ordinary differential equation:
\begin{align}
    \label{eq:gf_diffeq} 
    \frac{\partial}{\partial t}\hat\bbeta_t^{\gf} = \frac{1}{n} \bX^{\top}(\by - \bX \hat\bbeta_t^{\gf}),
\end{align}
with the initial condition $\hat\bbeta_0^{\gf} = 0$.
We refer to \eqref{eq:gf_diffeq} as the gradient flow differential equation.
The gradient flow (GF) estimate has a closed-form solution:
\begin{align}\label{eq:beta-gf}
    \hat\bbeta_t^{\gf} = \hat\bSigma^{\dagger} \big( \id_p - \exp(-t \hat\bSigma) \big) \cdot \frac{1}{n} \bX^{\top} \by, 
\end{align}
where $\hat\bSigma^{\dagger}$ stands for the Moore-Penrose generalized
inverse of $\hat\bSigma$. 
Also, recall from \Cref{sec:gcv_loocv} that by rolling out the iterations,
the gradient descent iterate at step $k$ can be expressed as:
\begin{align}\label{eq:betakk}
    \hat\bbeta_k = \sum_{j = 0}^{k - 1} \delta \big( \id_{p} - \delta \hat\bSigma \big)^{k - j - 1} \cdot \frac{1}{n} \bX^{\top}\by. 
\end{align}
We can define the corresponding GCV estimates for the squared risk as follows:
\begin{align*}
    \hat{R}^{\gcv}(\hat\bbeta_k) = \frac{1}{n} \frac{\|\by - \bX \hat\bbeta_k\|_2^2}{(1 - \tr(\bH_k) / n)^2}
    \quad
    \text{and}
    \quad
    \hat{R}^{\gcv}(\hat\bbeta_t^{\gf}) = \frac{1}{n} \frac{\|\by - \bX \hat\bbeta_t^{\gf}\|_2^2}{(1 - \tr(\bH_t^{\gf}) / n)^2}, 
\end{align*}
where 
\begin{align}\label{eq:Hat-A}
    \bH_k = \sum_{j = 0}^{k - 1} \frac{\delta}{n} \bX \big( \id_{p} - \delta \hat\bSigma \big)^{k - j - 1} \bX^{\top}
    \quad
    \text{and}
    \quad
    \bH_t^{\gf} = \frac{1}{n}\bX(\hat\bSigma)^{\dagger} \big( \id_p - \exp(-t \hat\bSigma) \big) \bX^{\top}. 
\end{align}

We first show that under the conditions of \Cref{thm:gcv-inconsistency}, estimates obtained from GD are in some sense asymptotically equivalent to that obtained from gradient flow (GF), which we define below.

\begin{restatable}
    [Prediction risks are asymptotically equivalent]
    {lemma}
    {LemGdGfEquiRisk}
    \label{lem:gd-gf-equi-risk}
    Under the assumptions of \Cref{thm:gcv-inconsistency}, we have
    \begin{align*}
    |R(\hat\bbeta_\K) - R(\hat\bbeta_T^{\gf})| \asto 0.
    \end{align*}
\end{restatable}

\begin{restatable}
    [GCV risk estimates are asymptotically equivalent]
    {lemma}
    {LemGdGfEquiGCV}
    \label{lem:gd-gf-equi-gcv}
    Under the assumptions of \Cref{thm:gcv-inconsistency}, we have
    \begin{align*}
    \big| \hat{R}^{\gcv}(\hat\bbeta_\K) - \hat{R}^{\gcv} (\hat\bbeta_T^{\gf})\big| \asto 0. 
    \end{align*}
\end{restatable}

The proofs of these equivalences in \Cref{lem:gd-gf-equi-risk,lem:gd-gf-equi-gcv} are provided in \Cref{sec:proof-lem:gd-gf-equi-risk,sec:proof-lem:gd-gf-equi-gcv}, respectively.

\subsection{Step 2: Limiting risk and GCV}
\label{sec:gcv-inconsistency-proof-ingredients-step2}

This step focuses on obtaining asymptotics (limiting behaviors) for risk and GCV when using gradient flow. 

According to \Cref{lem:gd-gf-equi-gcv,lem:gd-gf-equi-risk}, to show that the GCV estimator is inconsistent for the GD risk, it suffices to show that it is inconsistent for the GF risk. 
We next separately derive the limiting expressions for $ R(\hat\bbeta_T^{\gf})$ and $\hat{R}^{\gcv} (\hat\bbeta_T^{\gf})$, respectively. 

Let $F_{\zeta_{\ast}}(s)$ denote the Marchenko-Pastur law:
\begin{itemize}
    \setlength\itemsep{0em}
    \item \emph{Underparameterized.}
    For $\zeta_{\ast} \leq 1$,
    the density is given by:
    \begin{equation}\label{eq:MP-law-le1}
    \frac{\mathrm{d} F_{\zeta_{\ast}}(s)}{\mathrm{d} s} = \frac{1}{2 \pi \zeta_{\ast} s} \sqrt{(b-s)(s-a)} \cdot \ind_{[a, b]}(s).
    \end{equation}
    The density is supported on $[a,b]$, where $a = (1-\sqrt{\zeta_{\ast}})^2$ and $b = (1+\sqrt{\zeta_{\ast}})^2$.
    
    \item \emph{Overparameterized.}
    For $\zeta_{\ast} > 1$,
    the law $F_{\zeta_{\ast}}$ has an additional point mass at 0 of probability $1-1/\zeta_{\ast}$.
    In other words,
    \begin{equation}\label{eq:MP-law-gt1}
    \frac{\mathrm{d} F_{\zeta_{\ast}}(s)}{\mathrm{d} s}
    = \left(1 - \frac{1}{\zeta_{\ast}}\right) \delta_0(s)
    + \frac{1}{2 \pi \zeta_{\ast} s} \sqrt{(b-s)(s-a)} \cdot \ind_{[a, b]}(s).
    \end{equation}
    Here, $\delta_0$ is the Dirac delta function at $0$.
\end{itemize}
We will use some properties of the Marchenko-Pastur law in our proofs.
For some visual illustrations in this section, we will values of $\zeta_{\ast} = 0.5$ and $\zeta_{\ast} = 1.5$ in the underparameterized and overparameterized regimes, respectively.
We recall in \Cref{fig:mp_densities} the corresponding density plots for these two values of $\zeta_{\ast}$.

\begin{figure*}[!ht]
    \centering
  \includegraphics[width=0.495\textwidth]{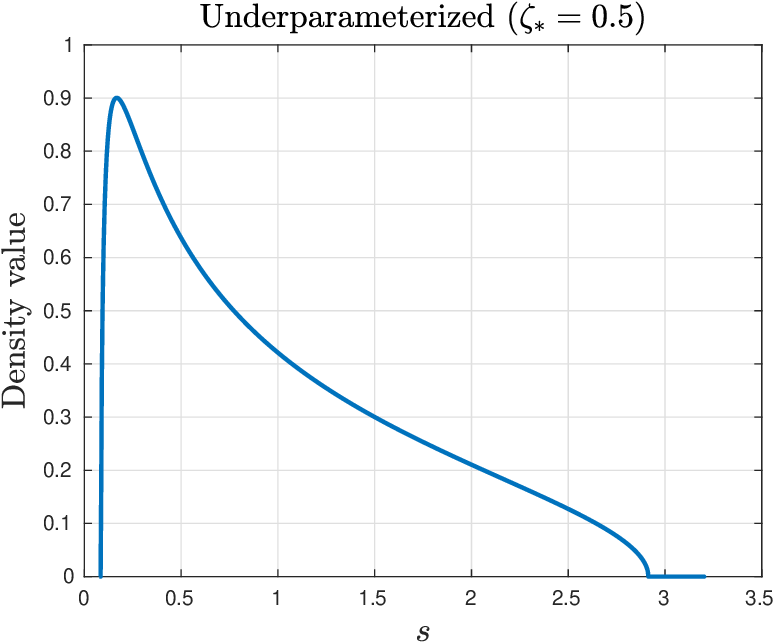}
  \includegraphics[width=0.49\textwidth]{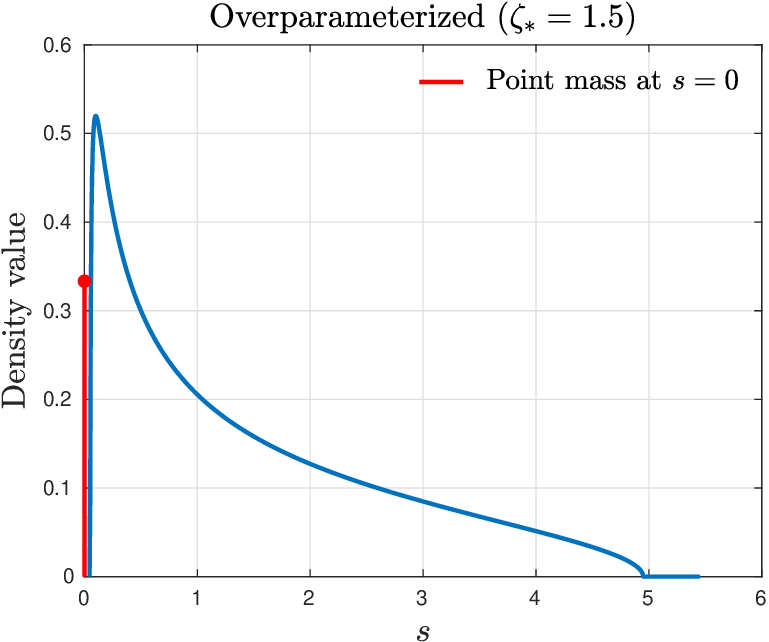}
  \caption{
    Illustration of the Marchenko-Pastur density in the underparameterized (\emph{left}) and overparameterized regimes (\emph{right}).
    Note that in the overparameterized regime, there is a point mass at $s = 0$ (shown with a red dot) as in \eqref{eq:MP-law-gt1}.
    This point mass will need special care in the subsequent asymptotic limits.
  }
  \label{fig:mp_densities}
\end{figure*}

\begin{restatable}
    [Risk limit for gradient flow]
    {lemma}
    {LemRiskAsymptoticsGf}
    \label{lem:risk-asymptotics-gf}
    Under the assumptions of \Cref{thm:gcv-inconsistency},
    \begin{align*}
        R(\hat\bbeta_T^{\gf}) \asto r^2 \int \exp(-2Tz) \, \mathrm{d} F_{\zeta_{\ast}}(z) + \zeta_{\ast}\sigma^2\int z^{-1}(1 - \exp(-Tz))^2 \, \mathrm{d} F_{\zeta_{\ast}}(z) + \sigma^2. 
    \end{align*}
\end{restatable}

\begin{restatable}
    [GCV limit for gradient flow]
    {lemma}
    {LemGCVAsymptoticsGf}
    \label{lem:gcv-asymptotics-gf}
    Under the assumptions of \Cref{thm:gcv-inconsistency}, 
    \begin{align*}
        \hR^{\gcv}(\hat \bbeta_k) \asto \ddfrac{ r^2 \int z \exp(-2Tz) \mathrm{d} F_{\zeta_{\ast}}(z) + \sigma^2(1 - \zeta_{\ast}) + \sigma^2 \zeta_{\ast} \int \exp(-2Tz) \, \mathrm{d} F_{\zeta_{\ast}}(z)}{\left( 1 - \zeta_{\ast} \int (1 - \exp(-Tz)) \, \mathrm{d}F_{\zeta_{\ast}}(z) \right)^{2}}.
    \end{align*}
\end{restatable}

The proofs of these asymptotic limits in \Cref{lem:risk-asymptotics-gf,lem:gcv-asymptotics-gf} are provided in \Cref{sec:proof-lem:risk-asymptotics-gf,sec:proof-lem:gcv-asymptotics-gf}, respectively.

\subsection{Step 3: Limits mismatch}
\label{sec:gcv-inconsistency-proof-ingredients-step3}

The final step involves showing a mismatch between the asymptotics of risk and GCV for gradient flow.

\begin{restatable}
    [Limits mismatch]
    {lemma}
    {LemRiskGCVMismatchGf}
    \label{lem:risk-gcv-mismatch-gf}
    Let $F_{\zeta_{\ast}}$ be the Marchenko-Pastur law.
    Then, assuming either $r^2 > 0$ or $\sigma^2 > 0$, for all $T > 0$, except for a set of Lebesgue measure zero,
    \begin{align}
        & r^2 \int \exp(-2Tz) \, \mathrm{d} F_{\zeta_{\ast}}(z) + \zeta_{\ast}\sigma^2\int z^{-1}(1 - \exp(-Tz))^2 \, \mathrm{d} F_{\zeta_{\ast}}(z) + \sigma^2 \notag \\
        & \quad \neq
        \ddfrac{ r^2 \int z \exp(-2Tz) \mathrm{d} F_{\zeta_{\ast}}(z) + \sigma^2(1 - \zeta_{\ast})_+ + \sigma^2 \zeta_{\ast} \int \exp(-2Tz) \, \mathrm{d} F_{\zeta_{\ast}}(z)}{\left( 1 - \zeta_{\ast} \int (1 - \exp(-Tz)) \, \mathrm{d}F_{\zeta_{\ast}}(z) \right)^{2}}. \label{eq:risk_gcv_asymp_mismatch}
    \end{align}
\end{restatable}

The proof of this asymptotic mismatch in \Cref{lem:risk-gcv-mismatch-gf} is provided in \Cref{sec:proof-lem:risk-gcv-mismatch-gf}.

\section{Supporting lemmas for the proof of \Cref{thm:gcv-inconsistency}}
\label{sec:supporting-lemmas-proof-thm:gcv-inconsistency}

\subsection{Connections between gradient descent and gradient flow}
\label{sec:gd-gf-equi}

We first show that under the conditions of \Cref{thm:gcv-inconsistency}, estimates obtained from gradient descent (GD) are in some sense asymptotically equivalent to that obtained from gradient flow (GF).

We next establish connections between GD and GF. 
This step is achieved by showing that the hat matrices as defined in \Cref{eq:Hat-A} when $k \to \infty$ and $k \delta \to T$ (for $\bH_k$) and when $t = T$ (for $\bH_t$) get closer under the matrix operator norm. 

Observe that the two matrices in \Cref{eq:Hat-A} share a common set of eigenvectors, and the eigenvalues are obtained by applying separate scalar transformations to the eigenvalues of $\hat\bSigma$. 
Hence, to show that $\bH_\K$ and $\bH_T^{\gf}$ are close in terms of operator norm, a natural first step is to show that the scalar transformations are uniformly close to each other. 
We characterize such closeness in \Cref{lemma:scalar-uniform-approximation} below. 

Let $g_{\delta, \K}(x) = \sum_{j = 0}^{\K - 1} \delta x(1 - \delta x)^{\K - j - 1}$ and $g_T(x) = 1 - \exp(-tx)$, which are exactly the scalar transformations of the hat matrices in \Cref{eq:Hat-A}. 
Our next lemma says that as $\K \to \infty$ and $\delta \to 0$ with $\K \delta \to T$, $g_{\delta, \K}$ uniformly approximates $g_T$ on a compact interval. 
\begin{lemma}
    [Scalar uniform approximation for GD and GF smoothing functions]
    \label{lemma:scalar-uniform-approximation}
    We assume $\K \to \infty$, $\delta \to 0$, and $\K \delta \to T$. Here, $T$ is a fixed positive constant. Then it holds that
    \begin{align*}
        \sup_{0 \leq x \leq \zeta_{\ast} + 2\sqrt{\zeta_{\ast}} + 2} \Big| g_{\delta, \K}(x) - g_T(x) \Big| \to 0, 
    \end{align*}
    where we recall that $\zeta_{\ast}$ is the limit of the aspect ratio. 
\end{lemma}

\begin{proof}
    For notational simplicity, we let $J_{\zeta_{\ast}} = [\, 0, \zeta_{\ast} + 2 \sqrt{\zeta_{\ast}} + 2\, ]$. 
    We will first show that 
    \begin{align}\label{eq:uniform1}
        \sup_{x \in J_{\zeta_{\ast}}} \Big| \K \log (1 - \delta x) + \K \delta x \Big| \to 0. 
    \end{align}
    To this end, we consider the first-order derivatives of the function inside the above absolute value sign with respect to $x$, which gives $-\delta \K / (1 - \delta x) + \K \delta $. 
    This quantity under the current conditions goes to zero uniformly for all $x \in J_{\zeta_{\ast}}$, thus proving \Cref{eq:uniform1}. 
    This further implies the following uniform convergence result:
    \begin{align*}
        \sup_{x \in J_{\zeta_{\ast}}, j + 1 \in [\K]} \Big| (\K - j - 1) \log (1 - \delta x) + (\K - j - 1)\delta \Big| \to 0.
    \end{align*}
    As a direct consequence of the above equation, we obtain
    \begin{align*}
        \sup_{x \in J_{\zeta_{\ast}}, j + 1 \in [\K]} \Big| (1 - \delta x)^{\K - j - 1} - \exp(-\delta(\K - j - 1)x) \Big| \to 0,
    \end{align*}
    which further gives 
    \begin{align*}
        \sup_{x \in J_{\zeta_{\ast}}} \Big| \sum_{j = 0}^{\K - 1} \delta x (1 - \delta x)^{\K - j - 1} - \sum_{j = 0}^{\K - 1} \delta x \exp(-\delta(\K - j - 1)x) \Big| \to 0 
    \end{align*}
    as $\sum_{j = 0}^{\K - 1} \delta x$ is uniformly upper bounded for all $x \in J_{\zeta_{\ast}}$. 

    Considering the derivative of an exponential function, it is not hard to see that
    \begin{align*}
       \sup_{j + 1 \in [\K]} \sup_{(\K - j - 1)\delta \leq z \leq (\K - j)\delta} \Big| \exp(- \delta (\K - j - 1) x) - \exp(-zx) \Big| \to 0. 
    \end{align*}
    Therefore, 
    \begin{align*}
        \sup_{x \in J_{\zeta_{\ast}}} \Big| \sum_{j = 0}^{\K - 1} \delta x \exp(-\delta(\K - j - 1)x) - \int_0^{\K\delta} x \exp(-z x) \, \mathrm{d} z \Big| \to 0. 
    \end{align*}
    Further, we note that 
    \[
    \sup_{x \in J_{\zeta_{\ast}}}\Big|\int_0^{\K \delta} x \exp(-zx) \, \mathrm{d} z - \int_0^{T} x \exp(-zx) \, \mathrm{d} z \Big| \to 0
    \]
    and 
    \[
    \int_{0}^T x \exp(-zx) \, \mathrm{d} z = 1 - \exp(-Tx).
    \]
    This completes the proof.
\end{proof}

We can apply \Cref{lemma:scalar-uniform-approximation} to establish several useful connections between GD and GF, which we state as \Cref{lem:gd-gf-equi-gcv,lem:gd-gf-equi-risk}.  
The proof of these two lemmas can be found in Appendices \ref{sec:proof-lem:gd-gf-equi-gcv} and \ref{sec:proof-lem:gd-gf-equi-risk}, respectively. 

\subsection{Useful concentration results}
\label{sec:useful-concen-results-gcv-inconsistency}

The following lemma provides the concentration of a linear form of a random vector with independent components.
It follows from a moment bound from Lemma 7.8 of \cite{erdos_yau_2017}, along with the Borel-Cantelli lemma and is adapted from Lemma S.8.5 of \cite{patil2022mitigating}.

\begin{lemma}
    [Concentration of linear form with independent components]
    \label{lem:concen-linform}
    Let $\bz_p \in \RR^{p}$ be a sequence of random vector with i.i.d.\ entries $z_{pi}$, $i = 1, \dots, p$ such that for each i, $\EE[z_{pi}] = 0$, $\EE[z_{pi}^2] = 1$, $\EE[|z_{pi}|^{4+\alpha}] \le M_\alpha$ for some $\alpha > 0$ and constant $M_\alpha < \infty$.
    Let $\ba_p \in \RR^{p}$ be a sequence of random vectors independent of $\bz_p$ such that $\limsup_{p} \| \ba_p \|_2^2 / p \le M_0$ almost surely for a constant $M_0 < \infty$.
    Then $\ba_p^\top \bz_p / p \to 0$ almost surely as $p \to \infty$.
\end{lemma}

The following lemma provides the concentration of a quadratic form of a random vector with independent components.
It follows from a moment bound from Lemma B.26 of \cite{bai_silverstein_2010}, along with the Borel-Cantelli lemma and is adapted from Lemma S.8.6 of \cite{patil2022mitigating}.

\begin{lemma}
    [Concentration of quadratic form with independent components]
    \label{lem:concen-quadform}
    Let $\bz_p \in \RR^{p}$ be a sequence of random vector with i.i.d.\ entries $z_{pi}$, $i = 1, \dots, p$ such that for each i, $\EE[z_{pi}] = 0$, $\EE[z_{pi}^2] = 1$, $\EE[|z_{pi}|^{4+\alpha}] \le M_\alpha$ for some $\alpha > 0$ and constant $M_\alpha < \infty$.
    Let $\bD_p \in \RR^{p \times p}$ be a sequence of random matrix such that $\limsup \| \bD_p \|_{\op} \le M_0$ almost surely as $p \to \infty$ for some constant $M_0 < \infty$.
    Then $\bz_p^\top \bD_p \bz_p / p - \tr[\bD_p] / p \to 0$
    almost surely as $p \to \infty$.
\end{lemma}

\section{Proof of \Cref{thm:gcv-inconsistency}}
\label{sec:proof-inconsistency}

\bigskip

\ThmGCVInconsistency*

\subsection{Proof schematic}
\label{sec:outline-thm:gcv-inconsistency}

A visual schematic for the proof of \Cref{thm:gcv-inconsistency} is provided in \Cref{fig:schematic-proof-thm:gcv-inconsistency}. 
The lemmas that appear in the figure shall be introduced in later parts of this section. 

\begin{figure}[!ht]
    \centering
    \begin{tikzpicture}[node distance=3cm]
        \node (gcv-inconsistency) [theorem] {\Cref{thm:gcv-inconsistency}} ;
        \node (risk-gcv-mismatch-gf) [lemma, below of=gcv-inconsistency, node distance=3cm] {\Cref{lem:risk-gcv-mismatch-gf}} ;
        \node (gd-gf-equi-gcv) [lemma, right of=risk-gcv-mismatch-gf, node distance=3cm] {\Cref{lem:gd-gf-equi-gcv}} ;
        \node (gd-gf-equi-risk) [lemma, right of=gd-gf-equi-gcv, node distance=3cm] {\Cref{lem:gd-gf-equi-risk}} ;
        \node (risk-asymptotics-gf) [lemma, below of=risk-gcv-mismatch-gf] {\Cref{lem:risk-asymptotics-gf}} ;
        \node (gcv-asymptotics-gf) [lemma, right of=risk-asymptotics-gf] {\Cref{lem:gcv-asymptotics-gf}} ;
        \node (scalar-uniform-approximation) [lemma, right of=gcv-asymptotics-gf] {\Cref{lemma:scalar-uniform-approximation}} ;
        \draw [arrow] (risk-gcv-mismatch-gf) -- (gcv-inconsistency) ;
        \draw [arrow] (gd-gf-equi-gcv) -- (gcv-inconsistency) ;
        \draw [arrow] (gd-gf-equi-risk) -- (gcv-inconsistency) ;
        \draw [arrow] (risk-asymptotics-gf) -- (risk-gcv-mismatch-gf) ;
        \draw [arrow] (gcv-asymptotics-gf) -- (risk-gcv-mismatch-gf) ;
        \draw [arrow] (scalar-uniform-approximation) -- (gd-gf-equi-risk) ;
        \draw [arrow] (scalar-uniform-approximation) -- (gd-gf-equi-gcv) ;
    \end{tikzpicture}
    \caption{Schematic for the proof of \Cref{thm:gcv-inconsistency} 
    }
    \label{fig:schematic-proof-thm:gcv-inconsistency}
\end{figure}

\subsection{Proof of \Cref{lem:gd-gf-equi-risk}}
\label{sec:proof-lem:gd-gf-equi-risk}

\bigskip

\LemGdGfEquiRisk*

\begin{proof}
Note that the prediction risks admit the following expressions:
\begin{align*}
    R(\hat\bbeta_\K) = \|\bbeta_0 - \hat\bbeta_\K \|_2^2 + \sigma^2
    \quad
    \text{and}
    \quad
    R(\hat\bbeta_T^{\gf}) = \|\bbeta_0 - \hat\bbeta_T^{\gf}\|_2^2 + \sigma^2. 
\end{align*}
We define $\bar{g}_{\delta, \K} (x) = \sum_{j = 0}^{\K - 1} \delta  (1 - \delta x)^{\K - j - 1}$ and $\bar{g}_T(x) = x^{-1}(1 - \exp(-Tx))$. We claim that
\begin{align}\label{eq:uniform-approximation-bar}
    \|x^{1/2}(\bar{g}_{\delta, \K}(x) - \bar{g}_T(x))\mathbbm{1}_{x \in J_{\zeta_{\ast}}}\|_{\infty} \to 0
\end{align}
under the asymptotics $\K \to \infty$, $\delta \to 0$, and $\K\delta \to T$. 
Proof for this claim is similar to that for \Cref{lemma:scalar-uniform-approximation}, and we skip it for the compactness of presentation.   

We note that 
\begin{align}\label{eq:diff-beta2}
     \hat\bbeta_\K - \hat\bbeta_T^{\gf} = \frac{1}{\sqrt{n}} \bV^{\top } \Big(\bar{g}_{\delta, \K} (\bLambda^{\top} \bLambda) - \bar{g}_T(\bLambda^{\top} \bLambda) \Big) \bLambda^{\top} \bU^{\top} \by, 
\end{align}
where we recall that $\bX / \sqrt{n} = \bV \bLambda \bU$ is the spectral decomposition. 
It is straightforward to obtain the following upper bound:
\begin{align*}
    \Big\| \Big(\bar{g}_{\delta, \K} (\bLambda^{\top} \bLambda) - \bar{g}_T(\bLambda^{\top} \bLambda) \Big) \bLambda^{\top} \Big\|_{\op} \leq \sup_{i \in [n]} \big| \lambda_i^{1/2} (\bar{g}_{\delta, \K}(\lambda_i) - \bar{g}_T(\lambda_i)) \big|. 
\end{align*}
Recall that $\max_{i \in [n]} \lambda_i \asto (1 + \sqrt{\zeta_{\ast}})^2$, hence the right-hand side of the above equation converges to zero almost surely (using \Cref{eq:uniform-approximation-bar}). 
By the law of large numbers, we obtain $\|\by\|_2 / \sqrt{n} \asto \E[y_1^2]^{1/2}$. Plugging these results into \Cref{eq:diff-beta2} gives $\|\hat\bbeta_\K - \hat\bbeta_T^{\gf}\|_2 \asto 0$ as $n, p \to \infty$. 
Furthermore, by \Cref{eq:beta-gf,eq:betakk} we have
\begin{align}\label{eq:beta-upper-bound-gfgd}
\begin{split}
    & \big\| \hat\bbeta_T^{\gf} \big\|_2 \leq \max_{i \in [n]} \lambda_i^{1/2} \cdot  \bar{g}_T\Big(\max_{i \in [n]} \lambda_i \Big)  \cdot \frac{1}{\sqrt{n}} \|\by\|_2, \\
    & \|\hat\bbeta_\K\|_2 \leq \max_{i \in [n]} \lambda_i^{1/2} \cdot  \bar{g}_{\delta, \K}\Big(\max_{i \in [n]} \lambda_i \Big) \cdot \frac{1}{\sqrt{n}} \|\by\|_2.
\end{split}
\end{align}
Standard analysis implies that $\sup_{x \in J_{\zeta_{\ast}}} \sqrt{x} \bar{g}_T(x) < \infty$ and $\limsup_{\K \to \infty, \delta \to 0}\sup_{x \in J_{\zeta_{\ast}}} \sqrt{x} \bar{g}_{\delta, \K}(x) < \infty$. 

Finally, combining all these results we have obtained, we conclude that 
\begin{align*}
    \left| \|\bbeta_0 - \hat\bbeta_\K\|_2^2 - \| \bbeta_0 - \hat\bbeta_T^{\gf}\|_2^2 \right| \asto 0
\end{align*}
as $n, p \to \infty$. 
This is equivalent to saying 
\begin{align*}%
    |R(\hat\bbeta_\K) - R(\hat\bbeta_T^{\gf})| \asto 0
\end{align*}
as $n, p \to \infty$. 
\end{proof}

\subsection{Proof of \Cref{lem:gd-gf-equi-gcv}}
\label{sec:proof-lem:gd-gf-equi-gcv}

\bigskip

\LemGdGfEquiGCV*

\begin{proof}
In the sequel, we will apply \Cref{lemma:scalar-uniform-approximation} to prove closeness between $\hat{R}^{\gcv}(\hat\bbeta_\K)$ and $\hat{R}^{\gcv}(\hat\bbeta_T^{\gf})$. 
This consists of proving the following three pairs of quantities are close:
\begin{itemize}
    \item[(1)]
    $(1 - \tr(\bH_\K) / n)^{-2}$ and $(1 - \tr(\bH_T^{\gf}) / n)^{-2}$.
    \item[(2)]
    $\hat\bbeta_\K^{\top} \hat\bSigma \hat\bbeta_\K$ and $(\hat\bbeta_T^{\gf})^{\top} \hat\bSigma \hat\bbeta_T^{\gf}$.
    \item[(3)]
    $\by^{\top} \bX \hat\bbeta_\K / n$ and $\by^{\top} \bX \hat\bbeta_T^{\gf} / n$.
\end{itemize} 
In what follows, we shall separately justify each of these closeness results.

\subsubsection*{Closeness result (1)}

We denote by $\{\lambda_i\}_{i \leq n}$ the top $n$ eigenvalues  of $\hat\bSigma$. From \citet[Theorem 5.8]{bai_silverstein_2010}, we know that $\max_{i \in [n]} \lambda_i \asto (1 + \sqrt{\zeta_{\ast}})^2$. 
Note that
\begin{align*}
    \frac{1}{n}\tr(\bH_\K) = \frac{1}{n} \sum_{i = 1}^n g_{\delta, \K}(\lambda_i) 
    \quad
    \text{and}
    \quad
    \frac{1}{n} \tr(\bH_T^{\gf}) = \frac{1}{n} \sum_{i = 1}^n g_T(\lambda_i). 
\end{align*}
Invoking \Cref{lemma:scalar-uniform-approximation}, we obtain that with probability one
\begin{align*}
    \limsup_{n,p \to \infty}\frac{1}{n} \left| \tr(\bH_\K) - \tr(\bH_T^{\gf}) \right| \leq  \sup_{0 \leq x \leq \zeta_{\ast} + 2\sqrt{\zeta_{\ast}} + 2} \Big| g_{\delta, \K}(x) - g_T(x) \Big|,
\end{align*}
which vanishes as $n,p \to \infty$. As a result, we derive that $|\tr(\bH_\K) - \tr(\bH_T^{\gf})| / n \asto 0$ as $n, p \to \infty$. 

Let $F_{\zeta_{\ast}}(s)$ denote the Marchenko-Pasture law as defined in \eqref{eq:MP-law-le1} and \eqref{eq:MP-law-le1}.
Standard results in random matrix theory \citep{bai_silverstein_2010} tell us that the empirical spectral distribution of $\hat\bSigma$ almost surely converges in distribution to $F_{\zeta_{\ast}}$. 
Note that $g_T$ is a bounded continuous function on $[0, \zeta_{\ast} + 2\sqrt{\zeta_{\ast}} + 2]$, thus
\begin{align*}
    \frac{1}{n} \sum_{i = 1}^n g_T(\lambda_i) \asto \int \big( 1 - \exp(-Tz) \big) \, \mathrm{d} F_{\zeta_{\ast}}(z),
\end{align*}
which one can verify is strictly smaller than $1$ for all $\zeta_{\ast} \in (0, \infty)$. 
Putting together the above analysis, we can deduce that both $(1 - \tr(\bH_\K) / n)^{-2}$ and $(1 - \tr(\bH_T^{\gf}) / n)^{-2}$ converge almost surely to one finite constant, hence concluding the proof for this part. 

\subsubsection*{Closeness result (2)}

We denote by $\bX / \sqrt{n} = \bU \bLambda \bV$ the singular value decomposition of $\bX / \sqrt{n}$, where $\bU \in \R^{n \times n}$ and $\bV \in \R^{p \times p}$ are orthogonal matrices. 
Combining \Cref{eq:beta-gf,eq:betakk}, we arrive at the following equation: 
\begin{align}\label{eq:close2}
     \hat\bbeta_\K^{\top} \hat\bSigma \hat\bbeta_\K -  (\hat\bbeta_T^{\gf})^{\top} \hat\bSigma \hat\bbeta_T^{\gf}  = \by^{\top} \bU^{\top} \cdot \left\{ g_{\delta, \K} (\bLambda \bLambda^{\top})^2 - g_T(\bLambda \bLambda^{\top})^2\right\} \cdot \bU \by / n. 
\end{align}
By the strong law of large numbers, we have $\|\by\|_2^2 / n \asto \E[y_1^2]$. 
By \Cref{lemma:scalar-uniform-approximation} and the fact that $\max_{i \in [n]} \lambda_i \asto (1 + \sqrt{\zeta_{\ast}})^2$, we conclude that
\begin{align*}
    \big\|  g_{\delta, \K} (\bLambda \bLambda^{\top})^2 - g_T(\bLambda \bLambda^{\top})^2 \big\|_{\op} \asto 0.
\end{align*}
Plugging these arguments into \Cref{eq:close2}, we obtain 
\begin{align*}
    \left| \hat\bbeta_\K^{\top} \hat\bSigma \hat\bbeta_\K -  (\hat\bbeta_T^{\gf})^{\top} \hat\bSigma \hat\bbeta_T^{\gf} \right| \asto 0,
\end{align*}
which concludes the proof of closeness result (2). 

\subsubsection*{Closeness result (3)}

Finally, we show one more closeness result (3). 
We note that
\begin{align*}
    \frac{1}{n} \big( \by^{\top} \bX \hat\bbeta_\K - \by^{\top} \bX \hat\bbeta_T^{\gf} \big) = \by^{\top} \bU^{\top} \cdot \left\{ g_{\delta, \K}(\bLambda \bLambda^{\top}) - g_T(\bLambda \bLambda^{\top}) \right\} \cdot \bU \by / n,
\end{align*}
which by the same argument as that we used to derive result (2) almost surely converges to zero as $n, p \to \infty$. 

Putting together (1), (2), and (3), we conclude the proof of the lemma.
\end{proof}

\subsection{Proof of \Cref{lem:risk-asymptotics-gf}}
\label{sec:proof-lem:risk-asymptotics-gf}

\bigskip

\LemRiskAsymptoticsGf*

\begin{proof}
Applying \Cref{eq:beta-gf} and the risk decomposition formula, we obtain
\begin{align*}
    R(\hat\bbeta_T^{\gf}) = & \bbeta_0^{\top} \exp(-2T \hat\bSigma ) \bbeta_0 - \frac{2}{n}\bbeta_0^{\top} \exp(-T \hat\bSigma) \hat\bSigma^{\dagger} (\id_p - \exp(-T \hat\bSigma)) \bX^{\top} \beps \\
    & + \frac{1}{n^2} \beps^{\top} \bX (\id_p - \exp(-T \hat\bSigma))(\hat\bSigma^{\dagger})^2 (\id_p - \exp(-T \hat\bSigma)) \bX^{\top} \beps + \sigma^2. 
\end{align*}
Note that 
\begin{align*}
    \frac{2}{\sqrt{n}} \big\|\bbeta_0^{\top} \exp(-T \hat\bSigma) \hat\bSigma^{\dagger} (\id_p - \exp(-T \hat\bSigma)) \bX^{\top}  \big\|_2 \leq 2 \|\bbeta_0\|_2 \cdot \sup_{i \in [n]} \frac{\exp(-T \lambda_i)(1 - \exp(-T\lambda_i))}{\lambda_i^{1/2}},
\end{align*}
where it is understood that $\lambda^{-1/2} e^{-T\lambda}(1 - e^{-T\lambda}) \mid_{\lambda = 0} = 0$. 
Recall that $\max_i \lambda_i \asto (1 + \sqrt{\zeta_{\ast}})^2$ and $\|\bbeta_0\|_2^2 \to r^2$. 
Hence, there exists a constant $M_0$ such that almost surely
\[
    \limsup_{n,p \to \infty} \|\bbeta_0^{\top} \exp(-T \hat\bSigma) \hat\bSigma^{\dagger} (\id_p - \exp(-T \hat\bSigma)) \bX^{\top}  \|_2^2 / n \leq M_0.
\]
Therefore, we can apply \Cref{lem:concen-linform} and deduce that
\begin{align}\label{eq:ccc1}
    \frac{2}{n}\bbeta_0^{\top} \exp(-T \hat\bSigma) \hat\bSigma^{\dagger} (\id_p - \exp(-T \hat\bSigma)) \bX^{\top} \beps \asto 0. 
\end{align}
By \Cref{lem:concen-quadform}, we have
\begin{align*}
    & \left| n^{-2} \beps^{\top} \bX (\id_p - \exp(-T \hat\bSigma))(\hat\bSigma^{\dagger})^2 (\id_p - \exp(-T \hat\bSigma)) \bX^{\top} \beps \right. \\
    & \quad \left. - n^{-2} \sigma^2 \tr( \bX (\id_p - \exp(-T \hat\bSigma))(\hat\bSigma^{\dagger})^2 (\id_p - \exp(-T \hat\bSigma)) \bX^{\top}) \right| \asto 0. 
\end{align*}
Standard random matrix theory result implies that almost surely the empirical spectral distribution of $\hat\bSigma$ converges in distribution to $F_{\zeta_{\ast}}$, which is the Marchenko-Pastur law defined in \eqref{eq:MP-law-le1} and \eqref{eq:MP-law-gt1}. 
Furthermore, $\|\hat\bSigma\|_{\op} \asto (1 + \sqrt{\zeta_{\ast}})^2$. 
Therefore, we conclude that 
\begin{align}
    \label{eq:ccc2}
     & n^{-2} \sigma^2 \tr( \bX (\id_p - \exp(-T \hat\bSigma))(\hat\bSigma^{\dagger})^2 (\id_p - \exp(-T \hat\Sigma)) \bX^{\top}) \notag \\ & \asto \zeta_{\ast}\sigma^2\int z^{-1}(1 - \exp(-Tz))^2 \, \mathrm{d} F_{\zeta_{\ast}}(z). %
\end{align}
Finally, we study the limit of $\bbeta_0^{\top} \exp(-2T \hat\bSigma ) \bbeta_0$. 
Let $\bOmega \in \RR^{p \times p}$ be a uniformly distributed orthogonal matrix that is independent of anything else. 
Since by assumption $\|\bbeta_0\|_2 \to r$, we can then couple $\bOmega \bbeta_0$ with $\bg \sim \cN(\mathbf{0}, \mathbf{I}_p)$, so that (1) $\bg$ is independent of $\hat\bSigma$, and (2) $\|\bOmega \bbeta_0 - r \bg / \sqrt{p}\|_2 \asto 0$. 
Note that all eigenvalues of $\exp(-2T \hat\bSigma)$ are between $0$ and $1$, hence
\begin{align*}
    \left|\bbeta_0^{\top} \exp(-2T \hat\bSigma ) \bbeta_0 - \frac{r^2}{p} \bg^{\top} \exp(-2T \hat\bSigma) \bg \right| \asto 0. 
\end{align*}
Leveraging \Cref{lem:concen-quadform}, we obtain 
\[
r^2  \bg^{\top} \exp(-2T \hat\bSigma) \bg / p \asto r^2 \int \exp(-2Tz) \, \mathrm{d} F_{\zeta_{\ast}}(z).
\]
Combining this with \eqref{eq:ccc1} and \eqref{eq:ccc2}, we finish the proof. 
\end{proof}

\subsection{Proof of \Cref{lem:gcv-asymptotics-gf}}
\label{sec:proof-lem:gcv-asymptotics-gf}

\bigskip

\LemGCVAsymptoticsGf*

\begin{proof}
We separately discuss the numerator and the denominator. We start with the denominator. 
Recall that the empirical spectral distribution of $\hat\bSigma$ almost surely converges to $F_{\zeta_{\ast}}$ and $\|\hat\bSigma\|_{\op} \asto (1 + \sqrt{\zeta_{\ast}})^2$. 
Hence, 
\begin{align}\label{eq:ccc2.5}
    (1 - \tr(\bH_T^{\gf}) / n)^{-2} \asto \left( 1 - \zeta_{\ast} \int (1 - \exp(-Tz)) \, \mathrm{d}F_{\zeta_{\ast}}(z) \right)^{-2}. 
\end{align}
Next, we consider the numerator. 
Straightforward computation implies that
\begin{align*}
    \frac{1}{n}\|\by - \bX \hat\bbeta_T^{\gf}\|_2^2 = & \bbeta_0^{\top} \exp(-T \hat\bSigma) \hat\bSigma \exp(-T \hat\bSigma) \bbeta_0 + \frac{1}{n}\big\| \big(\id_n - \frac{1}{n} \bX \hat\bSigma^{\dagger} (\id_p - \exp(-T \hat\bSigma)) \bX^{\top} \big) \beps \big\|_2^2 \\
    & + \frac{2}{n} \big\langle \bbeta_0,  \exp(-T \hat\bSigma) \bX^{\top}\big(\id_n - \frac{1}{n} \bX \hat\bSigma^{\dagger} (\id_p - \exp(-T \hat\bSigma)) \bX^{\top} \big) \beps \big\rangle. 
\end{align*}
Since $\|\hat\bSigma\|_{\op} \asto (1 + \sqrt{\zeta_{\ast}})^2$, we then obtain almost surely
\begin{align*}
     \limsup_{n,p \to \infty}\big\|\exp(-T \hat\bSigma) \bX^{\top}\big(\id_n - \frac{1}{n} \bX \hat\bSigma^{\dagger} (\id_p - \exp(-T \hat\bSigma)) \bX^{\top} \big)\big\|_{\op} \leq G(\zeta_{\ast}) < \infty,
\end{align*}
where $G(\zeta_{\ast})$ is a function of $\zeta_{\ast}$. 
Therefore, by \Cref{lem:concen-linform}, we obtain 
\begin{align}\label{eq:ccc3}
    \frac{2}{n} \big\langle \bbeta_0,  \exp(-T \hat\bSigma) \bX^{\top}\big(\id_n - \frac{1}{n} \bX \hat\bSigma^{\dagger} (\id_p - \exp(-T \hat\bSigma)) \bX^{\top} \big) \beps \big\rangle \asto 0.
\end{align}
Using the same argument that we used to compute the limiting expression of $\bbeta_0^{\top} \exp(-T \hat\bSigma) \bbeta_0$, we conclude that
\begin{align}\label{eq:ccc4}
    \bbeta_0^{\top} \exp(-T \hat\bSigma) \hat\bSigma \exp(-T \hat\bSigma) \bbeta_0 \asto r^2 \int z \exp(-2Tz) \, \mathrm{d} F_{\zeta_{\ast}}(z). 
\end{align}
In addition, by \Cref{lem:concen-quadform}, we have
\begin{align}\label{eq:ccc5}
    \frac{1}{n}\big\| \big(\id_n - \frac{1}{n} \bX \hat\bSigma^{\dagger} (\id_p - \exp(-T \hat\bSigma)) \bX^{\top} \big) \beps \big\|_2^2 \asto \sigma^2(1 - \zeta_{\ast}) + \sigma^2 \zeta_{\ast} \int \exp(-2Tz) \, \mathrm{d} F_{\zeta_{\ast}}(z). 
\end{align}

To see the limit in \eqref{eq:ccc5}, we expand the matrix of the quadratic form as follows:
\begin{align*}
    &(\bI_n - \frac{1}{n} \bX \bhSigma^{\dagger} (\bI_p - \exp(-T \bhSigma)) \bX^\top)
    (\bI_n - \frac{1}{n} \bX \bhSigma^{\dagger} (\bI_p - \exp(-T \bhSigma)) \bX^\top) \\
    &=
    (\bI_n - \frac{1}{n} \bX \bhSigma^{\dagger} (\bI_p - \exp(-T \bhSigma)) \bX^\top)
    - \frac{1}{n} \bX \bhSigma^{\dagger} (\bI_p - \exp(-T \bhSigma) \bX^\top)
    (\bI_n - \frac{1}{n} \bX \bhSigma^{\dagger} (\bI_p - \exp(-T \bhSigma)) \bX^\top) \\
    &=  (\bI_n - \frac{1}{n} \bX \bhSigma^{\dagger} (\bI_p - \exp(-T \bhSigma)) \bX^\top)
    - \frac{1}{n} \bX \bhSigma^{\dagger} (\bI_p - \exp(-T \bhSigma))
    (\bI_p - \bhSigma \bhSigma^{\dagger} (\bI_p - \exp(-T \bhSigma))) \bX^\top.
\end{align*}
The normalized (by $n$) trace of the matrix above is
\begin{align*}
    &1 - \zeta_{\ast} \tr[(\bI_p - \exp(-T \bhSigma))] / p - \zeta_{\ast} \tr[(\bI_p - \exp(-T \bhSigma)) \exp(-T \bhSigma)] / p \\
    &= 1 - \zeta_{\ast} + \zeta_{\ast} \tr[\exp(-2 T \bhSigma)] / p.
\end{align*}
In the above simplification, we used the fact that 
\[
    \bhSigma \bhSigma^{\dagger} (\bI_p - \exp(-T \bhSigma))
    = (\bI_p - \exp(-T \bhSigma)).
\]
This fact follows because $\bhSigma^{\dagger} \bhSigma$ is the projection onto the row space of $\bX$.
But the image of $\bI_p - \exp(-t \bhSigma)$ is already in the row space.
The limit for \eqref{eq:ccc5} therefore is
\[
    \sigma^2 (1 - \zeta_{\ast}) + \sigma^2 \zeta_{\ast} \int \exp(-2Tz) \, \mathrm{d}F_{\zeta_{\ast}}(z).
\]
We can do quick sanity checks for this limit:
\begin{itemize}[leftmargin=7mm]
    \item When $T = 0$, we should get $\sigma^2$ irrespective of $\zeta_{\ast}$ because we start with a null model.
    \item When $T = \infty$, we should get the training error of the least squares or ridgeless estimator due to noise.
    There are two cases:
    \begin{itemize}[leftmargin=7mm]
        \item When $\zeta_{\ast} < 1$: this is the variance component of the residual of least squares.
        This should be $\sigma^2 (1 - \zeta_{\ast})$.
        \item When $\zeta_{\ast} > 1$: this is the variance component of the training error of the ridgeless interpolator, which should be zero.
    \end{itemize}
\end{itemize}
To check the last point, it is worth noting that
\[
    \lim_{T \to \infty}
    \int \exp(-2 T z) \, \mathrm{d}F_{\zeta_{\ast}}(z)
    = 
    \begin{dcases}
        0 & \zeta_{\ast} < 1 \\
        1 - \frac{1}{\zeta_{\ast}} & \zeta_{\ast} > 1.
    \end{dcases}
\]
Now, \Cref{eq:ccc2.5,eq:ccc3,eq:ccc4,eq:ccc5} together imply the stated result.
\end{proof}

\subsection{Proof of \Cref{lem:risk-gcv-mismatch-gf}}
\label{sec:proof-lem:risk-gcv-mismatch-gf}

\bigskip

\LemRiskGCVMismatchGf*

\begin{proof}
Recall the asymptotics of the risk from \Cref{lem:risk-asymptotics-gf}:
\begin{align}
    &R(\hat\bbeta_T^{\gf}) \\
    &\asto r^2 \int \exp(-2Tz) \, \mathrm{d} F_{\zeta_{\ast}}(z) + \zeta_{\ast}\sigma^2\int z^{-1}(1 - \exp(-Tz))^2 \, \mathrm{d} F_{\zeta_{\ast}}(z) + \sigma^2 \nonumber \\
    &= r^2 \left\{ \int \exp(-2Tz) \, \mathrm{d} F_{\zeta_{\ast}}(z) \right\}
    + \sigma^2 \left\{ 1 + \zeta_{\ast} \int z^{-1}(1 - \exp(-Tz))^2 \, \mathrm{d} F_{\zeta_{\ast}}(z) \right\}. \label{eq:risk-asymptotics-bias-var-split}
\end{align}
Recall also the asymptotics of GCV from \Cref{lem:gcv-asymptotics-gf}:
\begin{align}
    &\hR^{\gcv}(\hat \bbeta_k) \\
    &\asto \ddfrac{ r^2 \int z \exp(-2Tz) \, \mathrm{d} F_{\zeta_{\ast}}(z) + \sigma^2(1 - \zeta_{\ast}) + \sigma^2 \zeta_{\ast} \int \exp(-2Tz) \, \mathrm{d} F_{\zeta_{\ast}}(z)}{\left( 1 - \zeta_{\ast} \int (1 - \exp(-Tz)) \, \mathrm{d}F_{\zeta_{\ast}}(z) \right)^{2}} \nonumber \\
    &= 
    r^2
    \ddfrac{ \int z \exp(-2Tz) \, \mathrm{d} F_{\zeta_{\ast}}(z)}{\left( 1 - \zeta_{\ast} \int (1 - \exp(-Tz)) \, \mathrm{d}F_{\zeta_{\ast}}(z) \right)^{2}}
    +
    \sigma^2 
    \ddfrac{ (1 - \zeta_{\ast}) + \zeta_{\ast} \int \exp(-2Tz) \, \mathrm{d} F_{\zeta_{\ast}}(z)}{\left( 1 - \zeta_{\ast} \int (1 - \exp(-Tz)) \, \mathrm{d}F_{\zeta_{\ast}}(z) \right)^{2}}. \label{eq:gcv-asymptotics-bias-var-split}
\end{align}
Here, $F_{\zeta_{\ast}}$ is the Marchenko-Pasture law, as defined in \eqref{eq:MP-law-le1} and \eqref{eq:MP-law-gt1}.
Observe that both functions \eqref{eq:risk-asymptotics-bias-var-split} and \eqref{eq:gcv-asymptotics-bias-var-split} are analytic (i.e., they can be represented by a convergent power series in a neighborhood of every point in their domain).
From the identity theorem for analytic functions (see, e.g., Chapter 1 of \cite{krantz2002primer}), it suffices to show that the functions do not agree in a neighborhood of a point inside the domain.
We will do this in the neighborhood of $t = 0$.
The function value and the derivatives match, but the second derivatives mismatch.
This is shown in \Cref{sec:signal-limits-mismatch-theoretical,sec:noise-limits-mismatch-theoretical}.
This supplies us with the desired function disagreement and concludes the proof.
\end{proof}

A couple of remarks on the proof of \Cref{lem:risk-gcv-mismatch-gf} follow.

\begin{itemize}[leftmargin=7mm]
    \item 
    Observe that both the risk and the GCV asymptotics in \eqref{eq:risk-asymptotics-bias-var-split} and \eqref{eq:gcv-asymptotics-bias-var-split} split into bias or bias-like and variance or variance-like components, respectively.
    The bias or bias-like component is scaled by the signal energy, and the variance or variance-like component is scaled by the noise energy.
    We can also show that except for a set of Lebesgue measure $0$, we have
    \begin{align}
        \int \exp(-2Ts) \, \mathrm{d}F_{\zeta_{\ast}}(s)
        &\neq
        \ddfrac{\int s \exp(-2Ts) \, \mathrm{d}F_{\zeta_{\ast}}(s)}{\left(1 - \zeta_{\ast} \int (1 - \exp(-Ts)) \, \mathrm{d}F_{\zeta_{\ast}}(s) \right)^2}, \label{eq:sig-mismatch} \\
        1
        +
        \zeta_{\ast}
        \int \frac{(1 - \exp(-Ts))^2}{s} \, \mathrm{d}F_{\zeta_{\ast}}(s) 
        &\neq
        \ddfrac{(1 - \zeta_{\ast})
        + \zeta_{\ast} \int \exp(-2 T s) \, \mathrm{d}F_{\zeta_{\ast}}(s)}{\left(1 - \zeta_{\ast} \int (1 - \exp(-Ts)) \, \mathrm{d}F_{\zeta_{\ast}}(s) \right)^2}. \label{eq:var-mismatch}
    \end{align}
    In the following, we will refer to \eqref{eq:sig-mismatch} as the signal component mismatch and \eqref{eq:var-mismatch} as the noise component mismatch.
    The functions on both sides of \eqref{eq:sig-mismatch} and \eqref{eq:var-mismatch} are again analytic in $T$.
    The mismatch of the second derivatives for the sum above in fact is a consequence of mismatches for the individual signal and noise component.
    This is shown in \Cref{sec:signal-limits-mismatch-theoretical,sec:noise-limits-mismatch-theoretical}.

    \item
    We can also numerically verify the mismatches since the Marchenko-Pasture law has an explicit density as indicated in \eqref{eq:MP-law-le1} and \eqref{eq:MP-law-gt1}.
    We can simply evaluate both the signal and noise component expressions and observe that the functions are indeed not equal.
    We numerically illustrate in \Cref{sec:signal-limits-mismatch-theoretical,sec:noise-limits-mismatch-theoretical} that the functions on the left-hand side and the right-hand side of \eqref{eq:sig-mismatch} and \eqref{eq:var-mismatch} are not equal for the entire range of $T$ plotted (except for when $T = 0$).
\end{itemize}

\subsubsection{Combined sum mismatch}
\label{sec:combined_sum_mismatch}

Our goal is to show that the two limiting functions (of $T$) in \eqref{eq:risk-asymptotics-bias-var-split} and \eqref{eq:gcv-asymptotics-bias-var-split} differ on a neighborhood of $T = 0$.
Since the common denominator in the two terms in \Cref{eq:gcv-asymptotics-bias-var-split} are away from $0$, it suffices to show that in the neighborhood around $T = 0$, the following function is not identically zero:
\begin{align}
    \cD(T)
    =
    &r^2
    \left\{
        \int \exp(-2Tz) \, \mathrm{d} F_{\zeta_{\ast}}(z)
        \left( 1 - \zeta_{\ast} \int (1 - \exp(-Tz)) \, \mathrm{d}F_{\zeta_{\ast}}(z) \right)^{2}
        -
        \int z \exp(-2Tz) \, \mathrm{d} F_{\zeta_{\ast}}(z)
    \right\} \notag \\
    & 
    + \sigma^2
    \left\{
        \left(
        1 + \zeta_{\ast} \int z^{-1}(1 - \exp(-Tz))^2 \, \mathrm{d} F_{\zeta_{\ast}}(z) 
        \right)
        \left( 1 - \zeta_{\ast} \int (1 - \exp(-Tz)) \, \mathrm{d}F_{\zeta_{\ast}}(z) \right)^{2} \right. \notag \\
    & \qquad \qquad
        \left. - (1 - \zeta_{\ast}) - \zeta_{\ast} \int \exp(-2Tz) \, \mathrm{d} F_{\zeta_{\ast}}(z) \right\}. \label{eq:diff_risk_gcv_asymp}
\end{align}
As argued in the proof of \Cref{lem:risk-gcv-mismatch-gf}, the function $\cD$ is analytic and it suffices to examine the Taylor coefficients.
Both $\cD(0)$ and $\cD'(0)$ are $0$ but it turns out that $\cD''(0) \neq 0$ for $\zeta_{\ast} > 0$.
Thus, our subsequent goal will be to compute $\cD''(T)$ and evaluate it at $T = 0$.
We will make use of double derivative calculations in \Cref{sec:signal-limits-mismatch-theoretical,sec:noise-limits-mismatch-theoretical} for this purpose, as summarized below.

\begin{claim}
    [Second derivatives mismatch for combined sum]
    \label{clm:double-derivatives-sum-mismatch}
    For the function $\cD$ as defined in \eqref{eq:diff_risk_gcv_asymp}, we have $\cD''(T) = - 2 \zeta_{\ast} (2 r^2 + \sigma^2)$.
    Thus, when $\zeta_{\ast} > 0$ and either $r^2 > 0$ or $\sigma^2 > 0$, we have $\cD''(0) \neq 0$.
\end{claim}
\begin{proof}
    The calculation follows from \Cref{clm:double-derivatives-signal-components-mismatch,clm:double-derivatives-noise-components-mismatch}.
    Specifically, using the notation defined in these claims, we have
    \begin{equation}
        \label{eq:diff_risk_gcv_asymp_doublederiv}
        \cD''(T)
        = r^2 (\cB''_\ell(T) - \cB''_r(T)) - \sigma^2 (\cV''_\ell(T) - \cV''_r(T)).
    \end{equation}
    Evaluating \eqref{eq:diff_risk_gcv_asymp_doublederiv} at $T = 0$ yields
    \begin{equation}
        \label{eq:diff_risk_gcv_asymp_doublederiv_T0}
        \cD''(0)
        = r^2 (4 + 12 \zeta_{\ast} + 4 \zeta_{\ast}^2) - r^2 (4 + 14 \zeta_{\ast} + 4 \zeta_{\ast}^2)
        - \sigma^2 (4 \zeta_{\ast}^2) + \sigma^2 (4 \zeta_{\ast} + 4 \zeta_{\ast}^2).
    \end{equation}
    Simplifying \eqref{eq:diff_risk_gcv_asymp_doublederiv_T0}, we obtain the desired conclusion.
\end{proof}

Admittedly, the calculations in \Cref{clm:double-derivatives-sum-mismatch} are tedious and do not shed much light on the ``why''.
We also provide numerical illustrations in \Cref{fig:gf_limit_mismatch_in_t_sum,fig:gf_limit_mismatch_surface_sum} to help visualize the mismatch for the choice of $(r^2, \sigma^2) = (1, 1)$.
One can also numerically check that the mismatch gets worse as either $r^2$ or $\sigma^2$ increases.
In \Cref{sec:mismatch_illustration_varying_snr}, we illustrate this behavior for increasing values of $r^2$ and $\sigma^2$.
While the illustrations may still not illuminate the reason for the mismatch any
more than the theoretical calculations just presented, at least they can
visually convince the reader of the mismatch. In the figures, we denote $\SNR = r^2 / \sigma^2$.

\begin{figure*}[!ht]
    \centering
  \includegraphics[width=0.495\textwidth]{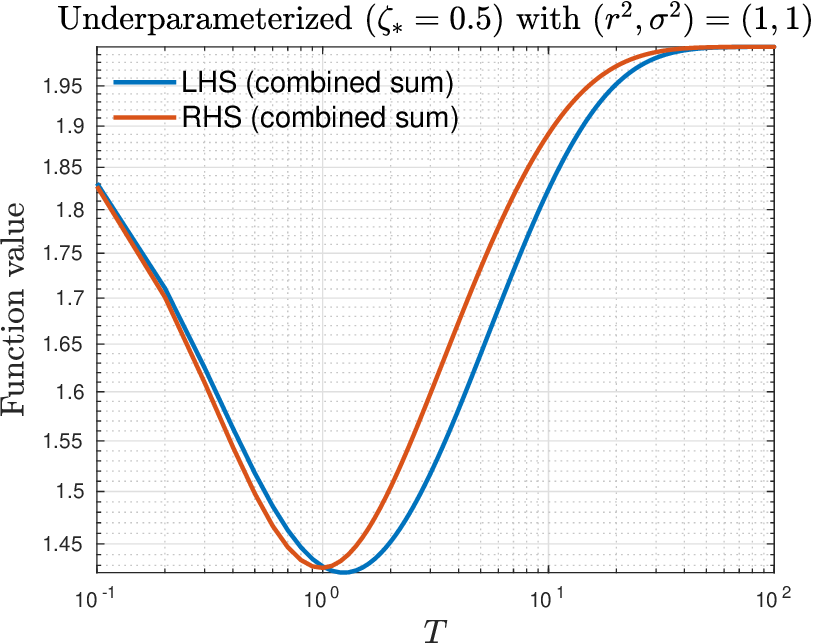}
  \includegraphics[width=0.485\textwidth]{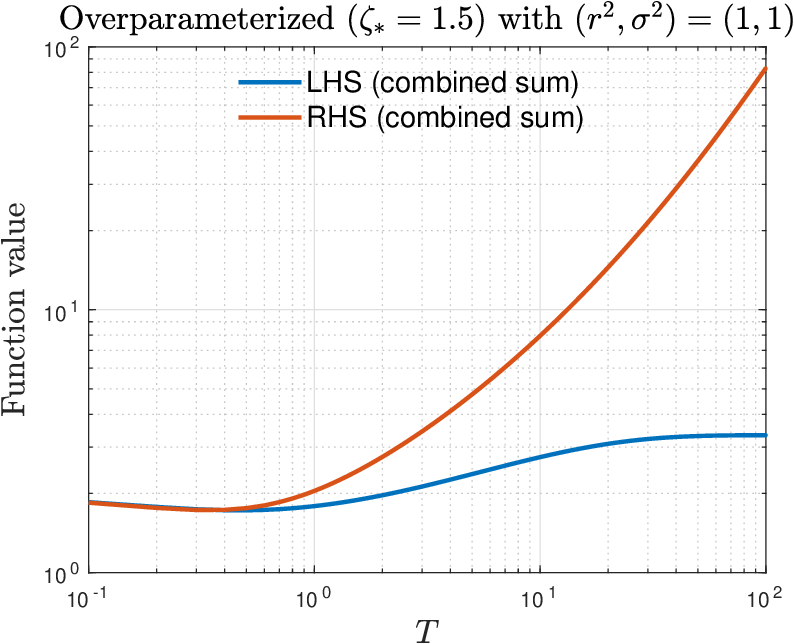}
  \caption{Comparison of the LHS and RHS in \eqref{eq:risk_gcv_asymp_mismatch} (combined sum) for the underparameterized (\emph{left}) and overparameterized (\emph{right}) regimes with $\SNR=1$.}
  \label{fig:gf_limit_mismatch_in_t_sum}
\end{figure*}

\medskip

\begin{figure*}[!ht]
    \centering
    \includegraphics[width=0.8\textwidth]{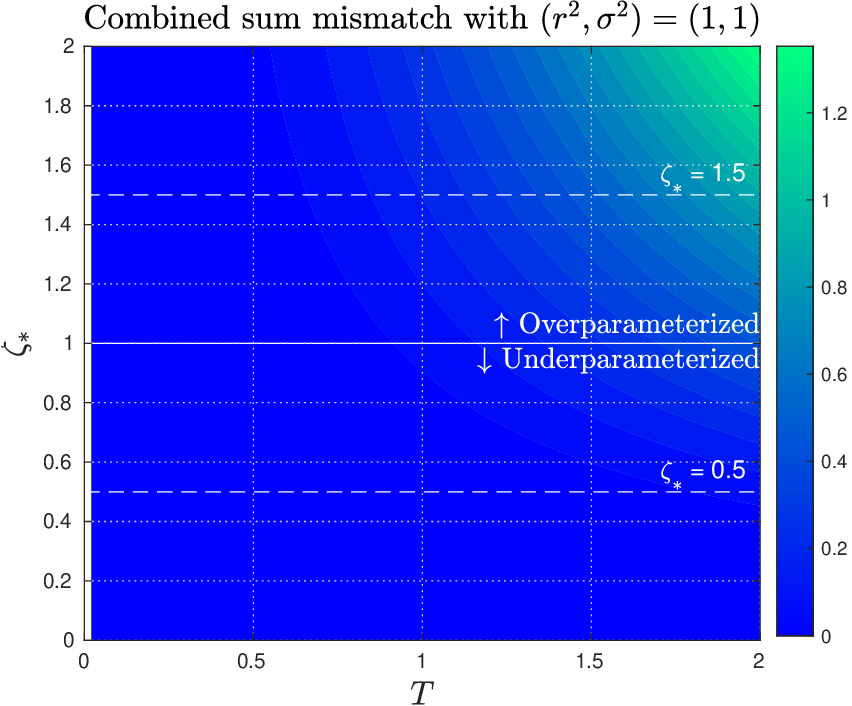}
    \caption{Contour plot of the absolute value of the difference between LHS and RHS of \eqref{eq:risk_gcv_asymp_mismatch} (combined sum) with $\SNR = 1$.
    We observe the mismatch in the north-west corner.
    In this example, the mismatch is predominantly due to the noise component.
    See \Cref{sec:mismatch_illustration_varying_snr} for further illustrations where we vary signal energy and noise energy and inspect how the mismatch changes.}
    \label{fig:gf_limit_mismatch_surface_sum}
\end{figure*}

\subsubsection{Signal component mismatch}
\label{sec:signal-limits-mismatch-theoretical}

\begin{claim}
    [Second derivatives mismatch for signal component]
    \label{clm:double-derivatives-signal-components-mismatch}
    Let $F_{\zeta_{\ast}}$ be the Marchenko-Pasture law as defined in \eqref{eq:MP-law-le1} and \eqref{eq:MP-law-gt1}.
    Let $\cB_\ell$ and $\cB_r$ be two functions defined as follows:
    \begin{align*}
        \cB_\ell(T)
        &=
        \left(1 - \zeta_{\ast} \int (1 - \exp(-Ts)) \, \mathrm{d}F_{\zeta_{\ast}}(s) \right)^2
        \int \exp(-2Ts) \, \mathrm{d}F_{\zeta_{\ast}}(s), \\
        \cB_r(T)
        &= 
        \int s \exp(-2Ts) \, \mathrm{d}F_{\zeta_{\ast}}(s).
    \end{align*}
    We have $\cB''_\ell(0) = 4 + 12 \zeta_{\ast} + 4 \zeta_{\ast}^2$ and $\cB''_r(0) = 4 + 14 \zeta_{\ast} + 4 \zeta_{\ast}^2$, and hence $\bB''_\ell(0) \neq \bB''_r(0)$.
\end{claim}
\begin{proof}
    For ease of notation, define the functions $w$, $v$, and $u$ as follows:
    \begin{align*}
        w(T) 
        &= 
        \left(1 - \zeta_{\ast} \int (1 - \exp(-Ts)) \, \mathrm{d}F_{\zeta_{\ast}}(s) \right)^2, \\
        v(T) 
        &= 
        \int \exp(-2Ts) \, \mathrm{d}F_{\zeta_{\ast}}(s), \\
        u(T) 
        &= 
        \int s \exp(-2Ts) \, \mathrm{d}F_{\zeta_{\ast}}(s).
    \end{align*}
    Then we have $\cB_\ell(T) = w(T) v(T)$ and $\cB_r(T) = u(T)$.
    The first-order derivatives are $\cB'_\ell(T) = w'(T) v(T) + w(T) v'(T)$ and $\cB'_r(T) = u'(T)$.
    The second-order derivatives are $\cB''_\ell(T) = w''(T) + 2 w'(T) v'(T) + v''(T)$ and $\cB''_r(T) = u''(T)$.
    From \Cref{clm:double-derivatives-signal-components}, we obtain
    \begin{align*}
        \cB''_\ell(0)
        = 2 \zeta_{\ast} (1 + 2 \zeta_{\ast}) + 8 \zeta_{\ast} + 4 (1 + \zeta_{\ast})
        = 4 + 14 \zeta_{\ast} + 4 \zeta_{\ast}^2.
    \end{align*}
    On the other hand, from \Cref{clm:double-derivatives-signal-components} again, we have
    \begin{align*}
        \cB'_r(0)
        = 4 (1 + 3 \zeta_{\ast} + \zeta_{\ast}^2)
        = 4 + 12 \zeta_{\ast} + 4 \zeta_{\ast}^2.
    \end{align*}
    Thus, for any $\zeta_{\ast} > 0$, we have that $\cB''_\ell(0) \neq \cB''_r(0)$, as desired.
\end{proof}

\begin{claim}
    [Second derivatives of various parts signal component]
    \label{clm:double-derivatives-signal-components}
    Let $w$, $v$, and $u$ be functions defined in the proof of \Cref{clm:double-derivatives-signal-components-mismatch}.
    Then the following claims hold.
    \begin{itemize}
        \item 
        $w(0) = 1$, $w'(0) = - 2 \zeta_{\ast}$, and $w''(0) = 2 \zeta_{\ast} (1 + 2 \zeta_{\ast})$.
        \item
        $v(0) = 1$, $v'(0) = - 2$, and $v''(0) = 4 (1 + \zeta_{\ast})$.
        \item
        $u(0) = 1$, $u'(0) = -2 (1 + \zeta_{\ast})$, and $u''(0) = 4 (1 + 3 \zeta_{\ast} + \zeta_{\ast}^2)$.
    \end{itemize}
\end{claim}
\begin{proof}
The functional evaluations are straightforward. 
We will split the first- and second-order derivative calculations into separate parts below.
For $k \ge 0$, let $M_k = \int s^k \, \mathrm{d} F_{\zeta_{\ast}}(s)$ be the $k$-th moment of the Marchenko-Pastur law.

\paragraph{Part 1.\hspace{-0.5em}}
Denote the inner integral by
\[
    I(T) = \zeta_{\ast} \int (1 - \exp(-Ts)) \, \mathrm{d}F_{\zeta_{\ast}}(s).
\]
Then, $w(T) = (1 - I(T))^2$.
The first derivative of $w(T)$ is
\[
    w'(T) = - 2 (1 - I(T)) \cdot I'(T)
    \quad
    \text{with}
    \quad
    I'(T) = \zeta_{\ast} \int s\exp(-Ts) \, \mathrm{d}F_{\zeta_{\ast}}(s).
\]
The second derivative of $w(T)$ is
\[
    w''(T) = 2 (I'(T))^2 - 2 (1 - I(T)) \cdot I''(T)
    \quad
    \text{with}
    \quad
    I''(T) = -\zeta_{\ast} \int s^2\exp(-Ts) \, \mathrm{d}F_{\zeta_{\ast}}(s).
\]
From \eqref{eq:mp-moments}, note that
$I(0) = 0$, $I'(0) = \zeta_{\ast} M_1 = \zeta_{\ast}$, and $I''(0) = - \zeta_{\ast} M_2 = - \zeta_{\ast} - \zeta_{\ast}^2$.
Thus, we have $w'(0) = - 2 \zeta_{\ast} M_1 = - 2 \zeta_{\ast}$ and  $w''(0) = 2 \zeta_{\ast}^2 + 2 \zeta_{\ast} + 2 \zeta_{\ast}^2 = 2 \zeta_{\ast} + 4 \zeta_{\ast}^2$.

\paragraph{Part 2.\hspace{-0.5em}}
For $v(T)$, the derivatives are straightforward.
The first derivative is
\[
    v'(T) = -2 \int s\exp(-2Ts) \, \mathrm{d}F_{\zeta_{\ast}}(s).
\]
The second derivative is
\[
    v''(T) = 4 \int s^2\exp(-2Ts) \, \mathrm{d}F_{\zeta_{\ast}}(s).
\]
Hence, from \eqref{eq:mp-moments}, we then get $v'(0) = - 2 M_1 = - 2$ and $v''(0) = 4 M_2 = 4 + 4 \zeta_{\ast}$.

\paragraph{Part 3.\hspace{-0.5em}}
For $u(T)$, the derivatives are similarly straightforward.
The first derivative is 
\[
    u'(T) = -2 \int s^2 \exp(-2Ts) \, \mathrm{d}F_{\zeta_{\ast}}(s).
\]
The second derivative is 
\[
    u''(T) = 4 \int s^3 \exp(-2Ts) \, \mathrm{d}F_{\zeta_{\ast}}(s).
\]
From \Cref{eq:mp-moments} again, we obtain that $u'(0) = -2 M_2 = -2 (1 + \zeta_{\ast})$ and $u''(0) = 4 M_3 = 4 (1 + 3 \zeta_{\ast} + \zeta_{\ast})$.
\end{proof}

We can also numerically verify that the functions in \eqref{eq:sig-mismatch} are indeed different in \Cref{fig:gf_limit_mismatch_in_t,fig:gf_limit_mismatch_surface}.

\begin{figure*}[!ht]
    \centering
  \includegraphics[width=0.495\textwidth]{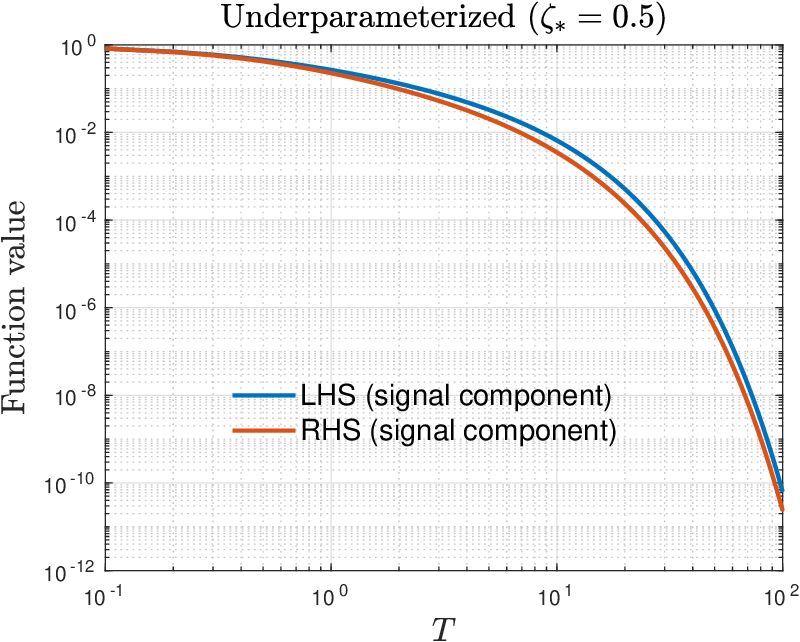}
  \includegraphics[width=0.485\textwidth]{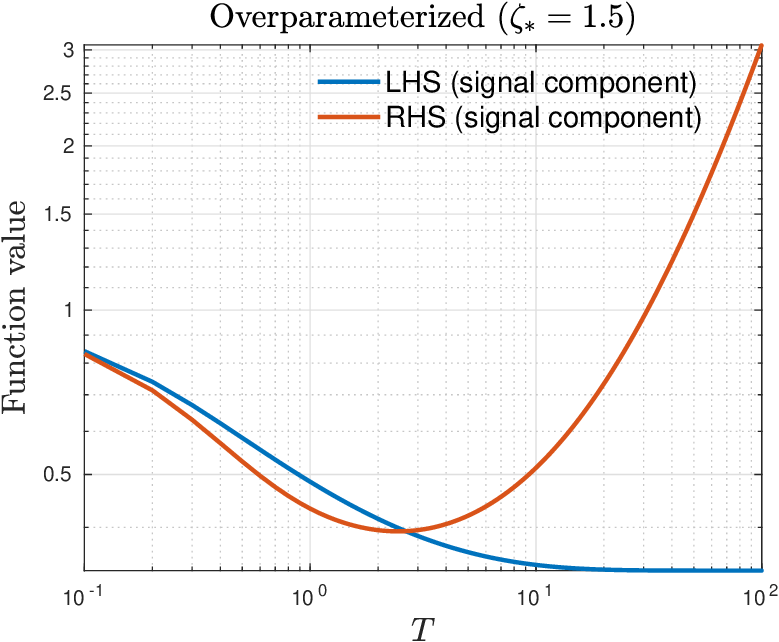}
  \caption{
      Comparison of the LHS and RHS in \eqref{eq:sig-mismatch} (signal component) for the underparameterized (\emph{left}) and overparameterized (\emph{right}) regimes.
      Note that the signal multiplier for both the regimes at $T = 0$ is $1$.
      This is because the estimator is simply the null estimator at $T = 0$, which has bias of $1$.
    }
  \label{fig:gf_limit_mismatch_in_t}
\end{figure*}

\begin{figure*}[!ht]
    \centering
    \includegraphics[width=0.8\textwidth]{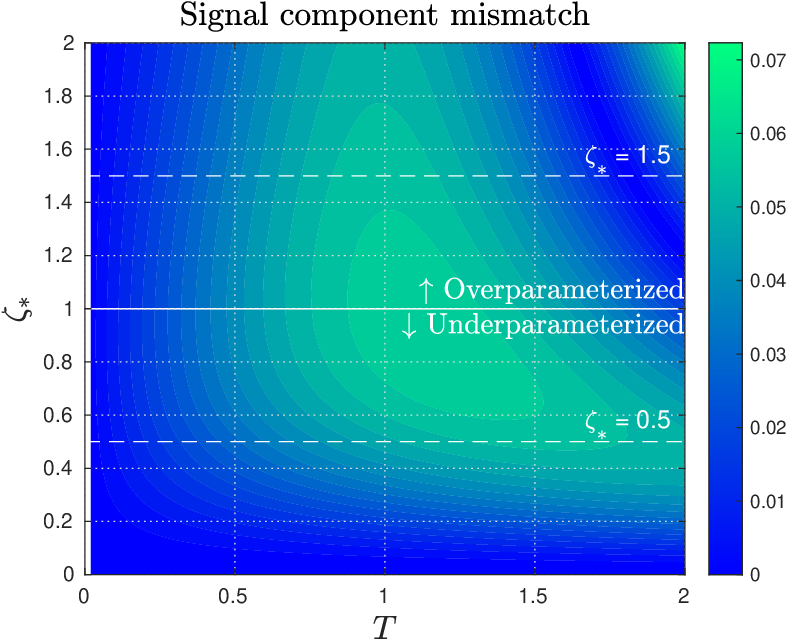}
    \caption{Contour plot of the absolute value of the difference between LHS and RHS of \eqref{eq:sig-mismatch} (signal component).}
    \label{fig:gf_limit_mismatch_surface}
\end{figure*}

\subsubsection{Noise component mismatch}
\label{sec:noise-limits-mismatch-theoretical}

\begin{claim}
    [Second derivatives mismatch for noise component]
    \label{clm:double-derivatives-noise-components-mismatch}
    Let $F_{\zeta_{\ast}}$ be the Marchenko-Pasture law as defined in \eqref{eq:MP-law-le1} and \eqref{eq:MP-law-gt1}.
    Let $\cV_\ell$ and $\cV_r$ be two functions defined as follows:
    \begin{align*}
        \cV_\ell(T)
        &=
        \left(
        1
        +
        \zeta_{\ast}
        \int \frac{(1 - \exp(-Ts))^2}{s} \, \mathrm{d}F_{\zeta_{\ast}}(s)
        \right)
        \left(1 - \zeta_{\ast} \int (1 - \exp(-Ts)) \, \mathrm{d}F_{\zeta_{\ast}}(s) \right)^2
        \\
        \cV_r(T)
        &= 
        {(1 - \zeta_{\ast})
        + \zeta_{\ast} \int \exp(-2 T s) \, \mathrm{d}F_{\zeta_{\ast}}(s)}.
    \end{align*}
    We have $\cV''_\ell(0) = 4 \zeta_{\ast}^2$ and $\cV''_r(0) = 4 \zeta_{\ast} + 4 \zeta_{\ast}^2$, and hence $\cV''_\ell(0) \neq \cV''_r(0)$ for $\zeta_{\ast} > 0$.
\end{claim}
\begin{proof}
    For ease of notation, define the functions $w$, $\tv$, and $\tu$ such that
    \begin{align*}
        w(T) 
        &= 
        \left(1 - \zeta_{\ast} \int (1 - \exp(-Ts)) \, \mathrm{d}F_{\zeta_{\ast}}(s) \right)^2, \\
        \tv(T) 
        &= 
        1
        +
        \zeta_{\ast}
        \int \frac{(1 - \exp(-Ts))^2}{s} \, \mathrm{d}F_{\zeta_{\ast}}(s), \\
        \tu(T) 
        &= 
        {(1 - \zeta_{\ast})
        + \zeta_{\ast} \int \exp(-2 T s) \, \mathrm{d}F_{\zeta_{\ast}}(s)}.
    \end{align*}
    (Note that the function $w$ is the same function as defined in \Cref{clm:double-derivatives-signal-components}.)
    Then we have $\cV_\ell(T) = w(T) \tv(T)$ and $\cV_r(T) = \tu(T)$.
    The first-order derivatives are $\cV'_\ell(T) = w'(T) \tv(T) + w(T) \tv'(T)$ and $\cV'_r(T) = \tu'(T)$.
    The second-order derivatives are $\cV''_\ell(T) = w''(T) + 2 w'(T) \tv'(T) + \tv''(T)$ and $\cV''_r(T) = \tu''(T)$.
    From \Cref{clm:double-derivatives-noise-components}, we obtain
    \begin{align*}
        \cV''_\ell(0)
        = 2 \zeta_{\ast} (1 + 2 \zeta_{\ast}) - 2 \zeta_{\ast}
        = 4 \zeta_{\ast}^2.
    \end{align*}
    On the other hand, from \Cref{clm:double-derivatives-noise-components} again, we have
    \begin{align*}
        \cV''_r(0)
        = 4 \zeta_{\ast} + 4 \zeta_{\ast}^2.
    \end{align*}
    Thus, we have that $\cV''_\ell(0) \neq \cV''_r(0)$ for any $\zeta_{\ast} > 0$.
    This concludes the proof.
\end{proof}

\begin{claim}
    [Second derivatives of various parts of noise component]
    \label{clm:double-derivatives-noise-components}
    Let $w$, $\tv$, and $\tu$ be functions defined in the proof of \Cref{clm:double-derivatives-noise-components-mismatch}.
    Then the following claims hold.
    \begin{itemize}
        \item 
        $w(0) = 1$, $w'(0) = - 2 \zeta_{\ast}$, and $w''(0) = 2 \zeta_{\ast} (1 + 2 \zeta_{\ast})$.
        \item
        $\tv(0) = 1$, $\tv'(0) = 0$, and $\tv''(0) = - 2 \zeta_{\ast}$.
        \item
        $\tu(0) = 1$, $\tu'(0) = - 2$, and $\tu''(0) = 4 (1 + \zeta_{\ast})$.
    \end{itemize}
\end{claim}
\begin{proof}
The functional evaluations are straightforward. 
We will split the first- and second-order derivative calculations into separate parts below.
Recall that for $k \ge 0$, we denote by $M_k = \int s^k \, \mathrm{d} F_{\zeta_{\ast}}(s)$ the $k$-th moment of the Marchenko-Pastur law.

\paragraph{Part 1.\hspace{-0.5em}}
This part is the same as Part 1 of \Cref{clm:double-derivatives-signal-components}.

\paragraph{Part 2.\hspace{-0.5em}}
We start by computing the derivative of the integrand. 
\[
\frac{\partial}{\partial T} \left( \frac{(1 - \exp(-Ts))^2}{s} \right) = 2(1 - \exp(-Ts)) \cdot (-\exp(-Ts)) \cdot (-s) \cdot \frac{1}{s}
= 2(1 - \exp(-Ts))\exp(-Ts)
\]
For the second derivative, note that
\[
\frac{\partial^2}{\partial T^2} \left( \frac{(1 - \exp(-Ts))^2}{s} \right) 
= \frac{\partial}{\partial T} (2 (1 - \exp(-Ts)) \exp(-Ts)) 
= -2 s \exp(-Ts) + 4 s \exp(-2Ts).
\]
Therefore, we have
\[
\tv'(T) = 2\zeta_{\ast} \int (1 - \exp(-Ts))\exp(-Ts) \, \mathrm{d}F_{\zeta_{\ast}}(s)
\]
and
\[
\tv''(T) = 2 \zeta_{\ast} \int s \exp(-Ts) (1 - 2 \exp(-Ts)) \, \mathrm{d}F_{\zeta_{\ast}}(s).
\]
Thus, $\tv'(0) = 0$ and $\tv''(0) = - 2 \zeta_{\ast} M_1 = - 2 \zeta_{\ast}$.

\paragraph{Part 3.\hspace{-0.5em}}
For $\tu(T)$, the derivatives are straightforward.
The first derivative is 
\[
    \tu'(T) = -2 \zeta_{\ast} \int s \exp(-2Ts) \, \mathrm{d}F_{\zeta_{\ast}}(s).
\]
The second derivative is 
\[
    \tu''(T) = 4 \zeta_{\ast} \int s^2 \exp(-2Ts) \, \mathrm{d}F_{\zeta_{\ast}}(s).
\]
From \Cref{eq:mp-moments} again, we obtain that $\tu'(0) = -2 \zeta_{\ast} M_1 = -2 \zeta_{\ast}$ and $\tu''(0) = 4 \zeta_{\ast} M_2 = 4 \zeta_{\ast} (1 + \zeta_{\ast})$.
\end{proof}

We numerically verify that the functions are indeed different in \Cref{fig:gf_limit_mismatch_in_t_var,fig:gf_limit_mismatch_var_surface}.

\begin{figure*}[!ht]
    \centering
    \includegraphics[width=0.495\textwidth]{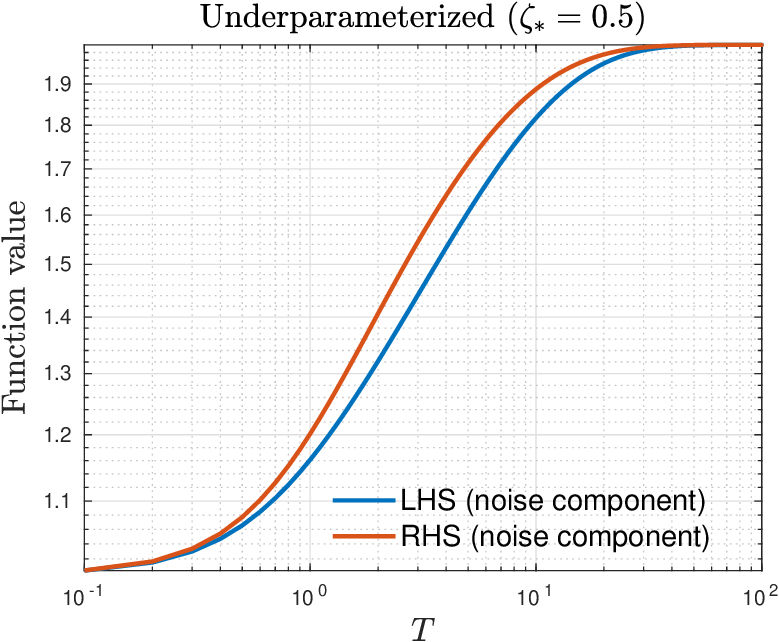}
    \includegraphics[width=0.495\textwidth]{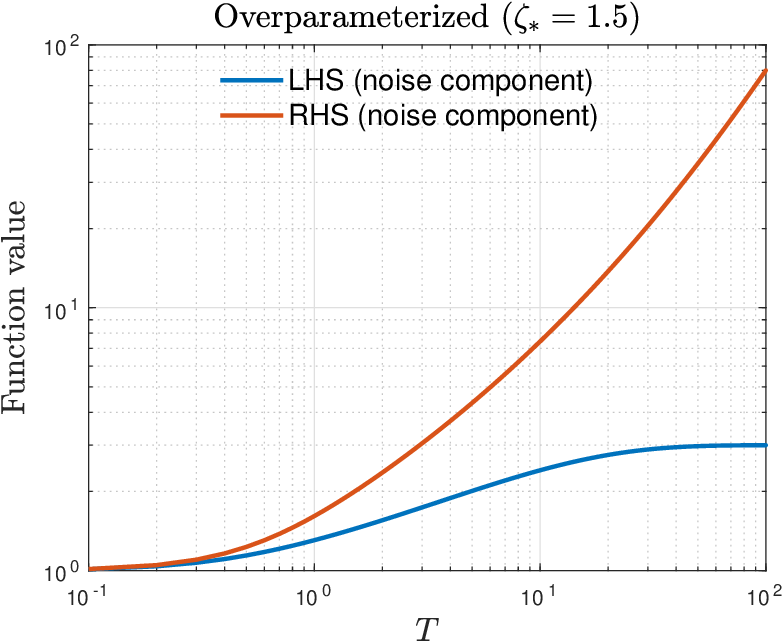}
    \caption{Comparison of the LHS and RHS in \eqref{eq:var-mismatch} (noise component) for the underparameterized the (\emph{left}) and overparameterized (\emph{right}) regimes.}
    \label{fig:gf_limit_mismatch_in_t_var}
\end{figure*}

\medskip

\begin{figure*}[!ht]
    \centering
    \includegraphics[width=0.8\textwidth]{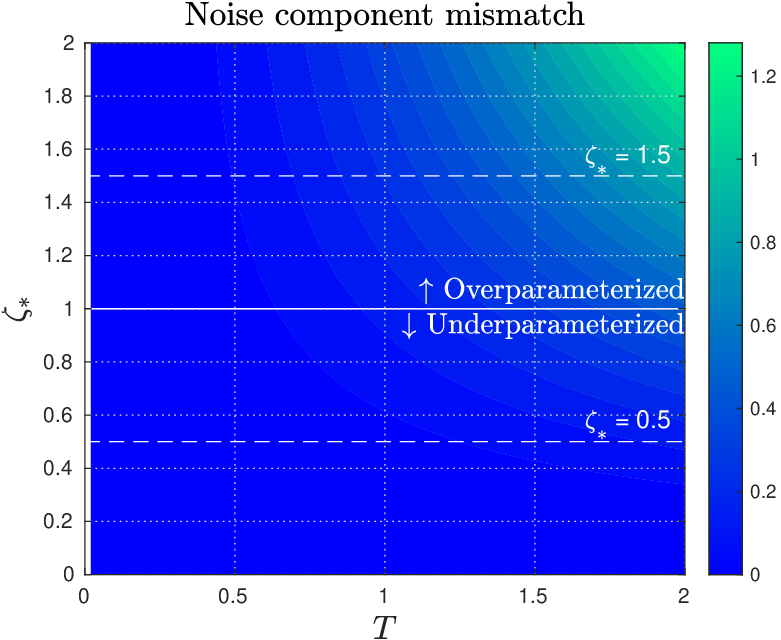}
    \caption{Contour plot of the absolute value of the difference between LHS and RHS of \eqref{eq:var-mismatch} (noise component).}
    \label{fig:gf_limit_mismatch_var_surface}
\end{figure*}

\subsection{A helper lemma related to the Marchenko-Pastur law}
\label{sec:mp-moments}

\bigskip

\begin{lemma}
    [Moments of the Marchenko-Pasture distribution]
    \label{lem:moments-mp}
    Let $F_{\zeta_{\ast}}$ be the Marchenko-Pasture law as defined in \eqref{eq:MP-law-le1} and \eqref{eq:MP-law-gt1}.
    For $k \ge 1$, we have
    \[
        \int s^k \, \mathrm{d}F_{\zeta_{\ast}}(s)
        = \sum_{i = 0}^{k-1}
        \frac{1}{i+1} \binom{k}{i} \binom{k-1}{i}
        \zeta_{\ast}^i.
    \]
\end{lemma}
The explicit moment formula in \Cref{lem:moments-mp} is well-known.
See, for example, Lemma 3.1 of \cite{bai_silverstein_2010}.
It can be derived using the Chu-Vandermonde identity, also known as Vandermonde's convolutional formula for binomial coefficients \citep[Chapter 3]{koepf1998hypergeometric}.

We will use \Cref{lem:moments-mp} to obtain the following moments explicitly.
Let $M_k$ denote the $k$-th moment $\int s^k \, \mathrm{d} F_{\zeta_{\ast}}(s)$ of the Marchenko-Pastur distribution.
We have
\begin{equation}
    \label{eq:mp-moments}
    M_0 = 1, 
    \quad
    M_1 = 1,
    \quad
    M_2 = 1 +  \zeta_{\ast},
    \quad
    M_3 = 1 + 3 \zeta_{\ast} + \zeta_{\ast}^2,
    \quad
    M_4 = 1 + 6 \zeta_{\ast} + 6 \zeta_{\ast}^2 + \zeta_{\ast}^3.
\end{equation}

\section{Proof sketch for \Cref{thm:uniform-consistency-squared}}
\label{sec:thm:uniform-consistency-squared-proof-ingredients}

In this section, we outline the proof idea of \Cref{thm:uniform-consistency-squared}. 
The extension to general test functionals can be found in \Cref{sec:proof-thm:uniform-consistency-general}.
We will prove both the theorems for a general starting estimator $\hat\bbeta_0$. 

\subsection{Step 1: LOO concentration}
\label{sec:thm:uniform-consistency-squared-proof-ingredients-step1}

The most challenging part of our proof is establishing concentration for $\hR^\loo(\widehat{\bbeta}_k)$. 
This is achieved by upper bounding the norm of the gradient of the mapping $(\bw_1, \cdots, \bw_n) \mapsto \hR^\loo(\widehat{\bbeta}_k)$, where $\bw_i = (\bx_i, y_i)$. 
Although this mapping is not exactly Lipschitz, it is approximately Lipschitz in the sense that its gradient is bounded on a set that occurs with high probability.

For $k \in \{0\} \cup [K]$, we define $f_k: \R^{n(p + 2)} \mapsto \R$ as $f_k(\bw_1, \cdots, \bw_n) = \hR^\loo(\hat{\bbeta}_k)$.
Our goal is to upper bound $\|\nabla f_k\|_2$. 
It will become clear that $f_k$ is Lipschitz continuous on a closed convex set $\Omega$. 
We define $\Omega$ as follows:
\begin{align}\label{eq:Omega}
    \Omega = \left\{ \|\hat{\bSigma}\|_{\op} \leq C_{\Sigma, \zeta}, \, \|\by\|_2^2 \leq n(m + \log n)\right\},
\end{align}
where $C_{\Sigma, \zeta} = 2C_0\sigma_{\Sigma}(1 + \zeta) + 1$, $m = m_2$, and $C_0 > 0$ is a numerical constant. 
It can be verified that $\Omega$ is a convex set of the data. 
Standard concentration results (see \Cref{lemma:op-Sigma} and \Cref{lemma:norm-y}) imply that with an appropriately selected $C_0$, we have $\P(\Omega) \geq 1 - 2(n + p)^{-4} - n^{-1}m_4 \log^{-2} n$. 
In other words, for large $(n, p)$, the input samples will fall inside $\Omega$ with high probability. 

In the following, we establish the Lipschitz continuity of ${f}_k$ when restricted to $\Omega$, which is a closed convex set. 
This can be equivalently stated as the Lipschitz continuity of the composition of the projection onto $\Omega$ and $f_k$. 
To prove this, we upper bound the Euclidean norm of the gradient, as detailed in \Cref{lemma:gradient-upper-bound}. 
The proof of \Cref{lemma:gradient-upper-bound} can be found in \Cref{sec:proof-lemma:gradient-upper-bound}.

\begin{restatable}
    [Gradient upper bound]
    {lemma}
    {LemmaGradientUpperBound}
    \label{lemma:gradient-upper-bound}
    There exists a constant $\xi(C_{\Sigma, \zeta}, \Delta, m, B_0) > 0$ that depends only on $(C_{\Sigma, \zeta}, \Delta, m, B_0)$, such that on the set $\Omega$, it holds that
    \begin{align*}
        \|\nabla_{\bW} f_k(\bW) \|_F \leq \frac{K \xi(C_{\Sigma, \zeta}, \Delta, m, B_0) \cdot \log n}{\sqrt{n}} 
    \end{align*}
    for all $k \in \{0\} \cup [K]$. 
    In the above display, we define $\bW = (\bw_1, \cdots, \bw_n)$ and $K$, we recall, is the total number of GD iterations.   
\end{restatable}

We define $h: \R^{n(p + 2)} \mapsto \R^{n(p + 2)}$ as the projection that projects its inputs onto $\Omega$.  
Define $\tilde{f}_k = f_k \circ h$. \Cref{lemma:gradient-upper-bound} implies that $\tilde f_k$ is a Lipschitz continuous mapping with a Lipschitz constant as stated in \Cref{lemma:gradient-upper-bound}. 
By assumption, the input data distribution satisfies a $T_2$-inequality, allowing us to apply a powerful concentration inequality stated in \Cref{prop:Gozlan} to obtain the desired concentration result. 
We state this result as \Cref{lemma:Rloo} below, and its proof can be found in \Cref{sec:proof-lemma:Rloo}.
\begin{restatable}
    [LOO concentration]
    {lemma}
    {LemmaRloo}
    \label{lemma:Rloo}
    We assume the assumptions of \Cref{thm:uniform-consistency-squared}. 
    Then with probability at least $1 - 2(n + p)^{-4} - (n\log^2n)^{-1}  m_4 - 2(K + 1)C_{\LSI}n^{-2}$, it holds that for all $k \in \{0\} \cup [K]$
\begin{align*}
		\left| \hR^\loo(\hat{\bbeta}_k) - \E[\tilde f_k(\bw_1, \cdots, \bw_n)] \right| \leq \frac{2\sigma_{\LSI} L K \xi(C_{\Sigma, \zeta}, \Delta, m, B_0) \cdot (\log n)^{3/2}}{\sqrt{n}}, 
\end{align*}
where we let $L = (L_{f}^2 \sigma_{\Sigma}  + L_{f}^2 + \sigma_{\Sigma})^{1/2}$, $\sigma_{\LSI}^2 = \sigma_{z}^2 \vee \sigma_{\ep}^2$, and $C_{\LSI}$ is a positive numerical constant that appears in \Cref{prop:Gozlan}.  
\end{restatable}

\subsection{Step 2: Risk concentration}
\label{sec:thm:uniform-consistency-squared-proof-ingredients-step2}

In the second part, we provide concentration bounds for the prediction risk $R(\hat\bbeta_k)$. 
We follow a similar approach as in Step 1, establishing that $R(\hat\bbeta_k)$ is a Lipschitz function of the input data with high probability. 
Leveraging the assumption of a $T_2$-inequality in the data distribution, we apply \Cref{prop:Gozlan} to derive a concentration result. 
The proof of this result is presented in \Cref{sec:proof-lemma:cRk}. 
We state the concentration result as \Cref{lemma:cRk}.

\begin{restatable}
    [Risk concentration]
    {lemma}
    {LemmacRk}
    \label{lemma:cRk}
    We write $R(\hat\bbeta_k) = r_k(\bw_1, \cdots, \bw_n)$ and define $\tilde{r}_k(\bw_1, \cdots, \bw_n) = r_k(h(\bw_1, \cdots, \bw_n))$. 
    Then under the assumptions of \Cref{thm:uniform-consistency-squared}, with probability at least $1 - 2(n + p)^{-4} - (n\log^2n)^{-1}m_4 - 2(K + 1)C_{\LSI}n^{-2}$, for all $k \in \{0\} \cup [K]$ we have 
	\begin{align*}
		\left| R(\hat\bbeta_k) - \E[\tilde{r}_k (\bw_1, \cdots, \bw_n)] \right| \leq \frac{2\sigma_{\LSI} L \bar\xi(C_{\Sigma, \zeta}, \Delta, m, B_0)(\log n)^{3/2}}{\sqrt{n}},
	\end{align*}
 where $\bar\xi(C_{\Sigma, \zeta}, \Delta, m, B_0) > 0$ depends uniquely on $(C_{\Sigma, \zeta}, \Delta, m, B_0)$. 
\end{restatable}

\subsection{Step 3: LOO bias analysis}
\label{sec:thm:uniform-consistency-squared-proof-ingredients-step3}

In Steps 1 and 2, we have proven concentration results for both $R(\hat\bbeta_k)$ and $\hR^\loo(\hat{\bbeta}_k)$. 
Specifically, we have shown that $R(\hat\bbeta_k)$ concentrates around $\E[\tilde{r}_k (\bw_1, \cdots, \bw_n)]$ and $\hR^\loo(\hat{\bbeta}_k)$ concentrates around $\E[\tilde{f}_k (\bw_1, \cdots, \bw_n)]$. 
These expectations represent the target functionals composed with the projection $h$.

Next, we demonstrate that incorporating the projection $h$ into the expectation does not significantly alter the quantities of interest. 
This result is presented as \Cref{lemma:projection-effects} below.
\begin{restatable}
    [Projection effects]
    {lemma}
    {LemmaProjectionEffects}
    \label{lemma:projection-effects}
    Under the assumptions of \Cref{thm:uniform-consistency-squared}, it holds that
\begin{align}\label{eq:truncation-expectation}
\begin{split}
	& \sup_{k \in \{0\} \cup [K]}\left|\E[\tilde{r}_k (\bw_1, \cdots, \bw_n)] - \E[r_k(\bw_1, \cdots, \bw_n)] \right| = o_{n}(1), \\
	& \sup_{k \in \{0\} \cup [K]} \left|\E[\tilde{f}_k (\bw_1, \cdots, \bw_n)] - \E[f_k(\bw_1, \cdots, \bw_n)] \right| = o_n(1). 
\end{split}
\end{align}
\end{restatable}

Finally, we aim to establish a result showing that the prediction risk is stable with respect to the sample size. 
Specifically, we seek to demonstrate that $\E[R(\hat\bbeta_k)]$ is approximately equal to $\E[R(\hat\bbeta_{k, -1})]$, which is equivalent to $\E[r_k(\bw_1, \cdots, \bw_n)] \approx \E[f_k(\bw_1, \cdots, \bw_n)]$.

Formally speaking, we prove the following lemma.

\begin{restatable}
    [LOO bias]
    {lemma}
    {LemmaExpectationClose}
    \label{lemma:expectation-close}
    Under the assumptions of \Cref{thm:uniform-consistency-squared}, it holds that
    \begin{align*}
        \sup_{k \in \{0\} \cup [K]} \big|\E[R(\hat\bbeta_k)] - \E[R(\hat\bbeta_{k, -1})] \big| = o_{n}(1). 
    \end{align*}
    This is equivalently saying 
    \begin{align*}
        \sup_{k \in \{0\} \cup [K]} \big|\E[{r}_k (\bw_1, \cdots, \bw_n)] - \E[{f}_k (\bw_1, \cdots, \bw_n)]\big| = o_{n}(1). 
    \end{align*}
\end{restatable}

We defer the proofs of \Cref{lemma:projection-effects} and \Cref{lemma:expectation-close} to Sections \ref{sec:proof-lemma:projection-effects} and \ref{sec:proof-lemma:expectation-close}, respectively. 

\Cref{thm:uniform-consistency-squared} then follows from these three steps. 
To be precise, by putting together \Cref{lemma:Rloo,lemma:cRk,lemma:projection-effects,lemma:expectation-close}, we obtain that with probability at least $1 - 4(n + p)^{-4} - 2(n \log^2 n)^{-1} m_4 - 4(K + 1) C_{\LSI} n^{-2}$, for all $k \in \{0\} \cup [K]$, we have
\begin{align}
    \label{eq:event-i}
    &\sup_{k \in \{0\} \cup [K]}\left| R(\hat\bbeta_k) - \hat R^{\loo} (\hat\bbeta_k) \right| \notag \\
    &\quad \leq \frac{2\sigma_{\LSI} L K \xi(C_{\Sigma, \zeta}, \Delta, m, B_0) \cdot (\log n)^{3/2} + 2\sigma_{\LSI} L \bar\xi(C_{\Sigma, \zeta}, \Delta, m, B_0)(\log n)^{3/2}}{\sqrt{n}}. 
\end{align}

Since $\zeta = p / n$ is both lower and upper bounded, thus we can conclude that 
\begin{align*}
    \sum_{n = 1}^{\infty} \left\{ 4(n + p)^{-4} + 2(n \log^2 n)^{-1} m_4 + 4(K + 1) C_{\LSI} n^{-2} \right\} < \infty.
\end{align*}
Hence, \Cref{thm:uniform-consistency-squared} follows immediately by applying the first Borel–Cantelli lemma. 
More precisely, we prove that almost surely the event depicted in \eqref{eq:event-i} occurs only finitely many times.

\section[Supporting lemmas for the proofs of \Cref{thm:uniform-consistency-squared}, \Cref{thm:uniform-consistency-general}, and \Cref{thm:coverage}]{Supporting lemmas for the proofs of \Cref{thm:uniform-consistency-squared,thm:uniform-consistency-general,thm:coverage}}
\label{sec:helper-lemmas-uniform-consistency}

We present in this section several supporting lemmas that are useful for the analysis presented in \Cref{sec:uniform-consistency-proof-squared} and \Cref{sec:proof-lemma:gradient-upper-bound}.  Without any loss of generality, in this section, we always assume $n \geq 3$, thus $\log n \geq 1$. 

\subsection{Technical preliminaries}
\label{sec:definitions}

We define below what it means for a distribution to satisfy log Sobolev inequality (LSI). 

\begin{definition}[LSI]
\label{def:lsi}
We say a distribution $\mu$ satisfies LSI if there exists a constant $\sigma(\mu) \geq 0$ such that for all smooth function $f$, it holds that 
	\begin{align}\label{eq:LSI}
		\ent_{w \sim \mu}[f(w)^2] \leq 2 \sigma^2(\mu) \E_{w \sim \mu} \big[\|\nabla f(w)\|_2^2\big],
	\end{align}
	where the entropy of a non-negative random variable $Z$ is defined as
\begin{align*}
	\ent[Z] = \E[Z \log Z] - \E[Z] \log \E[Z].  
\end{align*}
\end{definition}

\subsection[Concentration based on T2-inequality]{Concentration based on $T_2$-inequality}
\label{sec:T2}

In this section, we discuss useful properties of the $T_2$-inequality.
An important result that will be applied multiple times throughout the proof is Theorem 4.31 of \cite{van2014probability}, which we include below for the convenience of the readers.
See also \cite{gozlan2009characterization}.

\begin{proposition}
    [Equivalent characterizations of $T_2$-inequality]
    \label{prop:Gozlan}
	Let $\mu$ be a probability measure on a Polish space $(\mathcal{X}, d)$, and
  let $\{X_i\}_{i \leq n}$ be i.i.d. $\sim \mu$. Denote by $d_n(x, y) = [\sum_{i
    = 1}^n d(x_i, y_i)^2]^{1/2}$. Then the following are equivalent: 
	\begin{enumerate}[leftmargin=7mm]
		\item $\mu$ satisfies the $T_2$-inequality:
		\begin{align*}
			W_2(\mu, \nu) \leq \sqrt{2\sigma^2 \cD_{\KL}(\nu \, ||\, \mu)}\,\,\,\, \mbox{ for all }\nu.
		\end{align*}
		\item $\mu^{\otimes n}$ satisfies the $T_1$-inequality for every $n \geq 1$:
		\begin{align*}
			W_1(\mu^{\otimes n}, \nu) \leq \sqrt{2\sigma^2 \cD_{\KL} (\nu \, || \, \mu^{\otimes n})} \,\,\,\, \mbox{ for all }\nu \mbox{ and }n \geq 1. 
		\end{align*}
		\item There is an absolute constant $C_{\LSI}$, such that 
		\begin{align}\label{eq:LSI-concentration}
			\P\left( f(X_1, \cdots, X_n) - \E[f(X_1, \cdots, X_n)] \geq t \right) \leq C_{\LSI} e^{-t^2 / 2\sigma^2}
		\end{align}
		for every $n \geq 1$, $t \geq 0$ and $1$-Lipschitz function $f$. 
	\end{enumerate}
\end{proposition}

\subsection{Dimension-free concentration}
\label{sec:dimension-free}

Define $\bw_i = (\bx_i, y_i)$. 
The following lemma is a straightforward consequence of the assumptions and the $T_2$-inequality.

\begin{lemma}
    [Dimension-free concentration]
    \label{lemma:concentration}
	We let 	$\sigma_{\LSI}^2 = \sigma_{z}^2 \vee \sigma_{\ep}^2$, and $L = (L_{f}^2 \sigma_{\Sigma}  + L_{f}^2 + \sigma_{\Sigma})^{1/2}$. Then for any $n \geq 1$, $t \geq 0$, and 1-Lipschitz function $f$, it holds that
	\begin{align*}
	\P\left( f(\bw_1, \cdots, \bw_n) - \E[f(\bw_1, \cdots, \bw_n)] \geq Lt \right) \leq C_{\LSI} e^{-t^2 / 2\sigma_{\LSI}^2},
\end{align*}
where we recall that $C_{\LSI} > 0$ is an absolute constant introduced in \Cref{prop:Gozlan}.  
\end{lemma}

\begin{proof}
	Since $f$ is 1-Lipschitz, for any $\bw_i, \tilde\bw_i \in \R^p$ 
	\begin{align*}
		|f(\bw_1, \cdots, \bw_n) - f(\tilde\bw_1, \cdots, \tilde\bw_n)| \leq & \sqrt{ \sum_{i = 1}^n \|\bw_i - \tilde\bw_i\|_2^2} \\
		= & \sqrt{ \sum_{i = 1}^n \|\bx_i - \tilde\bx_i\|_2^2 + \sum_{i = 1}^n |\by_i - \tilde\by_i|^2} \\
		\leq  & \sqrt{ \sum_{i = 1}^n \sigma_{\Sigma} (L_{f}^2 + 1) \|\bz_i - \tilde\bz_i\|_2^2 + \sum_{i = 1}^n L_{f}^2 |\ep_i - \tilde\ep_i|^2} \\
		\leq & L \sqrt{\sum_{i = 1}^n \|\bz_i - \tilde\bz_i\|_2^2 + \sum_{i = 1}^n |\ep_i - \tilde\ep_i|^2}. 
	\end{align*}
	Invoking Corollary 4.16 of \cite{van2014probability}, we obtain that
	\begin{align*}
		W_1(\mu_z^{\otimes n} \otimes \mu_{\ep}^{\otimes n}, \nu) \leq \sqrt{2 \sigma^2_{\LSI} \cD_{\KL}(\nu \, || \, \mu_z^{\otimes n} \otimes \mu_{\ep}^{\otimes n})}
	\end{align*}
	for all $\nu$. We then see that the desired concentration inequality is a straightforward consequence of Proposition \ref{prop:Gozlan}. 
\end{proof}

\subsection{Upper bounding operator norms and response energy}
\label{sec:op-and-energy}

We then state several technical lemmas required for our analysis. Recall that $\hat\bSigma = \bX^{\top} \bX / n$. Our first lemma upper bounds the operator norm of $\hat\bSigma$. 

\begin{lemma}
    \label{lemma:op-Sigma}
    We assume the assumptions of \Cref{thm:uniform-consistency-squared}. Then there exists a numerical constant $C_0 > 0$, such that with probability at least $1 - (n + p)^{-4}$
	\begin{align*}
		\|\hat{\bSigma}\|_{\op} \leq 2C_0\sigma_{\Sigma}(1 + \zeta) + 1. 
	\end{align*}   
\end{lemma}

\begin{proof}
	Note that the operator norm of $\hat\bSigma$ is equal to the operator norm of $\bZ \bSigma \bZ^{\top} / n + \mathbf{1}_{n \times n} / n \in \R^{n \times n}$. 
	
	To proceed, we will utilize a canonical concentration inequality that bounds the operator norm of random matrices with sub-Gaussian entries. 
	This further requires the introduction of several related concepts. 
 
 To be specific, we say a random variable $R$ is \emph{sub-Gaussian} if and only if there exists $K_R > 0$ such that $\|R\|_{L^d} \leq K_R\sqrt{d}$ for all $d \geq 1$. Proposition 2.5.2 of \cite{vershynin2018high} tells us that when such upper bound is satisfied, the sub-Gaussian norm of this random variable $\|Z\|_{\Psi_2}$ is no larger than $4K_R$. 
	
	By Assumption \ref{asm:feat_dist_loocv} and Proposition \ref{prop:Gozlan}, it holds that 
	\begin{align*}
		\P\left( |z_{11}| \geq t \right) \leq 2C_{\LSI} e^{-t^2 / 2\sigma_z^2}.
	\end{align*}  
	Leveraging the above upper bound and applying an appropriate integral inequality, we can conclude that for all $d \geq 1$, 
	\begin{align*}
		\E[|z_{11}|^d] \leq C_{\LSI} d (d / 2)^{d / 2},
	\end{align*} 
	hence $\|z_{11}\|_{\Psi_2} \leq 8 + 8C_{\LSI}$. By Theorem 4.4.5 of \cite{vershynin2010introduction}, we see that for all $t \geq 0$, with probability at least $1 - 2\exp(-t^2)$ 
	\begin{align}\label{eq:Zop}
		\|\bZ\|_{\op} \leq C'(8 + 8C_{\LSI})(\sqrt{n} + \sqrt{p} + t),
	\end{align}  
	where $C' > 0$ is a numerical constant. Taking $t = 2\sqrt{\log (p + n)}$, we conclude that $\|\bZ\|_{\op} \leq C'(8 + 8C_{\LSI})(\sqrt{n} + \sqrt{p} + 2\sqrt{\log (n + p)})$ with probability at least $1 - 2(p + n)^{-4}$. When this occurs, a straightforward consequence is that
	\begin{align*}
		n\|\hat{\bSigma}\|_{\op} \leq \|\bZ\|_{\op}^2\|\bSigma\|_{\op} + n \leq C_0\sigma_{\Sigma}(n + p + \log (n + p)) + n
	\end{align*}
	for some positive numerical constant $C_0$,
	thus completing the proof of the lemma.
\end{proof}

Our next lemma upper bounds $\|\by\|_2^2 / n$. 
This lemma is a direct consequence of Chebyshev's inequality, and we skip the proof for the compactness of the presentation. 
\begin{lemma}\label{lemma:norm-y}
	We assume the assumptions of \Cref{thm:uniform-consistency-squared}. Then with probability at least $1 - n^{-1} m_4 \log^{-2} n   $, we have $\|\by\|_2^2 / n \leq m_2 + \log n$. 
\end{lemma}

\subsection{Other useful norm bounds}
\label{sec:other}

Our next lemma upper bounds the Euclidean norm of $\btheta = \E[y_0 \bx_0] \in \RR^{p + 1}$, where we recall that $(\bx_0, y_0) \overset{d}{=} (\bx_1, y_1)$. 
\begin{lemma}\label{lemma:norm-theta}
	Under the assumptions of \Cref{thm:uniform-consistency-squared}, we have $\|\btheta\|_2 \leq (\sigma_{\Sigma}^{1/2} + 1) m_2^{1/2}$. 
\end{lemma}
\begin{proof}
	We notice that  
	\begin{align*}
		\btheta = \left( \begin{array}{c}
			\bSigma^{1/2}\E[y_{\new} \bz_{\new}] \\
			\E[y_0]
		\end{array} \right). 
	\end{align*}
        We let $\bx_\new^{\top} = (\bz_{\new}^{\top}\bSigma^{1/2}, 1)$.
	By assumption, $\bz_{\new}$ is isotropic. Hence, $y_0$ admits the following decomposition:
 \begin{align*}
     y_0 = \sum_{i = 1}^p \E[y_0 z_{0, i}] z_{0, i} + \omega, \qquad \E[\omega z_{0, i}] = 0 \,\, \mbox{ for all }i \in [n]. 
 \end{align*}
In addition, $\E[y_0^2] = \E[w^2] + \sum_{i \in [p]} \E[y_0 z_{0,i}]^2$. 
 As a result, we are able to deduce that $\|\E[y_{\new} \bz_{\new}]\|_2 \leq m_2^{1/2}$, where we recall that $m_2 = \E[y_0^2]$. This further tells us $\|\btheta\|_2 \leq \|\bSigma\|_{\op}^{1/2} \times \|\E[y_{\new} \bz_{\new}]\|_2 + m_2^{1/2} \leq (\sigma_{\Sigma}^{1/2} + 1) m_2^{1/2}$, thus completing the proof of the lemma. 
\end{proof}

We next prove that $\|\hat\bSigma\|_{\op}$ is sub-exponential. 

\begin{lemma}\label{lemma:op-subGaussian}
	We define $\tilde C_0 = C' \sigma_{\Sigma}(8 + 8 C_{\LSI})$, where we recall that $C'$ is a positive numerical constant that appears in \Cref{eq:Zop}. Under the assumptions of \Cref{thm:uniform-consistency-squared}, for all $\lambda \geq 0$ and $n \geq \lambda \tilde C_0^2 + 1$, there exists a constant $\cE(\tilde C_0, \zeta, \lambda) > 0$ that depends only on $(\tilde C_0, \zeta, \lambda)$, such that 
	\begin{align*}
		\E[\exp(\lambda \|\hat{\bSigma}\|_{\op})]  \leq \cE(\tilde C_0, \zeta, \lambda). 
	\end{align*}
\end{lemma}
\begin{proof}
	By \Cref{eq:Zop}, for all $t \geq 0$, with probability at least $1 - 2 \exp(-nt^2)$
\begin{align*}
	\|\hat{\bSigma}\|_{\op}^{1/2} = n^{-1/2}\|\bX\|_{\op} \leq \tilde C_0(1 + \zeta^{1/2} + t). 
\end{align*}
As a result, for all $\lambda \geq 0$, 
\begin{align*}
	 & \E[\exp(\lambda \|\hat{\bSigma}\|_{\op})] \\
	  &\leq  1 + \int_{0}^{\infty} 2\lambda s e^{\lambda s^2}\P\Big( \|\hat{\bSigma}\|_{\op}^{1/2} \geq s \Big) \mathrm{d} s  \\
	 &\leq 1 + 2\lambda \tilde C_0^2(1 + \zeta^{1/2})^2 e^{\lambda \tilde C_0^2(1 + \zeta^{1/2})^2} + \int_{\tilde C_0(1 + \zeta^{1/2})} 2\lambda s e^{\lambda s^2}\P\Big( \|\hat{\bSigma}\|_{\op}^{1/2} \geq s \Big) \mathrm{d} s \\
	 &\leq 1 + 2\lambda \tilde C_0^2(1 + \zeta^{1/2})^2 e^{\lambda \tilde C_0^2(1 + \zeta^{1/2})^2} + \int_0^{\infty} 4\lambda \tilde C_0^2(1 + \zeta^{1/2} + t) e^{\lambda \tilde C_0^2 (1 + \zeta^{1/2} + t)^2 - nt^2} \mathrm{d} t \leq \cE( \tilde C_0, \zeta, \lambda),
\end{align*}
thus completing the proof of the lemma. 
\end{proof}

\subsection[Upper bounding full and leave-one-out estimators]{Upper bounding $\|\hat\bbeta_k\|_2$ and $\|\hat\bbeta_{k, -i}\|_2$}
\label{eq:upper-bound-two-beta}

We then prove that on $\Omega$, the Euclidean norm of the coefficient estimates $\{\hat\bbeta_{k}, \hat\bbeta_{k, -i}: k \in [K], i \in [n] \}$ are uniformly upper bounded. In addition, apart from a logarithmic factor, this upper bound depends only on the constants from our assumptions and in particular is independent of $(n, p)$. 

\begin{lemma}\label{lemma:upper-bound-beta}
	For the sake of simplicity, we let
 \begin{align}\label{eq:BstarB}
     B_{\ast} = ( B_0 + \Delta C_{\Sigma, \zeta}^{1/2} \sqrt{m + 1} ) \cdot e^{C_{\Sigma, \zeta}\Delta}.
 \end{align} 
 Then on the set $\Omega$, for all $k \in \{0\} \cup [K]$ and $i \in [n]$, it holds that
	\begin{align*}
		 \|\hat\bbeta_{k}\|_2 \leq B_{\ast} \sqrt{\log n}, \qquad  \|\hat\bbeta_{k,i}\|_2 \leq B_{\ast} \sqrt{\log n}.
	\end{align*}
\end{lemma}
\begin{proof}
	By definition, 
	\begin{align*}
		\hat\bbeta_{k + 1} = & \hat\bbeta_k + \frac{\delta_k}{n} \sum_{i = 1}^n (y_i - \bx_i^{\top} \hat\bbeta_k)\bx_i \\
		= & \hat\bbeta_k - \delta_k \hat{\bSigma} \hat\bbeta_k + \frac{\delta_k}{n}  \bX^{\top} \by.
	\end{align*}
 Applying the triangle inequality, we obtain the following upper bound:
	\begin{align*}
		\|\hat\bbeta_{k + 1}\|_2 \leq &  \|\hat\bbeta_k\|_2 + \delta_k \|\hat{\bSigma}\|_{\op} \cdot \|\hat\bbeta_k\|_2 + \delta_k \cdot \|\hat{\bSigma}\|_{\op}^{1/2} \cdot \|\by / \sqrt{n}\|_2 \\
		\leq & \left(1 + \delta_k C_{\Sigma, \zeta} \right) \cdot \|\hat\bbeta_k\|_2 + \delta_k C_{\Sigma, \zeta}^{1/2} \sqrt{m + \log n}.
	\end{align*}
	By induction, we see that on $\Omega$
	\begin{align*}
		\|\hat\bbeta_k\|_2 \leq \left( B_0 + \Delta C_{\Sigma, \zeta}^{1/2} \sqrt{m + \log n} \right) \cdot e^{C_{\Sigma, \zeta}\Delta}
	\end{align*}
	for all $k \in [K]$. The upper bound for $\|\hat\bbeta_{k, -i}\|_2$ follows using exactly the same argument. We complete the proof of the lemma as $\log n \geq 1$.  
\end{proof}

The following corollary is a straightforward consequence of \Cref{lemma:upper-bound-beta} and the Cauchy-Schwartz inequality. 
\begin{corollary}\label{cor:y-Xbeta}
	On the set $\Omega$, it holds that 
	\begin{align*}
		& \frac{1}{n}\|\by - \bX \hat\bbeta_{k, -i}\|_2^2 \leq \Big(2m + 2 + 2C_{\Sigma, \zeta} B_{\ast}^2 \Big) \cdot \log n, \\
            & \frac{1}{n}\|\by - \bX \hat\bbeta_{k}\|_2^2 \leq \Big(2m + 2 + 2C_{\Sigma, \zeta} B_{\ast}^2 \Big) \cdot \log n
	\end{align*}
	for all $k \in \{0\} \cup [K]$ and $i \in [n]$. 
\end{corollary}
For the compactness of future presentation, we define 
\begin{align}\label{eq:BstarBbar}
    \bar B_{\ast} = (2m + 2 + 2C_{\Sigma, \zeta} B_{\ast}^2)^{1/2}
\end{align}
We comment that both $B_{\ast}$ and $\bar B_{\ast}$ depend only on $(C_{\Sigma, \zeta}, \Delta, m, B_0)$. 

\subsection[Upper bounding residual]{Upper bounding $|y_i - \bx_i^{\top} \bbeta_{k, -i}|$}
\label{sec:upper-bound-loocv-res}

We next upper bound $|y_i - \bx_i^{\top} \hat\bbeta_{k, -i}|$ on $\Omega$. More precisely, we shall upper bound collectively the Frobenius norms of 
\begin{align*}
	& \ba_k = (y_i - \bx_i^{\top}\hat\bbeta_{k, -i})_{i = 1}^n \in \R^n \qquad \mbox{and} \\
	& \bE_k = \left[ \bX(\hat\bbeta_{k} - \hat\bbeta_{k, -1}) \mid \cdots \mid \bX(\hat\bbeta_{k} - \hat\bbeta_{k, -n}) \right] \in \R^{n \times n}
\end{align*}
respectively and recursively. For the base case $k = 0$, we have
\begin{align*}
	 \|\ba_0\|_2^2  \leq  \bar B_{\ast}^2 n\log n, \qquad \|\bE_0\|_F^2 = 0, 
\end{align*} 
where the first upper bound follows from Corollary \ref{cor:y-Xbeta}.

Our lemma for this part can be formally stated as follows:
\begin{lemma}\label{lemma:y-xbeta}
	We define
	\begin{align}
		& \cG_1(C_{\Sigma, \zeta}, \Delta, m, B_{0}) = \bar B_{\ast}\sqrt{e^{3\Delta C_{\Sigma, \zeta} + 2\Delta^2 C_{\Sigma, \zeta}^2}(\Delta C_{\Sigma, \zeta} + 2\Delta^2 C_{\Sigma, \zeta}^2)}, \label{eq:cG1def} \\
		& \cG_2(C_{\Sigma, \zeta}, \Delta, m, B_{0}) = \bar B_{\ast} + \Delta C_{\Sigma, \zeta}\sqrt{8\bar B_{\ast}^2 + 2\cG_1(C_{\Sigma, \zeta}, \Delta, m, B_{0})^2 }.  \label{eq:cG2def}
	\end{align}
	Then on the set $\Omega$, for all $k \in \{0\} \cup [K]$ we have
	\begin{align}
		& \frac{1}{\sqrt{n}} \|\bE_k\|_F \leq \cG_1(C_{\Sigma, \zeta}, \Delta, m, B_{0}) \cdot \sqrt{\log n}, \label{eq:Ek} \\
		& \frac{1}{\sqrt{n}} \|\ba_k\|_2 \leq \cG_2(C_{\Sigma, \zeta}, \Delta, m, B_{0}) \cdot \sqrt{\log n}. \label{eq:ak}
	\end{align}
\end{lemma}

\begin{proof}
	We first prove \Cref{eq:Ek}. We denote by $\bX_{-i} \in \R^{(n - 1) \times (p + 1)}$ the matrix obtained by deleting the $i$-th row from $\bX$.
	By definition, 
	\begin{align*}
		\bX(\hat\bbeta_{k + 1} - \hat\bbeta_{k + 1, -i}) = & \bX( \hat\bbeta_k - \hat\bbeta_{k, -i}) + \frac{\delta_k (y_i - \bx_i^{\top} \hat\bbeta_k)}{n} \bX \bx_i - \frac{\delta_k}{n} \bX \sum_{j \neq i } \bx_j \bx_j^{\top} (\hat\bbeta_k - \hat\bbeta_{k, -i}) \\
		= & \bX( \hat\bbeta_k - \hat\bbeta_{k, -i}) + \frac{\delta_k (y_i - \bx_i^{\top} \hat\bbeta_k)}{n} \bX \bx_i - \frac{\delta_k}{n} \bX \bX_{-i}^{\top}\bX_{-i} (\hat\bbeta_k - \hat\bbeta_{k, -i}),
	\end{align*}
	which further implies 
	\begin{align*}
		& \|\bX(\hat\bbeta_{k + 1} - \hat\bbeta_{k + 1, -i})\|_2^2 \\
		 & \leq \left(1 + \delta_k C_{\Sigma, \zeta} \right)^2 \|\bX(\hat\bbeta_k - \hat\bbeta_{k, -i})\|_2^2 + \frac{\delta_k^2 (y_i - \bx_i^{\top} \hat\bbeta_k)^2}{n^2}\|\bX \bx_i\|_2^2 \\
		 & \quad  + \frac{2\delta_k(1 + \delta_k C_{\Sigma, \zeta}) \cdot |y_i - \bx_i^{\top} \hat\bbeta_k|}{n} \|\bX(\hat\bbeta_k - \hat\bbeta_{k, -i})\|_2 \cdot \|\bX \bx_i\|_2 \\
		 & \leq \left(1 + \delta_k C_{\Sigma, \zeta} \right)^2 \|\bX(\hat\bbeta_k - \hat\bbeta_{k, -i})\|_2^2 + \delta_k^2 C_{\Sigma, \zeta}^2 (y_i - \bx_i^{\top} \hat\bbeta_k)^2 \\
		 & \quad + \delta_k C_{\Sigma, \zeta}(1 + \delta_k C_{\Sigma, \zeta}) \cdot \big\{ (y_i - \bx_i^{\top} \hat\bbeta_k)^2 + \|\bX(\hat\bbeta_k - \hat\bbeta_{k, -i})\|_2^2 \big\} \\
		 & \leq \big( 1 + 3\delta_k C_{\Sigma, \zeta} + 2\delta_k^2 C_{\Sigma, \zeta}^2 \big) \cdot \|\bX(\hat\bbeta_k - \hat\bbeta_{k, -i})\|_2^2 + \big( \delta_k C_{\Sigma, \zeta} + 2\delta_k^2 C_{\Sigma, \zeta}^2 \big) \cdot (y_i - \bx_i^{\top} \hat\bbeta_{k})^2,
	\end{align*}
        where we make use of the fact that $\|\bX \bx_i\|_2 / n \leq C_{\Sigma, \zeta}$ on $\Omega$. 
	Putting together the above upper bound and Corollary \ref{cor:y-Xbeta}, then summing over $i \in [n]$, we obtain the following inequality: 
	\begin{align*}
		\|\bE_{k + 1}\|_F^2 \leq \big( 1 + 3\delta_k C_{\Sigma, \zeta} + 2\delta_k^2 C_{\Sigma, \zeta}^2 \big) \cdot \|\bE_k\|_F^2 + \big( \delta_k C_{\Sigma, \zeta} + 2\delta_k^2 C_{\Sigma, \zeta}^2 \big) \cdot \bar B_{\ast}^2 \log n. 
	\end{align*}
	Employing the standard induction argument, we can conclude that 
	\begin{align*}
		\frac{1}{n}\|\bE_k\|_F^2 \leq e^{3\Delta C_{\Sigma, \zeta} + 2\Delta^2 C_{\Sigma, \zeta}^2}(\Delta C_{\Sigma, \zeta} + 2\Delta^2 C_{\Sigma, \zeta}^2)\cdot\bar B_{\ast}^2 \log n = \cG_1(C_{\Sigma, \zeta}, \Delta, m, B_{0})^2 \log n
	\end{align*}
	for all $k \in \{0\} \cup [K]$. This completes the proof of \Cref{eq:Ek}. 
	
	Next, we prove \Cref{eq:ak}. By definition, 
	\begin{align*}
		y_i - \bx_i^{\top} \hat\bbeta_{k + 1, -i} = &\, y_i - \bx_i^{\top} \hat\bbeta_{k, -i} - \frac{\delta_k}{n} \sum_{j \neq i} (y_j - \bx_j^{\top} \hat\bbeta_{k, -i}) \bx_i^{\top} \bx_j \\
		= & \, y_i - \bx_i^{\top} \hat\bbeta_{k, -i}- \frac{\delta_k}{n} \sum_{j \neq i} (y_j - \bx_j^{\top} \hat\bbeta_{k}) \bx_i^{\top} \bx_j - \frac{\delta_k}{n} \sum_{j \neq i} \bx_i^{\top} \bx_j \bx_j^{\top}(\hat\bbeta_k - \hat\bbeta_{k, -i}). 
	\end{align*}
	We let $\bD = \diag\{(\|\bx_i\|_2^2 / n)_{i = 1}^n\} \in \R^{n \times n}$.  We denote by $a_{k, i}$ the $i$-th entry of $\ba_k$. From the above equality, we can deduce that
	\begin{align*}
		(a_{k+ 1, i} - a_{k, i})^2 \leq \frac{2\delta_k^2}{n^2} \Big( \sum_{j \neq i} (y_j - \bx_j^{\top} \hat\bbeta_{k}) \bx_i^{\top} \bx_j \Big)^2 + \frac{2\delta_k^2}{n^2} \Big( \sum_{j \neq i} \bx_i^{\top} \bx_j \bx_j^{\top} (\hat\bbeta_k - \hat\bbeta_{k, -i}) \Big)^2. 
	\end{align*}
	Summing over $i \in [n]$, we obtain
	\begin{align*}
		& \|\ba_{k + 1} - \ba_k\|_2^2 \\
		& \leq \frac{2\delta_k^2}{n^2}\big\| (\bX \bX^{\top} - n\bD)(\by - \bX \hat\bbeta_k)  \big\|_2^2 + \frac{2\delta_k^2}{n^2} \sum_{i = 1}^n \Big( \sum_{j \neq i} (\bx_i^{\top} \bx_j)^2 \Big) \cdot \Big( \sum_{j \neq i}\big( \bx_j^{\top}(\hat\bbeta_k - \hat\bbeta_{k, -i}) \big)^2 \Big) \\
		& \overset{(i)}{\leq}  8n\delta_k^2 C_{\Sigma, \zeta}^2 \cdot \bar B_{\ast}^2 \log n + 2\delta_k^2C_{\Sigma, \zeta}^2 \sum_{i = 1}^n \|\bX (\hat\bbeta_k - \hat\bbeta_{k, -i})\|_2^2 \\
		& \overset{(ii)}{\leq} 8n\delta_k^2 C_{\Sigma, \zeta}^2 \cdot \bar B_{\ast}^2 \log n + 2n\delta_k^2C_{\Sigma, \zeta}^2\cG_1(C_{\Sigma, \zeta}, \Delta, m, B_{0})^2 \cdot \log n \\
		& = \,  n \delta_k^2 \cdot \cG'(C_{\Sigma, \zeta}, \Delta, m, B_{0})^2 \cdot \log n, 
	\end{align*}
	where to derive~$(i)$, we employ the following established results: (1) On $\Omega$ we have $\|n\bD\|_{\op} \leq nC_{\Sigma, \zeta}$ and $\|\bX\bX^{\top}\|_{\op} \leq n C_{\Sigma, \zeta}$. (2) By Corollary \ref{cor:y-Xbeta}, on $\Omega$ we have $\|\by - \bX \hat\bbeta_k\|_2^2 / n \leq \bar B_{\ast}^2 \cdot \log n$. To derive~$(ii)$, we simply apply \Cref{eq:Ek}, which we have already proved. Therefore, by triangle inequality
	\begin{align*}
		\frac{1}{\sqrt{n}}\|\ba_{k + 1}\|_2 \leq & \frac{1}{\sqrt{n}}\|\ba_k\|_2 + \frac{1}{\sqrt{n}}\|\ba_{k + 1} - \ba_k\|_2 \leq  \frac{1}{\sqrt{n}}\|\ba_k\|_2 + \delta_k \cG'(C_{\Sigma, \zeta}, \Delta, m, B_{0}) \cdot \sqrt{\log n}. 
	\end{align*}
	By standard induction argument, we see that for all $k \in \{0\} \cup [K]$, 
	\begin{align*}
		\frac{1}{\sqrt{n}}\|\ba_k\|_2 \leq \bar B_{\ast} \sqrt{\log n} + \Delta \cG'(C_{\Sigma, \zeta}, \Delta, m, B_{0})\sqrt{\log n} = \cG_2(C_{\Sigma, \zeta}, \Delta, m, B_{0}) \cdot \sqrt{\log n},
	\end{align*}
	which concludes the proof of \Cref{eq:ak}. 
\end{proof}

\section{Proof of \Cref{thm:uniform-consistency-squared}}
\label{sec:uniform-consistency-proof-squared}

\bigskip

\ThmUniformConsistencySquared*

To better present our proof idea, we consider in this section the quadratic functional $\psi(y, u) = (y - u)^2$. A compact version of proof for general functional estimation can be found in Appendix \ref{sec:proof-thm:uniform-consistency-general}. 

\subsection{Proof schematic}
\label{sec:proof-thm:uniform-consistency}

A visual schematic for the proof of \Cref{thm:uniform-consistency-squared}
is provided in \Cref{fig:schematic-proof-thm:uniform-consistency}.

\begin{figure}[!ht]
    \centering
    \begin{tikzpicture}[node distance=3cm]
    \node (thm-uniform-consistency) [theorem] {\Cref{thm:uniform-consistency-squared}} ;
    \node (lemma-expectation-close) [lemma, below of=thm-uniform-consistency] {\Cref{lemma:expectation-close}} ;
    \node (lemma-projection-effects) [lemma, right of=lemma-expectation-close] {\Cref{lemma:projection-effects}} ;
    \node (lemma-cRk) [lemma, right of=lemma-projection-effects] {\Cref{lemma:cRk}} ;
    \node (lemma-Rloo) [lemma, right of=lemma-cRk, node distance=3cm] {\Cref{lemma:Rloo}} ;
    \node (lemma-gradient-upper-bound) [lemma, below of=lemma-Rloo, node distance=3cm] {\Cref{lemma:gradient-upper-bound}} ;
    \node (lemma-concentration) [lemma, right of=lemma-gradient-upper-bound] {\Cref{lemma:concentration}} ;
    \node (sec-proof-lemma-gradient-upper-bound) [misc, below of=lemma-gradient-upper-bound, node distance=3cm] {\Cref{sec:proof-lemma:gradient-upper-bound}} ;
    \node (lemma-norm-theta) [lemma, below of=lemma-cRk] {\Cref{lemma:norm-theta}} ;
    \node (lemma-op-subGaussian) [lemma, below of=lemma-projection-effects] {\Cref{lemma:op-subGaussian}} ;
    \node (lemma-op-Sigma) [lemma, below of=lemma-op-subGaussian] {\Cref{lemma:op-Sigma}} ;
    \node (prop-T2-prop-Gozlan) [proposition, left of=lemma-op-Sigma, node distance = 4cm] {\Cref{prop:Gozlan}} ;
    \draw [arrow] (lemma-Rloo) -- (thm-uniform-consistency) ;
    \draw [arrow] (lemma-cRk) -- (thm-uniform-consistency) ;
    \draw [arrow] (lemma-projection-effects) -- (thm-uniform-consistency) ;
    \draw [arrow] (lemma-expectation-close) -- (thm-uniform-consistency) ;
    \draw [arrow] (lemma-gradient-upper-bound) -- (lemma-Rloo) ;
    \draw [arrow] (sec-proof-lemma-gradient-upper-bound) -- (lemma-gradient-upper-bound) ;
    \draw [arrow] (lemma-concentration) -- (lemma-Rloo) ;
    \draw [arrow] (sec-proof-lemma-gradient-upper-bound) -- (lemma-cRk) ;
    \draw [arrow] (lemma-norm-theta) -- (lemma-expectation-close);
    \draw [arrow] (lemma-norm-theta) -- (lemma-cRk) ;
    \draw [arrow] (lemma-op-subGaussian) -- (lemma-projection-effects) ;
    \draw [arrow] (lemma-op-subGaussian) -- (lemma-expectation-close) ;
    \draw [arrow] (lemma-op-Sigma) -- (lemma-op-subGaussian) ;
    \draw [arrow] (prop-T2-prop-Gozlan) -- (lemma-op-Sigma) ;
    \end{tikzpicture}
    \caption{Schematic for the proof of \Cref{thm:uniform-consistency-squared} 
    }
    \label{fig:schematic-proof-thm:uniform-consistency}
\end{figure}

 \subsection{Proof of \Cref{lemma:Rloo}} \label{sec:proof-lemma:Rloo}

\bigskip

\LemmaRloo*

\begin{proof}
 We claim that \Cref{lemma:gradient-upper-bound} can be used to show $\tilde{f}_k$ is Lipschitz continuous. More precisely, for $\bW, \bW' \in \R^{n(p + 2)}$, it holds that
\begin{align*}
	\left| \tilde{f}_k(\bW) - \tilde{f_k}(\bW')\right| = &  \left| f_k (h(\bW)) - f_k (h(\bW')) \right| \\
	\leq & \frac{K \xi(C_{\Sigma, \zeta}, \Delta, m, B_0) \cdot \log n}{\sqrt{n}} \cdot \|h(\bW) - h(\bW')\|_F \\
	\leq & \frac{K \xi(C_{\Sigma, \zeta}, \Delta, m, B_0) \cdot \log n}{\sqrt{n}} \cdot \|\bW - \bW'\|_F. 
\end{align*}
Namely, $\tilde{f}_k$ is $n^{-1/2} K \xi(C_{\Sigma, \zeta}, \Delta, m, B_0)\cdot \log n$-Lipschitz continuous for all $k \in \{0\} \cup [K]$. Applying \Cref{lemma:concentration}, we conclude that
\begin{align*}
	\P\left( \left| \tilde f_k(\bw_1, \cdots, \bw_n) - \E[\tilde f_k(\bw_1, \cdots, \bw_n)] \right| \geq \frac{2\sigma_{\LSI} L K \xi(C_{\Sigma, \zeta}, \Delta, m, B_0)\cdot (\log n)^{3/2}}{\sqrt{n}}\right) \leq  2C_{\LSI}n^{-2}.
\end{align*}
Note that on the set $\Omega$ we have $\tilde{f}_k(\bw_1, \cdots, \bw_n) = f_k(\bw_1, \cdots, \bw_n) = \hR^\loo(\hat{\bbeta}_k)$ for all $k \in \{0\} \cup [K]$. This completes the proof of the lemma. 
\end{proof}

\subsection[Proof of \Cref{lemma:cRk}]{Proof of \Cref{lemma:cRk}}\label{sec:proof-lemma:cRk}

\bigskip

\LemmacRk*

\begin{proof}
For $s \in [n]$, direct computation gives
\begin{align*}
	& \nabla_{\bx_s} R(\hat\bbeta_k) = 2\hat\bbeta_k^{\top} \tilde\bSigma \nabla_{\bx_s} \hat\bbeta_k  - 2 \hat\btheta^{\top} \nabla_{\bx_s} \hat\bbeta_k, \\
	& \frac{\partial}{\partial_{y_s}} R(\hat\bbeta_k) = 2 \hat\bbeta_k^{\top} \tilde\bSigma \frac{\partial}{\partial_{y_s}} \hat\bbeta_k - 2 \btheta^{\top}\frac{\partial}{\partial_{y_s}} \hat\bbeta_k,
\end{align*}
where 
\begin{align*}
	\btheta = \E[y_{\new} \bx_{\new}] \in \R^{p + 1}, \qquad \tilde\bSigma = \left[ \begin{array}{cc}
		\bSigma & \mathbf{0}_p\\
		\mathbf{0}_p^{\top} & 1
	\end{array} \right] \in \R^{(p + 1) \times (p + 1)}.
\end{align*}
By definition, 
\begin{align*}
	& \nabla_{\bx_s} \hat\bbeta_{k + 1} = \nabla_{\bx_s} \hat\bbeta_k  - \delta_k \hat\bSigma \cdot \nabla_{\bx_s}  \hat\bbeta_k + \frac{\delta_k}{n}(y_s - \bx_s^{\top} \hat\bbeta_k) \id_{p + 1} - \frac{\delta_k}{n} \bx_s \hat\bbeta_k^{\top}, \\
	& \frac{\partial}{\partial_{y_s}} \hat\bbeta_{k + 1} = \frac{\partial}{\partial_{y_s}} \hat\bbeta_k - \delta_k \hat\bSigma \frac{\partial}{\partial_{y_s}} \hat\bbeta_k + \frac{\delta_k}{n} \bx_s. 
\end{align*}
Standard induction argument leads to the following decomposition:
\begin{align*}
	& \nabla_{\bx_s} \hat\bbeta_{k + 1} = \sum_{k' = 1}^k \bH_{k', k} \cdot \Big(\frac{\delta_{k'}}{n}(y_s - \bx_s^{\top} \hat\bbeta_{k'}) \id_{p + 1} - \frac{\delta_{k'}}{n} \bx_s \hat\bbeta_{k'}^{\top} \Big),  \\
	& \frac{\partial}{\partial_{y_s}} \hat\bbeta_{k + 1} = \sum_{k' = 1}^k \bH_{k', k} \cdot \frac{\delta_{k'}}{n} \bx_s,
\end{align*}
where $\bH_{k', r} = \prod_{j = k' + 1}^{r} \bM_{k' + 1 + r - j}$ and $\bM_j = \id_{p + 1} - \delta_j \hat \bSigma$ are defined in \Cref{lemma:gradx}. 
Combining all these arguments, we arrive at the following equations: 
\begin{align*}
	& \bv^{\top} \nabla_{\bx_s} \hat\bbeta_{k + 1} = \sum_{k' = 1}^k \bv^{\top} \bH_{k', k} \cdot \frac{\delta_{k'}}{n}(y_s - \bx_s^{\top} \hat\bbeta_{k'}) - \sum_{k' = 1}^k \frac{\delta_{k'}}{n} \bx_s^{\top} \bH_{k', k} \bv \hat\bbeta_{k'}^{\top}, \\
	& \bv^{\top} \frac{\partial}{\partial_{y_s}} \hat\bbeta_{k + 1} = \sum_{k' = 1}^k \frac{\delta_{k'}}{n}  \bx_s^{\top} \bH_{k', k} \bv.
\end{align*}
The above equations hold for all $\bv \in \{\btheta, \tilde\bSigma \hat\bbeta_{k + 1}\}$. Recall that $\btheta = \E[y_0 \bx_0]$. This further implies that
\begin{align*}
	& \nabla_{\bX} R (\hat\bbeta_{k + 1}) \\
	 = & \sum_{k' = 1}^k \frac{2\delta_{k'}}{n} \cdot \left\{   (\by - \bX \hat\bbeta_{k'}) \hat\bbeta_{k + 1}^{\top}\tilde \bSigma  \bH_{k', k} - \bX\bH_{k, k'} \tilde\bSigma \hat\bbeta_{k + 1} \hat\bbeta_{k'}^{\top} - (\by - \bX \hat\bbeta_{k'}) \btheta^{\top}  \bH_{k', k} + \bX\bH_{k, k'} \btheta \hat\bbeta_{k'}^{\top} \right\} \\
	 & \nabla_{\by} \cR(\hat\bbeta_{k + 1}) = \sum_{k' = 1}^k \frac{2\delta_{k'}}{n} \cdot \left\{ \bX \bH_{k, k'} \tilde\bSigma \hat\bbeta_{k + 1} - \bX \bH_{k, k'}  \btheta  \right\}. 
\end{align*}
Recall that $B_{\ast}$ is defined in \Cref{eq:BstarB}. 
Invoking triangle inequality, we obtain that on $\Omega$, 
\begin{align*}
	 \|\nabla_{\bX} R (\hat\bbeta_{k + 1})\|_F \leq & \sum_{k' = 1}^k \frac{2\delta_{k'}}{n} \cdot \left\{ \|\by - \bX \hat\bbeta_{k'}\|_2 \cdot (\|\hat\bbeta_{k + 1}\|_2 \|\tilde \bSigma\|_{\op} + \|\btheta\|_2) \cdot \|\bH_{k', k}\|_{\op} \right. \\
	&\left. + \|\bX\|_{\op} \cdot \|\bH_{k, k'}\|_{\op} \cdot  (\|\hat\bbeta_{k + 1}\|_2 \|\tilde \bSigma\|_{\op} + \|\btheta\|_2) \cdot  \|\hat\bbeta_{k'}\|_2\right\} \\
	\leq & \frac{2\Delta e^{\Delta C_{\Sigma, \zeta}} \cdot \sqrt{\log n}}{\sqrt{n}} \cdot  \Big(\bar{B}_{\ast} + C_{\Sigma, \zeta}^{1/2}B_{\ast}\Big)\Big(B_{\ast}(\sigma_{\Sigma} + 1) \cdot \sqrt{\log n} + (\sigma_{\Sigma}^{1/2} + 1) m_2^{1/2}\Big),
 \end{align*}
 where the inequality follows by invoking \Cref{lemma:norm-theta}
 to upper bound $\| \btheta \|_2$.
Also, by \Cref{lemma:gradx} we know that $\|\bH_{k', r}\|_{\op} \leq e^{\Delta C_{\Sigma, \zeta}}$. 
 Similarly, we obtain
 \begin{align*}
	\|\nabla_{\by} R(\hat\bbeta_{k + 1})\|_2 \leq & \sum_{k' = 1}^k \frac{2\delta_{k'}}{n} \cdot \left\{ \|\bX\|_{\op} \cdot \| \bH_{k, k'} \|_{\op} \cdot \|\tilde\bSigma\|_{\op} \cdot \| \bbeta_{k + 1}\| + \|\bX\|_{\op} \cdot \| \bH_{k, k'}\|_{\op} \cdot \| \btheta\|_2 \right\}  \\
	\leq & \frac{2\Delta e^{\Delta C_{\Sigma, \zeta}}C_{\Sigma, \zeta}^{1/2}}{\sqrt{n}} \cdot \Big(B_{\ast}(\sigma_{\Sigma} + 1) + (\sigma_{\Sigma}^{1/2} + 1) m_2^{1/2}\Big) \cdot \sqrt{\log n}. 
\end{align*} 
The above inequalities give an upper bound for $\|\nabla_{\bW} R(\hat\bbeta_{k + 1})\|_2$ on $\Omega$. The rest parts of the proof are similar to the proof of \Cref{lemma:Rloo} given \Cref{lemma:gradient-upper-bound}.
\end{proof}

\subsection{Proof of \Cref{lemma:projection-effects}}\label{sec:proof-lemma:projection-effects}

\bigskip

\LemmaProjectionEffects*

\begin{proof}
We shall first upper bound the fourth moments  $\E[(y_\new - \bx_\new^{\top} \hat\bbeta_{k, -1})^4]$ and $\E[(y_{\new} - \bx_{\new}^{\top} \hat\bbeta_k)^4]$. By standard induction, it is not hard to see that for all $0 \leq k \leq K$ and $i \in [n]$, 
\begin{align}
	& \|\hat\bbeta_k\|_2  \leq \exp(\Delta \|\hat\bSigma\|_{\op}) \cdot \Big(B_0 + \Delta n^{-1} \|\bX \|_{\op} \cdot \| \by\|_2 \Big), \label{eq:betak-norm} \\
	& \|\hat\bbeta_{k, -i}\|_2  \leq \exp(\Delta \|\hat\bSigma\|_{\op}) \cdot \Big(B_0 + \Delta n^{-1} \|\bX \|_{\op} \cdot \| \by\|_2 \Big). \label{eq:betak-i-norm}  
\end{align}
For technical reasons that will become clear soon, we need to upper bound the expectations of $\|\hat\bbeta_k\|_2$ and $\|\hat\bbeta_{k, -i}\|_2$. To this end, we find it useful to show $\|\hat{\bSigma}\|_{\op}^{1/2}$ is sub-Gaussian. 
Next, we will employ \Cref{lemma:op-subGaussian} to upper bound $\E[(y_\new - \bx_\new^{\top} \bbeta_{k, -1})^4]$ and $\E[(y_{\new} - \bx_{\new}^{\top} \bbeta_k)^4]$. Invoking the Cauchy-Schwartz inequality and triangle inequality, we obtain that for $n \geq N(\sigma_{\Sigma}, \zeta, B_0, m_8, \Delta)$,
\begin{align*}
	\E[(y_\new - \bx_\new^{\top} \hat\bbeta_{k, -1})^4] 
 & \leq \E[\|(y_0, \bx_0^{\top} \hat\bbeta_{k, -1})\|_2^4]  \\
 & \leq 8 \E[y_1^4] + 8 \E[(\bx_1^{\top} \hat\bbeta_{k, -1})^4] = 8m_4 + 8 \E[((\bz_1^{\top}, 1) \tilde\bSigma^{1/2} \hat\bbeta_{k, -1} )^4] \\
	& \overset{(i)}{\leq} 8m_4 + C_{z} \E[\|\tilde\bSigma^{1/2} \hat\bbeta_{k, -1}\|_2^4] \\
	& \overset{(ii)}{\leq} \cH(\sigma_{\Sigma}, \zeta, B_0, m_8, \Delta)^2
\end{align*}
where $C_z > 0$ is a constant that depends only on $\mu_z$, $\cH(\sigma_{\Sigma}, \zeta, B_0, m_8, \Delta) \in \R_+$ and $N(\sigma_{\Sigma}, \zeta, B_0, m_8, \Delta) \in \NN_+$ depend only on $(\sigma_{\Sigma}, \zeta, B_0, m_8, \Delta)$. 
To derive \emph{(i)} we use the following facts: (1) $\mu_z$ has zero expectation; (2) $\bz_1$ is independent of $\tilde\bSigma^{1/2} \bbeta_{k, -1}$. To derive \emph{(ii)} we apply \Cref{eq:betak-i-norm} and \Cref{lemma:op-subGaussian}. Similarly, we can show that for $n \geq N(\sigma_{\Sigma}, \zeta, B_0, m_8, \Delta)$,
\begin{align}\label{eq:four-moment}
	\E[(y_{\new} - \bx_{\new}^{\top} \bbeta_{k})^4] \leq \E[\|(y_0, \bx_0^{\top} \bbeta_k)\|_2^4] \leq &\cH(\sigma_{\Sigma}, \zeta, B_0, m_8, \Delta)^2. 
\end{align}
Finally, we are ready to establish \Cref{eq:truncation-expectation}. By the Cauchy-Schwartz inequality, 
\begin{align*}
	 & \Big| \E[r_k(\bw_1, \cdots, \bw_n)] - \E[\tilde r_k(\bw_1, \cdots, \bw_n)] \Big| \leq \P(\Omega^c)^{1/2} \E[(y_{\new} - \bx_{\new}^{\top}\hat\bbeta_{k})^4]^{1/2}, \\
	 & \Big| \E[f_k(\bw_1, \cdots, \bw_n)] - \E[\tilde f_k(\bw_1, \cdots, \bw_n)] \Big| \leq \P(\Omega^c)^{1/2}\E[(y_1 - \bx_1^{\top} \hat\bbeta_{k, -1})^4]^{1/2}, 
\end{align*}
which for $n \geq N(\sigma_{\Sigma}, \zeta, B_0, m_8, \Delta)$ are upper bounded by 
\begin{align*}
	\left( 2(n + p)^{-1} + n^{-1}m_4 + 2C_{\LSI}n^{-2} \right)^{1/2}\cH(\sigma_{\Sigma}, \zeta, B_0, m_8, \Delta).  
\end{align*}
The above upper bound goes to zero as $n, p \to \infty$, thus completing the proof of the lemma. 
\end{proof}

\subsection{Proof of \Cref{lemma:expectation-close}}
\label{sec:proof-lemma:expectation-close}

\bigskip

\LemmaExpectationClose*

\begin{proof}
By \Cref{eq:betak-norm}, \Cref{eq:betak-i-norm}, and \Cref{lemma:op-subGaussian}, we know that there exists a constant $C''$ that depends only on $(\sigma_{\Sigma}, \zeta, \Delta, B_0, m_2)$, such that 
\begin{align}\label{eq:Cprime}
    \max\left\{\E[\|\bbeta_k\|_2^2]^{1/2}, \E[\|\bbeta_{k, -i}\|_2^2]^{1/2} \right\}\leq C''. 
\end{align}

To show this result, we first prove that $\hat\bbeta_{k} \approx \hat\bbeta_{k, -i}$. By definition, 
\begin{align*}
	\hat\bbeta_{k + 1} - \hat\bbeta_{k + 1, -i} = \big(\id_p - \delta_k \hat\bSigma \big) \cdot \big(\hat\bbeta_k - \hat\bbeta_{k, -i} \big) + \frac{\delta_k}{n} y_i \bx_i - \frac{\delta_k}{n} \bx_i \bx_i^{\top} \hat\bbeta_{k, -i}. 
\end{align*}
Invoking the triangle and Cauchy-Schwartz inequalities, we conclude that 
\begin{align*}
	 & \|\hat\bbeta_{k + 1} - \hat\bbeta_{k + 1, -i}\|_2^2 \\
	  & \leq (1 + \delta_k \|\hat\bSigma\|_{\op})^2 \|\hat\bbeta_k - \hat\bbeta_{k, -i}\|_2^2 + \frac{\delta_k^2}{n^2}(y_i - \bx_i^{\top} \hat\bbeta_{k, -i})^2 \cdot \|\bx_i\|_2^2 \\
	 & \quad + \frac{2\delta_k(1 + \delta_k \|\hat\bSigma\|_{\op})}{n} \cdot \|\hat\bbeta_k - \hat\bbeta_{k, -i}\|_2 \cdot |y_i - \bx_i^{\top} \hat\bbeta_{k, -i}| \cdot \|\bx_i\|_2 \\
	 & \leq (1 + \delta_k \|\hat\bSigma\|_{\op})(1 + 2\delta_k \|\hat\bSigma\|_{\op}) \|\hat\bbeta_k - \hat\bbeta_{k, -i}\|_2^2 + \frac{\delta_k(1 + \delta_k + \delta_k \|\hat\bSigma\|_{\op})}{n^2}(y_i - \bx_i^{\top} \hat\bbeta_{k, -i})^2 \cdot \|\bx_i\|_2^2. 
\end{align*}
By induction, 
\begin{align*}
	\|\hat\bbeta_{k + 1} - \hat\bbeta_{k + 1, -i}\|_2^2 \leq \sum_{j = 1}^k \frac{\delta_j \exp(3\Delta \|\hat\bSigma\|_{\op} + \Delta)}{n^2} \cdot (y_i - \bx_i^{\top}\hat \bbeta_{k, -i})^2 \cdot \|\bx_i\|_2^2.
\end{align*}
By the \Holder's inequality and \Cref{lemma:op-subGaussian}, we see that for $n \geq 12 \Delta \tilde C_0^2 + 1$ 
\begin{align}\label{eq:diff-beta}
\begin{split}
	& \E\left[ \|\hat\bbeta_{k + 1} - \hat\bbeta_{k + 1, -i}\|_2^2 \right] \\
	& \leq \sum_{j = 1}^k \frac{\delta_j}{n^2} \cdot \E[(y_i - \bx_i^{\top} \hat\bbeta_{k, -i})^4]^{1/2} \cdot \E[\|\bx_i\|_2^8]^{1/4} \cdot \E[\exp(12\Delta \|\hat\bSigma\|_{\op} + 4\Delta)]^{1/4} \\
	& \leq \frac{\Delta e^{\Delta} \sigma_{\Sigma} \cH(\sigma_{\Sigma}, \zeta, B_0, m_8, \Delta)}{n} \cdot \cE(\tilde C_0, \zeta, 12 \Delta)^{1/4}. 
\end{split}
\end{align}
In addition, direct computation gives
\begin{align*}
	& \E[r_k(\bw_1, \cdots, \bw_n)] = m_2 + \E[\hat\bbeta_k^{\top} \bSigma \hat\bbeta_k] + 2 \langle \E[\hat\bbeta_k], \btheta \rangle, \\
	& \E[f_k(\bw_1, \cdots, \bw_n)] = m_2 + \E[\hat\bbeta_{k, -i}^{\top} \bSigma \hat\bbeta_{k, -i}] + 2 \langle \E[\hat\bbeta_{k, -i}], \btheta \rangle, 
\end{align*}
where we recall that $\btheta = \E[y_{\new} \bx_{\new}]$. By \Cref{lemma:norm-theta} we know that $\|\btheta\|_2 \leq (\sigma_{\Sigma}^{1/2} + 1)m_2^{1/2}$. Therefore, 
\begin{align*}
   &  \left|\E[r_k(\bw_1, \cdots, \bw_n)] - \E[f_k(\bw_1, \cdots, \bw_n)]  \right| \\
   & \leq 2\|\btheta\|_2 \cdot  \E\left[ \|\hat\bbeta_{k + 1} - \hat\bbeta_{k + 1, -i}\|_2^2 \right]^{1/2} + \sigma_{\Sigma} \E\left[ \|\hat\bbeta_{k } - \hat\bbeta_{k, -i}\|_2^2 \right]^{1/2} \cdot \Big( \E\big[\|\hat\bbeta_{k }\|_2^2 \big]^{1/2} + \E\big[\|\hat\bbeta_{k,-i }\|_2^2 \big]^{1/2} \Big),
\end{align*}
which by \Cref{eq:Cprime,eq:diff-beta} goes to zero as $n, p \to \infty$. Furthermore, the convergence is uniform for all $k \in \{0\} \cup [K]$. 
This completes the proof of the lemma.
\end{proof}

\section{Proof of \Cref{lemma:gradient-upper-bound}}\label{sec:proof-lemma:gradient-upper-bound}

\bigskip

\LemmaGradientUpperBound*

\subsection{Proof schematic}
\label{sec:outline-proof-lemma:gradient-upper-bound}

We divide the proof of the lemma into two parts: upper bounding $\|\nabla_{\bX} \hR^\loo(\hat{\bbeta}_k)\|_F$ and $\|\nabla_{\bX} \hR^\loo(\hat{\bbeta}_k)\|_F$.
A visual schematic for the proof of \Cref{lemma:gradient-upper-bound}
is provided in \Cref{fig:schematic-proof-lemma:gradient-upper-bound}.

\begin{figure}
    \centering
    \begin{tikzpicture}[node distance=3cm]
    \node (lemma-gradient-upper-bound) [lemma] {\Cref{lemma:gradient-upper-bound}} ;
    \node (lem-bound-norm-gradient-wrt-response) [lemma, below of=lemma-gradient-upper-bound] {\Cref{lem:bound-norm-gradient-wrt-response}} ;
    \node (lem-bound-norm-gradient-wrt-features) [lemma, right of=lem-bound-norm-gradient-wrt-response] {\Cref{lem:bound-norm-gradient-wrt-features}} ;
    \node (lemma-y-xbeta) [lemma, below of=lem-bound-norm-gradient-wrt-response] {\Cref{lemma:y-xbeta}} ;
    \node (lemma-gradx) [lemma, right of=lemma-y-xbeta] {\Cref{lemma:gradx}} ;
    \node (lemma-V) [lemma, right of=lemma-gradx] {\Cref{lemma:V}} ;
    \node (cor-y-Xbeta) [lemma, below of=lemma-y-xbeta] {\Cref{cor:y-Xbeta}} ;
    \node (lemma-upper-bound-beta) [lemma, right of=cor-y-Xbeta] {\Cref{lemma:upper-bound-beta}} ;
    \draw [arrow] (lem-bound-norm-gradient-wrt-features) -- (lemma-gradient-upper-bound) ;
    \draw [arrow] (lem-bound-norm-gradient-wrt-response) -- (lemma-gradient-upper-bound) ;
    \draw [arrow] (lemma-y-xbeta) -- (lem-bound-norm-gradient-wrt-response) ;
    \draw [arrow] (lemma-gradx) -- (lem-bound-norm-gradient-wrt-features) ; 
    \draw [arrow] (lemma-gradx) -- (lem-bound-norm-gradient-wrt-response) ; 
    \draw [arrow] (lemma-V) -- (lem-bound-norm-gradient-wrt-features) ;
    \draw [arrow] (cor-y-Xbeta) -- (lemma-y-xbeta) ;
    \draw [arrow] (lemma-upper-bound-beta) -- (cor-y-Xbeta) ;
    \end{tikzpicture}
    \caption{Schematic for the proof of \Cref{lemma:gradient-upper-bound} 
    }
    \label{fig:schematic-proof-lemma:gradient-upper-bound}
\end{figure}

\subsection[Upper bounding norm LOOCV gradient]{Upper bounding $\|\nabla_{\bX} \hR^\loo(\hat{\bbeta}_k)\|_F$}
\label{sec:proof-lem:bound-norm-gradient-wrt-features}

We start with the most challenging part, namely, upper bounding $\|\nabla_{\bX} \hR^\loo(\hat{\bbeta}_k)\|_F$.
We will show the following:

\begin{lemma}
    [Bounding norm of gradient with respect to features]
    \label{lem:bound-norm-gradient-wrt-features}
    On the set $\Omega$,
    for all $k \in \{0\} \cup [K]$, 
	\begin{align*}
		\|\nabla_{\bX} \hR^\loo(\hat{\bbeta}_k)\|_F & \leq \frac{2B_{\ast}\cG_2(C_{\Sigma, \zeta}, \Delta, m, B_{0})\log n}{\sqrt{n}} + \frac{2\Delta K e^{2\Delta C_{\Sigma, \zeta}} C_{\Sigma, \zeta}B_{\ast} \log n}{\sqrt n} \cdot \cG_2(C_{\Sigma, \zeta}, \Delta, m, B_{0}) \\
		& \quad + \frac{2\Delta Ke^{2\Delta C_{\Sigma, \zeta}}C_{\Sigma, \zeta}^{1/2}\bar{B}_{\ast} \log n}{\sqrt{n}} \cdot \cG_2(C_{\Sigma, \zeta}, \Delta, m, B_{0}). 
	\end{align*}
        In the above equation, we recall that $B_{\ast}$ is defined in \Cref{eq:BstarB}, $\bar{B}_{\ast}$ is defined in \Cref{eq:BstarBbar}, and $\cG_2(C_{\Sigma, \zeta}, \Delta, m, B_{0})$ is defined in \Cref{eq:cG2def}. 
\end{lemma}

\begin{proof}
We prove \Cref{lem:bound-norm-gradient-wrt-features} in the remainder of this section. 
For $s \in [n]$ and $k \in [K]$, we can compute $\nabla_{\bx_s} \hR^\loo(\hat{\bbeta}_k)$, which takes the following form: 
\begin{align}\label{eq:gradxs}
	\nabla_{\bx_s} \hR^\loo(\hat{\bbeta}_k) = - \frac{2}{n}(y_s - \bx_s^{\top} \hat\bbeta_{k, -s}) \hat\bbeta_{k, -s}^{\top} - \frac{2}{n} \sum_{i = 1}^n (y_i - \bx_i^{\top} \hat\bbeta_{k, -i})\bx_i^{\top} \nabla_{\bx_s} \hat\bbeta_{k, -i}. 
\end{align}
The above formula suggests that we should analyze the Jacobian matrix $\nabla_{\bx_s} \hat\bbeta_{k, -i}$, which can be done recursively. More precisely, the following update rule is a direct consequence of the gradient descent update rule: 
\begin{align*}
	& \nabla_{\bx_s} \hat\bbeta_{k + 1, -i} \\
	&= \nabla_{\bx_s} \hat\bbeta_{k, -i} + \frac{\delta_k \mathbbm{1}\{i \neq s\}}{n} (y_s - \bx_s^{\top} \hat\bbeta_{k, -i}) \id_{p + 1} - \frac{\delta_k \mathbbm{1}\{i \neq s\}}{n} \bx_s\hat\bbeta_{k, -i}^{\top} - \frac{\delta_k}{n} \sum_{j \neq i} \bx_j \bx_j^{\top} \nabla_{\bx_s} \hat\bbeta_{k, -i} \\
	&= \left(\id_{p + 1} - \delta_k \hat\bSigma \right) \cdot \nabla_{\bx_s} \hat\bbeta_{k, -i}  + \frac{\delta_k}{n} \bx_i \bx_i^{\top} \nabla_{\bx_s} \hat\bbeta_{k, -i} +  \mathbbm{1}\{i \neq s\} \cdot \left\{ \frac{\delta_k }{n} (y_s - \bx_s^{\top} \hat\bbeta_{k, -i}) \id_{p + 1} - \frac{\delta_k}{n} \bx_s\hat\bbeta_{k, -i}^{\top} \right\}.
\end{align*}
Note that the above process is initialized at $\nabla_{\bx_s} \hat\bbeta_{0, -i} = \mathbf{0}_{(p + 1) \times (p + 1)}$. Clearly when $i = s$, the Jacobian $\nabla_{\bx_s} \hat\bbeta_{k, -i}$ remains zero for all $k$ that is concerned, and we automatically get an upper bound for $\|\nabla_{\bx_s} \hat\bbeta_{k, -i}\|_2$.  

In what follows, we focus on the non-trivial case $i \neq s$. For this part, we will mostly fix $i$ and $s$, and ignore the dependency on $(i, s)$ when there is no confusion. Note that we can reformulate the Jacobian update rule as follows:  
\begin{align}\label{eq:update-gradx}
	\nabla_{\bx_s} \hat\bbeta_{k + 1, -i} = \bM_k \nabla_{\bx_s} \hat\bbeta_{k, -i} + \bM_{k, i} \nabla_{\bx_s} \hat\bbeta_{k, -i} + \frac{\delta_k }{n} (y_s - \bx_s^{\top} \hat\bbeta_{k, -i}) \id_{p + 1} - \frac{\delta_k}{n} \bx_s\hat\bbeta_{k, -i}^{\top}, 
\end{align}
where $\bM_k = \id_{p + 1} - \delta_k \hat\bSigma$ and $\bM_{k, i} = \delta_k \bx_i \bx_i^{\top} / n$. By induction, it is not hard to see that for all $0 \leq k \leq K - 1$, the matrix $\nabla_{\bx_s} \hat\bbeta_{k + 1, -i} - \bR_0^{(k)}$ can be expressed as the sum of terms that take the form
\begin{align*}
	\left( \prod_{j = 1}^{k - k'} \bR_{k +1  -j} \right) \bR_0^{(k')} ,
\end{align*}
where $k' \in \{0\} \cup [k - 1]$, $\bR_0^{(k')} = \frac{\delta_{k'} }{n} (y_s - \bx_s^{\top} \hat\bbeta_{k', -i}) \id_{p + 1} - \frac{\delta_{k'}}{n} \bx_s \hat\bbeta_{k', -i}^{\top}$, and  $\bR_j$ is either $\bM_j$ or $\bM_{j, i}$. %

To put it formally, we summarize this result as the following lemma:
\begin{lemma}\label{lemma:gradx}
	For $i, s \in [n]$ with $i \neq s$ and all $k \in \{0\} \cup [K - 1]$, it holds that
	\begin{align*}
		\bx_i^{\top}\nabla_{\bx_s} \hat\bbeta_{k + 1, -i} = \sum_{k' = 0}^k \sum_{r = k'}^k c_{i, k,  k', r} \bx_i^{\top} \bH_{k', r} \cdot \left( \frac{\delta_{k'}}{n} (y_s - \bx_s^{\top} \hat\bbeta_{k', -i}) \id_{p + 1} - \frac{\delta_{k'}}{n} \bx_s \hat\bbeta_{k', -i}^{\top} \right),
	\end{align*}
	where $c_{i, k, k', r} \in \R$ and $\bH_{k', r} = \prod_{j = k' + 1}^{r} \bM_{k' + 1 + r - j}$. We adopt the convention that $\bH_{k', k'} = \id_{p + 1}$. Furthermore,   on the set $\Omega$, it holds that 
	\begin{align}\label{eq:cH}
		\|\bH_{k', r}\|_{\op} \leq e^{\Delta C_{\Sigma, \zeta}}, \qquad  \|c_{i, k, k', r} \bH_{k', r}\|_{\op} \leq e^{2\Delta C_{\Sigma, \zeta}}. 
	\end{align}
\end{lemma}
\begin{proof}[Proof of \Cref{lemma:gradx}]
To derive the first inequality in \Cref{eq:cH}, we simply notice that 
\begin{align*}
	\|\bH_{k', r}\|_{\op} \leq \prod_{j = k' + 1}^{r} \|\bM_{k' + 1 + r - j} \|_{\op} \leq \prod_{j = k' + 1}^{r} (1 + \delta_{k' + 1 + r - j} C_{\Sigma, \zeta}) \leq e^{\Delta C_{\Sigma, \zeta}}.
\end{align*}
We next prove the second inequality in \Cref{eq:cH}. 
As discussed before, $\bx_i^{\top}\nabla_{\bx_s} \bbeta_{k + 1, -i} - \bx_i^{\top} \bR_0^{(k)}$ can be expressed as the sum of terms that take the form 
	\begin{align*}
		\bx_i^{\top} \left( \prod_{j = 1}^{k - k'} \bR_{k +1  -j} \right) \bR_0^{(k')},\end{align*}
	with $k'$ ranging from 0 to $k - 1$. The subtracting $\bx_i^{\top}\bR_0^{(k)}$ part implies that we should set $c_{i, k, k, k} = 1$ and $\bH_{k, k} = \id_{p + 1}$.
	
	We then study $c_{i, k, k', r}$ in general. For this purpose, we analyze each summand.  
	 Without loss, we let $\bR_{j_{\ast}}$ be the  last matrix in the sequence $(\bR_{k + 1 - j})_{j = 1}^{k - k'}$ that takes the form $\bM_{j_{\ast}, i}$. Then 
	\begin{align*}
		\bx_i^{\top} \left( \prod_{j = 1}^{k - k'} \bR_{k + 1 - j} \right) \bR_0^{(k')} & =  \bx_i^{\top} \left(\prod_{j = 1}^{k - j_{\ast}} \bR_{k + 1 - j} \right) \cdot \frac{\delta_{j_{\ast}}}{n} \bx_i \bx_i^{\top} \cdot \left( \prod_{j = k - j_{\ast} + 2}^{k - k'}  \bR_{k + 1 - j} \right) \bR_0^{(k')} \\
		& = \frac{\delta_{j_{\ast}}}{n} \bx_i^{\top} \left(\prod_{j = 1}^{k - j_{\ast}} \bR_{k + 1 - j}\right) \bx_i \bx_i^{\top} \bH_{k', j_{\ast} - 1} \bR_0^{(k')}. 
	\end{align*}
	This implies that
	\begin{align*}
		c_{i, k, k', j_{\ast} - 1} = \sum_{\bR_{k + 1 - j} \in \{\bM_{k + 1 - j}, \bM_{k + 1 - j, i}\}, 1 \leq j \leq k - j_{\ast}} \frac{\delta_{j_{\ast}}}{n} \bx_i^{\top} \left(\prod_{j = 1}^{k - j_{\ast}} \bR_{k + 1 - j}\right) \bx_i,
	\end{align*}
	which further tells us
	\begin{align*}
		& \|c_{i, k, k', j_{\ast} - 1}\bH_{k', j_{\ast} - 1} \|_{\op} \\
		& = \left\|\sum_{\bR_{k + 1 - j} \in \{\bM_{k + 1 - j}, \bM_{k + 1 - j, i}\}, 1 \leq j \leq k - j_{\ast}}\frac{\delta_{j_{\ast}}}{n} \bx_i^{\top} \left(\prod_{j = 1}^{k - j_{\ast}} \bR_{k + 1 - j}\right) \bx_i \cdot \bH_{k', j_{\ast} - 1} \right\|_{\op}\\
		 & \leq \prod_{k = 0}^{K - 1} \left( 1 + \|\bM_k\|_{\op} + \|\bM_{k, i}\|_{\op} \right) \\
		& \leq \prod_{k = 0}^{K - 1} \left(1 + \delta_k C_{\Sigma, \zeta} + \delta_k C_{\Sigma, \zeta} \right) \leq e^{2\Delta C_{\Sigma, \zeta} }.
	\end{align*}
	This completes the proof.
\end{proof}

As a consequence of  \Cref{lemma:gradx}, we can write
\begin{align*}
	& \frac{2}{n} \sum_{i = 1}^n (y_i - \bx_i^{\top} \hat\bbeta_{k + 1, -i})\bx_i^{\top} \nabla_{\bx_s} \hat\bbeta_{k + 1, -i} \\
	& = \frac{2}{n} \sum_{i = 1}^n \sum_{k' = 0}^k \sum_{r = k'}^k c_{i, k, k', r} (y_i - \bx_i^{\top} \hat\bbeta_{k + 1, -i} ) \bx_i^{\top} \bH_{k', r} \cdot \left( \frac{\delta_{k'}}{n} (y_s - \bx_s^{\top} \hat\bbeta_{k', -i}) \id_{p + 1} - \frac{\delta_{k'}}{n} \bx_s \hat\bbeta_{k', -i}^{\top} \right) \\
	& = \sum_{k' = 0}^k \sum_{r = k'}^k \left( \bg_{k, k', r, s} + \bar\bg_{k, k', r, s} \right),
\end{align*}
where we define
\begin{align*}
	& \bg_{k, k', r, s} = \frac{2 \delta_{k'} }{n^2} \sum_{i = 1}^n  c_{i, k, k', r}(y_i - \bx_i^{\top} \hat\bbeta_{k + 1, -i} ) (y_s - \bx_s^{\top} \hat\bbeta_{k', -i})\bx_i^{\top} \bH_{k', r}, \\
	& \bar\bg_{k, k', r, s} = - \frac{2 \delta_{k'}}{n^2}\sum_{i = 1}^n c_{i, k, k', r}(y_i - \bx_i^{\top} \hat\bbeta_{k + 1, -i} )\bx_i^{\top} \bH_{k', r} \bx_s \hat\bbeta_{k', -i}^{\top}. 
\end{align*}
We define $\bV_{k, k', r}, \bar\bV_{k, k', r} \in \R^{(p + 1) \times n}$ such that the $s$-th columns correspond to $\bg_{k, k', r, s}^{\top}$ and $\bar\bg_{k, k', r, s}^{\top}$, respectively. We also define $\tilde\bV_{k} \in \R^{(p + 1) \times n}$ such that the $s$-th column of this matrix corresponds to $2(y_s - \bx_s^{\top} \hat\bbeta_{k + 1, -s}) \hat\bbeta_{k + 1, -s} / n$. Inspecting \Cref{eq:gradxs}, we see that to upper bound the Frobenius norm of $\nabla_{\bX}\hat{R}^{\loo}(\hat\bbeta_{k + 1})$, it suffices to upper bound the Frobenius norms of matrices $\bV_{k, k', r}, \bar\bV_{k, k', r}$, and $\tilde\bV_{k}$, which we analyze in the lemma below. 
\begin{lemma}\label{lemma:V}
	On the set $\Omega$, we have
	\begin{align}
		\|\bV_{k, k', r}\|_F^2 & \leq \frac{4\delta_{k'}^2e^{4\Delta C_{\Sigma, \zeta}}C_{\Sigma, \zeta}}{n} \cdot \bar{B}_{\ast}^2 \cdot \cG_2(C_{\Sigma, \zeta}, \Delta, m, B_{0})^2 \cdot (\log n)^2, \label{eq:V} \\
		  \|\bar\bV_{k, k', r}\|_F^2 & \leq \frac{4\delta_{k'}^2 e^{4\Delta C_{\Sigma, \zeta}} C_{\Sigma, \zeta}^2 B_{\ast}^2}{n} \cdot \cG_2(C_{\Sigma, \zeta}, \Delta, m, B_{0})^2 \cdot (\log n)^2, \label{eq:bV} \\
		\|\tilde\bV_{k}\|_F^2 & \leq \frac{4B_{\ast}^2 }{n} \cdot \cG_2(C_{\Sigma, \zeta}, \Delta, m, B_{0})^2  \cdot (\log n)^2. \label{eq:tV}
	\end{align}
\end{lemma} 
\begin{proof}[Proof of \Cref{lemma:V}]
	We observe that
\begin{align*}
	\bV_{k, k', r} = \frac{2 \delta_{k'} }{n^2} \bH_{k', r} \bX^{\top} \bA_{k, k', r},
\end{align*} 
where $\bA_{k, k', r} \in \R^{n \times n}$, and $(\bA_{k, k', r})_{is} = c_{i, k, k', r} (y_i - \bx_i^{\top} \hat\bbeta_{k + 1, -i})(y_s - \bx_s^{\top} \hat\bbeta_{k', -i})$. Note that on $\Omega$, by Corollary \ref{cor:y-Xbeta} and \Cref{lemma:y-xbeta}, we have 
\begin{align*}
	& \frac{1}{n}\|\by - \bX \hat\bbeta_{k', -i}\|_2^2 \leq \bar{B}_{\ast}^2 \cdot \log n, \\
	& \frac{1}{n}\|\ba_{k + 1}\|_2^2 \leq \cG_2(C_{\Sigma, \zeta}, \Delta, m, B_{0})^2 \cdot \log n. 
\end{align*}
This further implies that
\begin{align*}
	\|\bA_{k, k', r}\|_F^2 \leq n^2\sup_{i \in [n]}|c_{i, k, k', r}|^2 \cdot \bar{B}_{\ast}^2 \cdot \cG_2(C_{\Sigma, \zeta}, \Delta, m, B_{0})^2\cdot (\log n)^2. 
\end{align*}
As a result, 
\begin{align*}
	\|\bV_{k, k', r}\|_F^2 \leq & \frac{4\delta_{k'}^2}{n^4} \cdot \|\bH_{k', r}\|_{\op}^2\cdot \|\bX\|_{\op}^2 \cdot \|\bA_{k, k', r}\|_F^2 \\
	\leq &  \frac{4\delta_{k'}^2e^{4\Delta C_{\Sigma, \zeta}}C_{\Sigma, \zeta}}{n} \cdot \bar{B}_{\ast}^2 \cdot \cG_2(C_{\Sigma, \zeta}, \Delta, m, B_{0})^2 \cdot (\log n)^2, 
\end{align*}
which concludes the proof for the first inequality. 

We then consider upper bounding $\|\bar\bV_{k, k', r}\|_F$. Note that
\begin{align*}
	& \bar{\bV}_{k, k', r} = - \frac{2\delta_{k'}}{n^2} 
	\bQ_{k, k', r} \bX \bH_{k', r} \bX^{\top}, \\
	& \bQ_{k, k', r} = \left[ \hat\bbeta_{k, -1} \mid \cdots \mid \hat\bbeta_{k, -n} \right] \cdot \diag\{(c_{i, k, k', r}(y_i - \bx_i^{\top}\hat \bbeta_{k + 1, -i}))_{i = 1}^n\} \in \R^{(p + 1) \times n}. 
\end{align*}
Therefore, 
\begin{align*}
	\|\bar{\bV}_{k, k', r}\|_F^2 & \leq  \frac{4\delta_{k'}^2}{n^4} \cdot \|\bQ_{k, k', r}\|_F^2 \cdot \|\bX \bX^{\top}\|_{\op}^2 \cdot \|\bH_{k', r}\|_{\op}^2 \\
	& \leq \frac{4\delta_{k'}^2 e^{4\Delta C_{\Sigma, \zeta}} C_{\Sigma, \zeta}^2 B_{\ast}^2}{n} \cdot \cG_2(C_{\Sigma, \zeta}, \Delta, m, B_{0})^2  \cdot (\log n)^2.
\end{align*}
This completes the proof of \Cref{eq:bV}. Finally, we prove \Cref{eq:tV}. By \Cref{lemma:upper-bound-beta} and \Cref{lemma:y-xbeta}, we obtain
\begin{align*}
	\|\tilde{\bV}_{k}\|_F^2 \leq \frac{4B_{\ast}^2(\log n)^2}{n} \cdot \cG_2(C_{\Sigma, \zeta}, \Delta, m, B_{0})^2  \cdot (\log n)^2. 
\end{align*} 
This is exactly what we aim to prove. 
\end{proof}

By triangle inequality, 
\begin{align*}
    \|\nabla_{\bX} \hat{R}^{\loo} (\hat\bbeta_k)\|_F \leq \|\tilde{\bV}_{k}\|_F + \sum_{k' = 0}^k \sum_{r = k'}^k \left(\|\bar{\bV}_{k, k', r}\|_F + \|{\bV}_{k, k', r}\|_F \right).
\end{align*}

The proof of \Cref{lem:bound-norm-gradient-wrt-features} now follows by putting together the above upper bound and \Cref{lemma:V}.
\end{proof}

\subsection[Upper bounding norm of LOO gradient]{Upper bounding $\nabla_{\by} \hR^\loo(\hat{\bbeta}_k)$}
\label{sec:proof-lem:bound-norm-gradient-wrt-response}

Next, we upper bound the Euclidean norm of $\nabla_{\by} \hR^\loo(\widehat{\bbeta}_k)$. This part is in spirit similar to the upper bounding of the Euclidean norm of $\nabla_{\bX} \hR^\loo(\widehat{\bbeta}_k)$ that we discussed in the previous section.

More precisely, we will show the following:
\begin{lemma}
    [Bounding norm of gradient with respect to response]
    \label{lem:bound-norm-gradient-wrt-response}
    On the set $\Omega$, 
   \begin{align}\label{eq:gradrY-lemma}
    \begin{split}
    	& \|\nabla_{\by} \hR^\loo(\hat{\bbeta}_{k}) \|_2 \\
     & \leq  \frac{2 \cG_2(C_{\Sigma, \zeta}, \Delta, m, B_{0})}{\sqrt{n}} \cdot \sqrt{\log n} + \frac{2\Delta K C_{\Sigma, \zeta} e^{2\Delta  C_{\Sigma, \zeta}}}{\sqrt{n}} \cdot \cG_2(C_{\Sigma, \zeta}, \Delta, m, B_{0}) \cdot \sqrt{\log n}. 
    \end{split}
    \end{align} 
\end{lemma}
\begin{proof}
For $s \in [n]$, we note that
\begin{align}\label{eq:ysloo}
	\frac{\partial}{\partial y_s} \hR^\loo(\hat{\bbeta}_{k})  = \frac{2}{n} (y_s - \bx_s^{\top} \hat\bbeta_{k, -s}) - \frac{2}{n} \sum_{i = 1}^n (y_i - \bx_i^{\top} \hat\bbeta_{k, -i}) \bx_i^{\top}\frac{\partial}{\partial y_s} \hat\bbeta_{k, -i}. 
\end{align}
If $i = s$, then $\frac{\partial}{\partial y_s} \hat\bbeta_{k, -i} = 0$ for all $k \in \{0\} \cup [K]$. Moving forward, we focus on the more interesting case $i \neq s$. We also have
\begin{align*}
	\frac{\partial}{\partial y_s} \hat\bbeta_{k + 1, -i} = & \frac{\partial}{\partial y_s} \hat\bbeta_{k, -i} + \frac{\delta_k}{n} \bx_s - \frac{\delta_k}{n} \sum_{j \neq i} \bx_j \bx_j^{\top} \frac{\partial}{\partial y_s} \hat\bbeta_{k, -i} \\
	= & \bM_k\frac{\partial}{\partial y_s} \hat\bbeta_{k, -i} + \bM_{k, i}\frac{\partial}{\partial y_s} \hat\bbeta_{k, -i} + \frac{\delta_k}{n} \bx_s,
\end{align*}
where we recall that $\bM_k = (\id_{p + 1} - \delta_k \hat\bSigma)$ and $\bM_{k, i} = \delta_k \bx_i \bx_i^{\top} / n$.
Invoking the same argument that we employed to derive \Cref{lemma:gradx}, we can conclude that 
\begin{align*}
	\frac{\partial}{\partial y_s} \bx_i^{\top} \hat\bbeta_{k + 1, -i} = \sum_{k' = 0}^k \sum_{r = k'}^k c_{i, k, k', r} \bx_i^{\top} \bH_{k', r} \cdot \frac{\delta_{k'}}{n} \bx_s. 
\end{align*}
Plugging this into \Cref{eq:ysloo} leads to the following equality:
\begin{align*}
	\frac{\partial}{\partial y_s} \hR^\loo(\hat{\bbeta}_{k + 1}) = \frac{2}{n} (y_s - \bx_s^{\top} \hat\bbeta_{k + 1, -s}) - \sum_{k' = 0}^k \sum_{r = k'}^k \eta_{k, k', r, s},
\end{align*}
where
\begin{align*}
	\eta_{k, k', r, s} = \frac{2}{n} \sum_{i = 1}^n (y_i - \bx_i^{\top} \hat\bbeta_{k + 1, -i})c_{i, k, k', r} \bx_i^{\top} \bH_{k', r} \cdot \frac{\delta_{k'}}{n} \bx_s.  
\end{align*}
We define $\boldeta_{k, k', r} = (\eta_{k, k', r, s})_{s = 1}^n \in \R^n$. 
It then holds that
\begin{align*}
	& \boldeta_{k, k', r} = \frac{2\delta_{k'}}{n^2} \bX \bH_{k', r} \bX^{\top} \bq_{k, k', r}, \\
	& \bq_{k, k', r} = \big( c_{i, k, k', r} (y_i - \bx_i^{\top} \hat\bbeta_{k + 1, -i}) \big)_{i = 1}^n \in \R^n. 
\end{align*}
We can upper bound the Euclidean norm of $\boldeta_{k, k', r}$ using \Cref{lemma:y-xbeta} and \ref{lemma:gradx}. More precisely, 
\begin{align*}
	\|\boldeta_{k, k', r}\|_2 \leq \frac{2\delta_{k'}}{n^2}\|\bX^{\top}\bX\|_{\op} \cdot \|\bH_{k', r}\|_{\op} \cdot \|\bq_{k, k', r}\|_2 \leq \frac{2\delta_{k'} C_{\Sigma, \zeta} e^{2\Delta  C_{\Sigma, \zeta}}}{\sqrt{n}} \cdot \cG_2(C_{\Sigma, \zeta}, \Delta, m, B_{0}) \cdot \sqrt{\log n}. 
\end{align*}
Note that
\begin{align*}
	\nabla_{\by} \hR^\loo(\hat{\bbeta}_{k}) = \frac{2}{n} \ba_k - \sum_{k' = 0}^k \sum_{r = k'}^k \boldeta_{k, k', r}.
\end{align*}
Invoking triangle inequality and \Cref{lemma:y-xbeta}, we obtain
\begin{align}\label{eq:gradrY}
\begin{split}
	& \|\nabla_{\by} \hR^\loo(\hat{\bbeta}_{k}) \|_2 \\
 \leq &  \frac{2 \cG_2(C_{\Sigma, \zeta}, \Delta, m, B_{0})}{\sqrt{n}} \cdot \sqrt{\log n} + \frac{2\Delta K C_{\Sigma, \zeta} e^{2\Delta  C_{\Sigma, \zeta}}}{\sqrt{n}} \cdot \cG_2(C_{\Sigma, \zeta}, \Delta, m, B_{0}) \cdot \sqrt{\log n}. 
\end{split}
\end{align}
This completes the proof.
\end{proof}

\section{Proof of \Cref{thm:uniform-consistency-general}}
\label{sec:proof-thm:uniform-consistency-general}

\bigskip

\ThmUniformConsistencyGeneral*

\subsection{Proof schematic}
\label{sec:outline-thm:uniform-consistency-general}

A visual schematic for the proof of \Cref{thm:uniform-consistency-general} for general risk functionals is provided in \Cref{fig:schematic-proof-thm:uniform-consistency-general}.

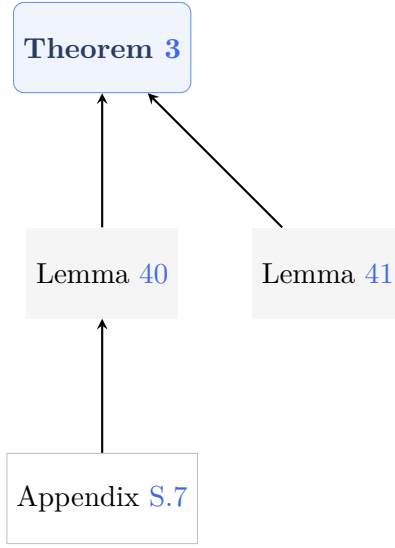
\begin{figure}[!ht]
    \centering
    \begin{tikzpicture}[node distance=3cm]
    \node (thm-uniform-consistency-general) [theorem] {\Cref{thm:uniform-consistency-general}} ;
    \node (lem-loo-risk-concentration-general) [lemma, below of=thm-uniform-consistency-general] {\Cref{lem:loo-risk-concentration-general}} ;
    \node (lem-projection-effects-general) [lemma, right of=lem-loo-risk-concentration-general] {\Cref{lem:projection-effects-general}} ;
    \node (lemma-gradient-upper-bound) [misc, below of=lem-loo-risk-concentration-general] {\Cref{sec:proof-lemma:gradient-upper-bound}} ;
    \draw [arrow] (lem-loo-risk-concentration-general) -- (thm-uniform-consistency-general) ;
    \draw [arrow] (lem-projection-effects-general) -- (thm-uniform-consistency-general) ;
    \draw [arrow] (lemma-gradient-upper-bound) -- (lem-loo-risk-concentration-general) ;
    \end{tikzpicture}
    \caption{Schematic for the proof of \Cref{thm:uniform-consistency-general} for general risk functionals 
    }
    \label{fig:schematic-proof-thm:uniform-consistency-general}
\end{figure}

Once again, we will work on the set $\Omega$, which we recall is defined in \Cref{eq:Omega}. 
The proof idea is similar to that for the squared loss. More precisely, if we can prove \Cref{eq:Tloo,eq:Tbetak,eq:rktfkt} listed below, then once again can add up the probabilities and show that the sum is finite. Next, we just apply the first Borel–Cantelli lemma, which leads to the following uniform consistency result: 
\begin{align*}
    \sup_{k \in \{0\} \cup [K]} \left| \hat{\Psi}^{\loo}(\hat\bbeta_k) - \Psi(\hat\bbeta_k) \right|   \asto 0. 
\end{align*}

\subsection[Concentration analysis]{Concentration analysis}
\label{sec:concentration-analysis-general}

As before, we will first prove that both $\hat{\Psi}^{\loo}(\hat\bbeta_{k + 1})$ and $\Psi(\hat\bbeta_{k + 1})$ concentrate. 
To this end, we shall again analyze the gradients and show that they are Lipschitz functions of the input data. 
The proof for this part is similar to the proof of Lemmas \ref{lemma:Rloo} and \ref{lemma:cRk}. 

We define
\begin{itemize}
    \item $f_{k + 1}^\psi(\bw_1, \cdots, \bw_n) = \hat{\Psi}^{\loo} (\hat\bbeta_{k + 1})$
    \item $\tilde f_{k + 1}^\psi = f_{k + 1}^\psi \circ h$
    \item $r_k^\psi(\bw_1, \cdots, \bw_n) = \Psi(\hat\bbeta_k)$
    \item $\tilde r_k^\psi = r_k \circ h$
\end{itemize}
Our formal statement then is as follows. 
\begin{lemma}
    [LOO and risk concentration analysis]
    \label{lem:loo-risk-concentration-general}
    Under the assumptions of Theorem \ref{thm:uniform-consistency-general}, 
    with probability at least $1 - 2(n + p)^{-4} - (n \log^2 n)^{-1} m_4 - 2(K + 1)C_{\LSI} n^{-2}$, for all $k \in \{0\} \cup [K]$
    \begin{align*}
    \left| \hat{\Psi}^{\loo} (\hat\bbeta_{k}) - \E[\tilde{f}_k^\psi(\bw_1, \cdots, \bw_n)] \right| 
    &\leq \frac{2\sigma_{\LSI} L K \xi^\psi(C_{\Sigma, \zeta}, \Delta, m, B_0) \cdot (\log n)^{3/2}}{\sqrt{n}}, \\
    \left| \Psi(\hat\bbeta_k) - \E[\tilde r_k^\psi (\bw_1, \cdots, \bw_n)] \right| 
    &\leq \frac{2\sigma_{\LSI} L \bar\xi^\psi(C_{\Sigma, \zeta}, \Delta, m, B_0) (\log n)^{3/2}}{\sqrt{n}}.
    \end{align*}
    In the above display, $\xi^\psi(C_{\Sigma, \zeta}, \Delta, m, B_0)$ and $\bar\xi^\psi(C_{\Sigma, \zeta}, \Delta, m, B_0)$ are positive constants that depend only on $(C_{\Sigma, \zeta}, \Delta, m, B_0)$. 
\end{lemma}
\begin{proof}
We start by writing down the gradient. For all $s \in [n]$, note that
\begin{align*}
    \nabla_{\bx_s}\hat{\Psi}^{\loo}(\hat\bbeta_{k + 1}) = -\frac{1}{n}\partial_2 \psi(y_s, \bx_s^{\top}\hat\bbeta_{k + 1, -s}) \hat\bbeta_{k + 1, -s}^{\top} - \frac{1}{n} \sum_{i = 1}^n \partial_2 \psi(y_i, \bx_i^{\top} \hat\bbeta_{k + 1, -i}) \bx_i^{\top} \nabla_{\bx_s} \hat\bbeta_{k + 1, -i}, %
\end{align*}
where $\partial_i$ stands for taking the partial derivative with respect to the $i$-th input. Here, $i \in \{1,2\}$. 
By \Cref{lemma:gradx}, on $\Omega$ we have
\begin{align*}
    & \frac{1}{n} \sum_{i = 1}^n \partial_2 \psi(y_i, \bx_i^{\top} \hat\bbeta_{k + 1, -i}) \bx_i^{\top} \nabla_{\bx_s} \hat\bbeta_{k + 1, -i} \\
    = & \sum_{k' = 0}^k \sum_{r = k'}^k c_{i, k, k', r} \frac{1}{n} \sum_{i = 1}^n \partial_2 \psi(y_i, \bx_i^{\top} \hat\bbeta_{k + 1, -i}) \bx_i^{\top} \bH_{k', r} \cdot \left( \frac{\delta_{k'}}{n} (y_s - \bx_s^{\top} \hat\bbeta_{k', -i}) \id_{p + 1} - \frac{\delta_{k'}}{n} \bx_s\hat\bbeta_{k', -i}^{\top} \right) \\
    = & \sum_{k' = 0}^k \sum_{r = k'}^k \big( \bg_{k, k', r, s}^{\psi} + \bar\bg_{k, k', r, s}^{\psi} \big), 
\end{align*}
where 
\begin{align*}
    & \bg_{k, k', r, s}^\psi = \frac{\delta_{k'}}{n^2} \sum_{i = 1}^n c_{i, k, k', r} \partial_2  \psi(y_i, \bx_i^{\top} \hat\bbeta_{k + 1, -i})(y_s - \bx_s^{\top} \hat\bbeta_{k', -i})\bx_i^{\top} \bH_{k', r}, \\
    & \bar\bg_{k, k', r, s}^{\psi} = -\frac{\delta_{k'}}{n^2} \sum_{i = 1}^n c_{i, k, k', r} \partial_2 \psi(y_i, \bx_i^{\top} \hat\bbeta_{k + 1, -i}) \bx_i^{\top} \bH_{k', r} \bx_s \hat\bbeta_{k', -i}^{\top}. 
\end{align*}
We let $\bV_{k, k', r}^\psi, \bar\bV_{k, k', r}^\psi \in \R^{(p + 1) \times n}$, such that the $s$-th columns are set to be $(\bg_{k, k', r, s}^\psi)^{\top}$ and $(\bar\bg_{k, k', r, s}^{\psi} )^{\top}$, respectively. We also define $\tilde\bV_{k}^\psi \in \R^{(p + 1) \times n}$ such that the $s$-th column of this matrix corresponds to $\partial_2 \psi(y_s, \bx_s^{\top} \hat\bbeta_{k + 1, -s}) \hat\bbeta_{k + 1, -s} / n$. Using triangle inequality, we immediately obtain that 
\begin{align}\label{eq:gradXT}
    \|\nabla_{\bX} \hat{\Psi}^{\loo}(\hat\bbeta_{k + 1})\|_F \leq  \|\tilde\bV_{k}^\psi\|_F + \sum_{k' = 0}^k \sum_{r = k'}^k\left\{ \|\bV_{k, k', r}^\psi\|_F + \| \bar\bV_{k, k', r}^\psi\|_F \right\}. 
\end{align}
Next, we upper bound $\|\bV_{k, k', r}^\psi\|_F$, $\| \bar\bV_{k, k', r}^\psi\|_F$, and $\|\tilde\bV_{k}^\psi\|_F$. We observe that
\begin{align*}
    & \bV_{k, k', r}^{\psi} = \frac{\delta_{k'}}{n^2} \bH_{k', r} \bX^{\top} \bA_{k, k', r}^\psi, \qquad \bar\bV_{k, k', r}^{\psi} = -\frac{\delta_{k'}}{n^2} \bQ_{k, k', r}^\psi \bX \bH_{k', r} \bX^{\top}, 
\end{align*}
where
\begin{align*}
    & \bQ_{k, k', r}^\psi = \left[ \bbeta_{k, -1} \mid \cdots \mid \bbeta_{k, -n} \right] \cdot \diag\{ (c_{i, k, k', r} \partial_2 \psi(y_i, \bx_i^{\top} \hat\bbeta_{k + 1, -i}))_{i = 1}^n \} \in \R^{(p + 1) \times n}, \\
    & (\bA_{k, k', r}^\psi)_{is} = c_{i, k, k', r} \partial_2  \psi(y_i, \bx_i^{\top}\hat\bbeta_{k + 1, -i}) (y_s -  \bx_s^{\top} \hat\bbeta_{k', -i}).
\end{align*}
We let $\ba_{k + 1}^\psi = ( \partial_2  \psi(y_i, \bx_i^{\top} \hat\bbeta_{k + 1, -i}) )_{i = 1}^n$. 
Recall that $\ba_{k + 1} = (y_i - \bx_i^{\top} \bbeta_{k + 1, -i})_{i = 1}^n$. 
Using triangle inequality, we obtain that 
\[
    \|\ba_{k + 1}^\psi\|_2 \leq 3C_{\psi} (\|\ba_{k + 1}\|_2 + \|\by\|_2) + \sqrt{2n} \bar{C}_{\psi}.
\]
Invoking \Cref{lemma:y-xbeta}, we know that on $\Omega$, $\|\ba_{k + 1}\|_2 \leq \sqrt{n} \cG_2(C_{\Sigma, \zeta}, \Delta, m, B_0) \cdot \sqrt{\log n}$.
Furthermore, by definition we know that on $\Omega$, $\|\by\|_2 \leq \sqrt{n(m + \log n)}$. 
By Corollary \ref{cor:y-Xbeta} we see that $\|\by - \bX \bbeta_{k, -i}\|_2 \leq \sqrt{n} \bar B_{\ast} \cdot \sqrt{\log n}$. By \Cref{lemma:gradx}, we have $\|c_{i, k, k', r} \bH_{k', r}\|_{\op} \leq e^{2\Delta C_{\Sigma, \zeta}}$. 
Putting together all these results, we conclude that 
\begin{align}\label{eq:Vt1}
\begin{split}
     \|{\bV}_{k, k', r}^\psi\|_F \leq & \frac{\delta_{k'}}{n^2} \cdot \|\bH_{k', r}\|_{\op} \cdot \|\bX\|_{\op} \cdot \|\bA_{k, k', r}^\psi\|_F \\
     \leq & \frac{\delta_{k'} e^{2\Delta C_{\Sigma, \zeta}} C_{\Sigma, \zeta}^{1/2} \bar B_{\ast} \cdot(3C_{\psi} \cG_2(C_{\Sigma, \zeta}, \Delta, m, B_0) + 3C_{\psi}  \sqrt{m} + \sqrt{2}\bar C_{\psi}) \cdot {\log n}}{\sqrt n}.
\end{split}
\end{align}
Applying Lemma \ref{lemma:upper-bound-beta}, we deduce that 
$$\|\bQ_{k, k', r}^\psi\|_F \leq \sqrt{n} B_{\ast} \sup_{i \in [n]}|c_{i, k, k', r}| \cdot (3C_{\psi} \cG_2(C_{\Sigma, \zeta}, \Delta, m, B_0) + 3C_{\psi}m^{1/2} + \sqrt{2}\bar C_{\psi}) \cdot {\log n} .$$ 
Therefore, 
\begin{align}\label{eq:Vt2}
\begin{split}
   &  \|\bar{\bV}_{k, k', r}^\psi\|_F \leq \frac{\delta_{k'}}{n^2} \cdot \|\bQ_{k, k', r}^\psi\| \cdot \|\bX\|_{\op}^2 \cdot \|\bH_{k', r}\|_{\op} \\
   & \leq  \frac{\delta_{k'}B_{\ast} e^{2\Delta C_{\Sigma, \zeta}} C_{\Sigma, \zeta} \cdot (3C_{\psi} \cG_2(C_{\Sigma, \zeta}, \Delta, m, B_0) + 3 C_{\psi} \sqrt{m} + \sqrt{2}\bar C_{\psi})  \cdot {\log n}}{\sqrt n}, \\
   & \|\tilde \bV_k^\psi\|_F \leq \frac{1}{n} \|\ba_{k + 1}^\psi\|_2 \cdot \|\hat\bbeta_{k + 1, -s}\|_2 \leq \frac{B_{\ast}(3C_{\psi} \cG_2(C_{\Sigma, \zeta}, \Delta, m, B_0) + 3C_{\psi} \sqrt{m} + \sqrt{2} \bar C_{\psi}) \cdot {\log n}}{\sqrt{n}}. 
\end{split}
\end{align}
Combining \Cref{eq:gradXT,eq:Vt1,eq:Vt2}, we see that there exists a constant $\xi^\psi_1(C_{\Sigma, \zeta}, \Delta, m, B_0)$ that depends only on $(C_{\Sigma, \zeta}, \Delta, m, B_0)$, such that on $\Omega$, for all $k \in \{0\} \cup [K]$ we have 
\begin{align*}
    \|\nabla_{\bX} \hat{\Psi}^{\loo} (\hat\bbeta_{k + 1})\|_F \leq \frac{K \xi^\psi_1(C_{\Sigma, \zeta}, \Delta, m, B_0) \cdot {\log n}}{\sqrt{n}}. 
\end{align*}
Analogously, we can conclude the existence of a non-negative constant $\xi_2^\psi(C_{\Sigma, \zeta}, \Delta, m, B_0)$, such that on $\Omega$, it holds that
\begin{align*}
    \|\nabla_{\by} \hat{\Psi}^{\loo} (\hat\bbeta_{k + 1})\|_F \leq \frac{K \xi_2^\psi(C_{\Sigma, \zeta}, \Delta, m, B_0) \cdot {\log n}}{\sqrt{n}}.
\end{align*}
Hence, we know that 
\[
\|\nabla_{\bW} \hat{\Psi}^{\loo} (\hat\bbeta_{k + 1})\|_F \leq K \xi^\psi(C_{\Sigma, \zeta}, \Delta, m, B_0) \cdot {\log n}
\]
if we set $\xi^\psi(C_{\Sigma, \zeta}, \Delta, m, B_0) = \xi_1^\psi(C_{\Sigma, \zeta}, \Delta, m, B_0) + \xi_2^\psi(C_{\Sigma, \zeta}, \Delta, m, B_0)$. 
 Following the same steps that we used to derive Lemma \ref{lemma:Rloo}, we deduce that with probability at least $1 - 2(n + p)^{-4} - (n\log^2 n)^{-1} m_4 - 2(K + 1)C_{\LSI} n^{-2}$,
\begin{align}\label{eq:Tloo}
    \left| \hat{\Psi}^{\loo} (\hat\bbeta_{k}) - \E[\tilde{f}_k^\psi(\bw_1, \cdots, \bw_n)] \right| \leq \frac{2\sigma L K \xi^\psi(C_{\Sigma, \zeta}, \Delta, m, B_0) \cdot (\log n)^{3/2}}{\sqrt{n}}. 
\end{align}
Similarly, we can prove that with probability at least $1 - 2(n + p)^{-4} - (n \log^2 n)^{-1} m_4 - 2(K + 1)C_{\LSI} n^{-2}$, for all $k \in \{0\} \cup [K]$, 
\begin{align}\label{eq:Tbetak}
    \left| \Psi(\hat\bbeta_k) - \E[\tilde r_k^\psi (\bw_1, \cdots, \bw_n)] \right| \leq \frac{2\sigma L \bar\xi^\psi(C_{\Sigma, \zeta}, \Delta, m, B_0) (\log n)^{3/2}}{\sqrt{n}}.
\end{align}
for some constant $\bar\xi^\psi(C_{\Sigma, \zeta}, \Delta, m, B_0)$ that depends only on $(C_{\Sigma, \zeta}, \Delta, m, B_0)$. 
\end{proof}

\subsection[Uniform consistency]{Uniform consistency}
\label{sec:uniform-consistency-general}

Next, we shall prove that projection has little effect on the expected risk. 
\begin{lemma}
    [LOO and risk bias analysis]
    \label{lem:projection-effects-general}
    On the set $\Omega$,
    it holds that
    \begin{align}\label{eq:rktfkt}
    \begin{split}
    	& \sup_{k \in \{0\} \cup [K]}\left|\E[\tilde{r}_k^\psi (\bw_1, \cdots, \bw_n)] - \E[r_k^\psi(\bw_1, \cdots, \bw_n)] \right| = o_{n}(1), \\
    	& \sup_{k \in \{0\} \cup [K]} \left|\E[\tilde{f}_k^\psi (\bw_1, \cdots, \bw_n)] - \E[f_k^\psi(\bw_1, \cdots, \bw_n)] \right| = o_n(1). 
    \end{split}
    \end{align}
\end{lemma}
\begin{proof}
Using the Cauchy-Schwartz inequality, we obtain
\begin{align}\label{rktfktexp}
\begin{split}
	 & \Big| \E[r_k^\psi(\bw_1, \cdots, \bw_n)] - \E[\tilde r_k^\psi(\bw_1, \cdots, \bw_n)] \Big| \leq \P(\Omega^c)^{1/2} \E[\psi(y_{\new}, \bx_{\new}^{\top}\hat\bbeta_{k})^2]^{1/2}, \\
	 & \Big| \E[f_k^\psi(\bw_1, \cdots, \bw_n)] - \E[\tilde f_k^{\psi}(\bw_1, \cdots, \bw_n)] \Big| \leq \P(\Omega^c)^{1/2}\E[\psi(y_1, \bx_1^{\top} \hat\bbeta_{k, -1})^2]^{1/2}.
\end{split}
\end{align}
Since $\|\nabla \psi(x)\|_2 \leq C_{\psi} \|x\|_2 + \bar C_{\psi}$, we are able to conclude that there exist constants $\phi_\psi, \bar\phi_\psi$ that depend only on $\psi(\cdot)$, such that $\|\psi(x)\|_2^2 \leq \phi_\psi \|x\|_2^4 + \bar\phi_\psi$ for all $x \in \R^2$. 
Putting this and \Cref{eq:four-moment} together, we obtain that
\begin{align}\label{eq:t2bound}
    \E[\psi(y_1, \bx_1^{\top} \hat\bbeta_{k, -1})^2] \leq \phi_\psi \E[\|(y_1, \bx_1^{\top} \hat\bbeta_{k, -1})\|_2^4] + \bar\phi_\psi \leq \phi_\psi \cH(\sigma_{\Sigma}, \zeta, B_0, m_8, \Delta)^2 + \bar\phi_\psi. 
\end{align}
Recall that $\P(\Omega^c) \leq 2(n + p)^{-4} + n^{-1} m_4$. 
Combining this, \Cref{rktfktexp,eq:t2bound}, we can establish \Cref{eq:rktfkt}. 

To derive uniform consistency, we also need to show that the expected prediction risk is robust to the sample size. 
Namely, we will prove $\E[\Psi(\hat\bbeta_k)] \approx \E[\Psi(\hat\bbeta_{k, -1})]$. 

Since $\|\nabla \psi(x)\|_2 \leq C_{\psi} \|x\|_2 + \bar C_{\psi}$, we see that there exist constants $\varphi_\psi \in \R$, such that for all $x, y \in \R^2$, 
\begin{align*}
    |\psi(x) - \psi(y)| \leq \varphi_\psi\|x - y\|_2 \cdot (1 + \|x\|_2^2 + \|y\|_2^2). 
\end{align*}
Therefore, 
\begin{align*}
   & \left|\E[\Psi(\hat\bbeta_k)]  - \E[\Psi(\hat\bbeta_{k, -1})] \right| \\
   &= \left| \E[r_k^\psi(\bw_1, \cdots, \bw_n)] - \E[f_k^\psi(\bw_1, \cdots, \bw_n)]  \right| \\
    &= \left|\E[ \psi(y_\new, \bx_\new^{\top} \hat\bbeta_k) ]- \E[\psi(y_\new, \bx_\new^{\top} \hat\bbeta_{k, -1})] \right| \\
    &\leq \varphi_\psi \E\left[ \big( 1 + \|(y_\new, \bx_\new^{\top} \hat\bbeta_k)\|_2^2 + \|(y_\new, \bx_\new^{\top} \hat\bbeta_{k, -1})\|_2^2 \big) \cdot |\bx_0^{\top}(\hat\bbeta_k - \hat\bbeta_{k, -1})|\right] \\
    &\leq 3\varphi_\psi \E\left[ (\bx_0^{\top}(\hat\bbeta_k - \hat\bbeta_{k, -1}))^2 \right]^{1/2} \cdot \E\left[ 1 + \|(y_0, \bx_0^{\top} \hat\bbeta_k)\|_2^4 + \|(y_0, \bx_0^{\top} \hat\bbeta_{k, -1})\|_2^4 \right] \\
    &\leq 3\varphi_\psi (\sigma_{\Sigma} + 1)^{1/2} \E\left[ \|\hat\bbeta_k  - \hat\bbeta_{k, -1}\|_2^2 \right]^{1/2} \cdot   \E\left[ 1 + \|(y_0, \bx_0^{\top} \hat\bbeta_k)\|_2^4 + \|(y_0, \bx_0^{\top} \hat\bbeta_{k, -1})\|_2^4 \right], 
\end{align*}
which by \Cref{eq:diff-beta,eq:four-moment} goes to zero as $n,p \to \infty$.
\end{proof}

\section{Proof of \Cref{thm:coverage}}
\label{sec:proof-thm:coverage}

\bigskip

\ThmCoverage*

\begin{proof}
For $z \in \R$, we define $\I_z(y, u) = \mathbbm{1}\{y - u \leq z\}$. We first prove that if we replace $\psi(y, u)$ by $\I_z(y, u)$ in \Cref{thm:uniform-consistency-general}, then as $n,p \to \infty$ we still have 
\begin{align}\label{eq:43}
    \sup_{k \in \{0\} \cup [K]} | \hat{\Psi}^{\loo}(\hat\bbeta_k) - \Psi(\hat\bbeta_k) | \asto 0.
\end{align}
This step is achieved via uniformly approximating $\I_z$ using Lipschitz functions. 
To be specific, we let $\{g_j\}_{j \in \NN_+}$ be a sequence of Lipschitz  functions satisfying $\|g_j - \I_z\|_{\infty} \leq 2^{-j}$. 
We define
\begin{align*}
    \hat{\Psi}_j^{\loo} (\hat\bbeta_k) = \frac{1}{n} \sum_{i = 1}^n g_j(y_i - \bx_i^{\top} \hat\bbeta_{k, -i})
    \quad
    \text{and}
    \quad
    \Psi_j(\hat\bbeta_k) = \E[g_j(y_{\new} - \bx_{\new}^{\top} \hat\bbeta_k) \mid \bX, \by]. 
\end{align*}
By \Cref{thm:uniform-consistency-general}, we know that for every $j$,
\[ 
    \sup_{k \in \{0\} \cup [K]} |  \hat{\Psi}_j^{\loo} (\hat\bbeta_k) - \Psi_j(\hat\bbeta_k) | \asto 0.
\]
Furthermore, notice that 
\[
|\hat{\Psi}_j^{\loo} (\hat\bbeta_k) - \hat{\Psi}^{\loo} (\hat\bbeta_k)| \leq 2^{-j}
\quad
\text{and}
\quad
|\Psi_j(\hat\bbeta_k) - \Psi(\hat\bbeta_k)| \leq 2^{-j},
\]
and $j$ is arbitrary, thus completing the proof of \Cref{eq:43}. 

We denote by $\hat F_k$ the cumulative distribution function (CDF) of the uniform distribution over $\{y_i - \bx_i^{\top} \hat\bbeta_{k, -i}: i \in [n]\}$, and denote by $F_k$ the CDF of $y_{\new} - \bx_{\new}^{\top} \hat\bbeta_k$ conditioning on $(\bX, \by)$. We emphasize that both $F_k$ and $\hat F_k$ are random distributions that depend on $(\bX, \by)$. Next, we prove that $F_k$ is Lipschitz continuous. 
\begin{lemma}\label{lemma:Lipschitz-F}
    Under the conditions of \Cref{thm:coverage}, $F_k$ is $\kappa_{\pdf}$-Lipschitz continuous. 
\end{lemma}
\begin{proof}[Proof of \Cref{lemma:Lipschitz-F}]
    Note that $y_0 - \bx_0^{\top}\hat\bbeta_k = f(\bx_0) - \bx_0^{\top}\hat\bbeta_k + \eps_0$, where $\eps_0$ is independent of $f(\bx_0) - \bx_0^{\top}\hat\bbeta_k$. Since $\eps_0$ has a probability density function (PDF), we see that $y_0 -  \bx_0^{\top}\hat\bbeta_k$ also has a PDF, and we denote it by $h$. We denote by $h_{\eps}$ the PDF of $\eps_{\new}$ and denote by $G$ the CDF of $f(\bx_0) - \bx_0^{\top}\hat\bbeta_k$, then we have
    \begin{align*}
        h(x) = \int h_{\eps}(x - z) d G(z), 
    \end{align*}
    which is uniformly upper bounded by $\kappa_{\pdf}$ for all $x \in \R$. 
\end{proof}
As a consequence of \Cref{lemma:Lipschitz-F} and the fact that $y_\new - \bx_{\new}^{\top} \hat\bbeta_k$ has bounded fourth moment (see \Cref{eq:four-moment} for derivation), we immediately obtain that $\sup_{k \in \{0\} \cup [K]}\|\hat{F}_k - F_k \|_{\infty} \asto 0$ as $n,p \to \infty$.  

In addition, it is not hard to see that
\begin{align*}
    \left| \hat{F}_k(\hat{\alpha}_k(q_i)) - q_i \right| \leq n^{-1}
\end{align*}
for all $i \in \{1,2\}$ and $k \in \{0\} \cup [K]$. Therefore, 
\begin{align*}
    \sup_{k \in \{0\} \cup [K]}\left|F_k(\hat{\alpha}_k(q_i)) - q_i\right| \leq \sup_{k \in \{0\} \cup [K]}\left| \hat{F}_k(\hat{\alpha}_k(q_i)) - q_i \right| + \sup_{k \in \{0\} \cup [K]}\|\hat{F}_k - F_k \|_{\infty} \asto 0
\end{align*}
as $n, p \to \infty$,
thus completing the proof of the theorem.
\end{proof}

\section{Additional details for \Cref{sec:computational}}
\label{sec:compare}

\subsection{Proof of \Cref{lemma:beta_tilde=beta}}
\label{sec:proof-lemma:beta_tilde=beta}

\bigskip

\LemmaBetaTildeBeta*

\begin{proof}
We prove the lemma through induction on $k$. For $k = 0$, by definition $\tilde{\bbeta}_{0, -i} = \hat{\bbeta}_{0, -i} = \bbeta_0$ for all $i \in [n]$. 
Suppose that we have $\tilde \bbeta_{k, -i} = \hat{\bbeta}_{k, -i}$ iteration $k$ and all $i \in [n]$, we then prove that it also holds for iteration $k + 1$ via induction. 
Using its definition, we see that
\begin{align*}
    \tilde{\bbeta}_{k + 1, -i} 
    &= \tilde{\bbeta}_{k, -i} - \frac{2\delta_{k}}{n} \bX^{\top} \bX \, \tilde{\bbeta}_{k, -i} \, + \frac{2\delta_{k}}{n} \bX^{\top} \tilde{\yy}_{k,-i} \\
    &= \tilde{\bbeta}_{k, -i} - \frac{2\delta_{k}}{n} \bX_{-i}^{\top} \bX_{-i} \,\tilde{\bbeta}_{k, -i} \, + \frac{2\delta_{k}}{n} \bX_{-i}^{\top} \yy_{-i} - \frac{2\delta_{k}}{n} \xx_i \big(\xx_i^{\top}\tilde{\bbeta}_{k, -i} - \xx_i^{\top} \hat{\bbeta}_{k, -i}  \big) \\
    &= \hat{\bbeta}_{k, -i} - \frac{2\delta_{k}}{n} \bX_{-i}^{\top} \bX_{-i} \,\hat{\bbeta}_{k, -i} \, + \frac{2\delta_{k}}{n} \bX_{-i}^{\top} \yy_{-i} \\
    &= \hat{\bbeta}_{k + 1, -i},
\end{align*}
thus completing the proof of the lemma by induction.
\end{proof}

\subsection{Proof of \Cref{lemma:hat-b}}
\label{sec:proof-lemma:hat-b}

\bigskip

\LemmaHatb*

\begin{proof}
We prove this lemma by induction over $k$. 
For the base case $k = 0$, the requirement of the lemma can be satisfied by setting 
\begin{align*}
    h^{(0)}_{ij} = 0, \qquad b^{(0)}_i = \xx_i^{\top} \bbeta_0, \qquad i,j \in [n]. 
\end{align*}
Suppose we can find $(h_{ij}^{(k)})_{i,j \leq n}$ and $(b_{i}^{(k)})_{i \leq n}$ for iteration $k$, we next show that the counterpart quantities also exist for iteration $k + 1$. 
We define $\bH^{(k)} \in \R^{n \times n}$, $\bb^{(k)} \in \R^n$, such that $\bH^{(k)}_{ij} = h^{(k)}_{ij}$ and $\bb^{(k)}_i = b^{(k)}_i$. 
Using induction hypothesis and \Cref{lemma:beta_tilde=beta}, we have
\begin{align*}
    &\xx_i^{\top} \hat{\bbeta}_{k+1,-i} \\
    &= \bx_i^\top \tilde{\bbeta}_{k,-i} \\
    &= \xx_i^{\top} \left(\tilde{\bbeta}_{k,-i} - \frac{2\delta_{k}}{n} \bX^{\top} \bX \tilde{\bbeta}_{k,-i} + \frac{2\delta_{k}}{n} \bX^{\top} \tilde{\by}_{k,-i} \right) \\
    &= \xx_i^{\top} \left(\hat{\bbeta}_{k,-i} - \frac{2\delta_{k}}{n} \bX^{\top} \bX \hat{\bbeta}_{k,-i} + \frac{2\delta_{k}}{n} \bX_{-i}^{\top} \yy_{-i} + \frac{2 \delta_{k+1}}{n} \bx_i \bx_i^\top \hat{\bbeta}_{k,-i}\right) \\
    &= \sum_{j = 1}^n h_{ij}^{(k)} y_j + b_i^{(k)} - \frac{2\delta_{k}}{n} \xx_i^{\top} \bX^{\top}  (\bH^{(k)} \yy + \bb^{(k)}) + \frac{2\delta_{k}}{n} \xx_i^{\top} \bX_{-i}^{\top} \yy_{-i} + \frac{2 \delta_{k+1}}{n} \| \bx_i \|_2^2 \left(\sum_{j = 1}^n h_{ij}^{(k)} y_j + b_i^{(k)}\right).
\end{align*} 
Note that the right-hand of the display above is affine in $\by$, which completes the proof for iteration $k + 1$. 
This completes our induction proof. 
\end{proof}

\subsection{Proof of \Cref{prop:efficient-loo-gd}}
\label{sec:additional-details-computational}

\bigskip

\PropEfficientLooGd*

\begin{proof}
By definition, $\hat\bbeta_{0, -i} = \hat\bbeta_0$ for all $i \in [n]$. After implementing the first step of gradient descent, we have
\begin{align*}
	\hat\bbeta_{1, -i} &= \hat\bbeta_{0, -i} - \frac{2\delta_1}{n} \bX_{-i}^{\top} \bX_{-i} \hat\bbeta_{0, -i} + \frac{2\delta_1}{n} \bX_{-i}^{\top} \by_{-i} \\
	&= \hat\bbeta_1 + \frac{2\delta_1}{n} \bx_i \bx_i^{\top} \hat\bbeta_0 - \frac{2\delta_1}{n} y_i \bx_i.
\end{align*}
We define $A_{i,1} = 2\delta_1(\bx_i^{\top} \hat\bbeta_0 - y_i) / n$, then $\hat\bbeta_{1, -i} = \hat\bbeta_1 + A_{i, 1} \bx_i$. Now suppose $\hat\bbeta_{k,-i}$ admits the decomposition
\begin{align*}
	\hat\bbeta_{k, -i} = \hat\bbeta_k + A_{i, k} \bx_i + \sum_{j = 1}^{k - 1} B_{i, k}^{(j)} (\bX^{\top} \bX)^j \bx_i
\end{align*}
for some $A_{i,k}, B_{i,k}^{(j)} \in \R$. Then, in the next step of gradient descent, by definition we have
\begin{align*}
	\hat\bbeta_{k + 1, -i} &=  \hat\bbeta_{k, -i} - \frac{2\delta_{k}}{n}\bX_{-i}^{\top} \bX_{-i} \hat\bbeta_{k, -i} + \frac{2\delta_{k}}{n} \bX_{-i}^{\top} \by_{-i} \\
	&= \hat\bbeta_{k + 1} + A_{i, k} \bx_i + \sum_{j = 1}^{k - 1} B_{i, k}^{(j)} (\bX^{\top} \bX)^j \bx_i - \frac{2\delta_{k} A_{i, k}}{n} \bX^{\top} \bX \, \bx_i - \sum_{j = 1}^{k - 1} \frac{2\delta_{k} B_{i, k}^{(j)}}{n}(\bX^{\top} \bX)^{j + 1} \bx_i \\
	& \quad + \frac{2\delta_{k}A_{i,k} \|\bx_i\|_2^2}{n} \bx_i + \sum_{j = 1}^{k - 1}\frac{2\delta_{k} B_{i,k}^{(j)} \bx_i^{\top}(\bX^{\top} \bX)^j \bx_i}{n} \bx_i + \frac{2\delta_{k+1}(\bx_i^{\top}\hat\bbeta_k - y_i)}{n} \bx_i.
\end{align*}
As a result, we obtain the following update equations:
\begin{align*}
	& A_{i, k + 1} = A_{i, k} + \frac{2\delta_{k}A_{i,k} \|\bx_i\|_2^2}{n} + \sum_{j = 1}^{k - 1}\frac{2\delta_{k} B_{i,k}^{(j)} \bx_i^{\top}(\bX^{\top} \bX)^j \bx_i}{n} + \frac{2\delta_{k+1}(\bx_i^{\top}\hat \bbeta_k - y_i)}{n}, \\
	& B_{i, k + 1}^{(1)} = B_{i, k}^{(1)} - \frac{2\delta_{k} A_{i, k}}{n}, \\
	& B_{i, k + 1}^{(j)} = B_{i, k}^{(j)} - \frac{2\delta_{k} B_{i,k}^{(j-1)}}{n}, \quad 2 \leq j \leq k, 
\end{align*}
where we make the convention that  $B_{i,k}^{(k)} = 0$.
\end{proof}

\clearpage

\newgeometry{left=1in,top=0.65in,right=1in,bottom=0.65in,head=.1in}

\section{Additional numerical illustrations and setup details}
\label{sec:additional-numerical-illustrations}

\subsection{Additional illustrations of GCV and risk asymptotic mismatch}
\label{sec:mismatch_illustration_varying_snr}

We provide further visualizations of the asymptotic mismatch between GCV and risk for varying signal-to-noise ratios, as promised in \Cref{sec:combined_sum_mismatch}.
We vary the signal energy $r^2$ for fixed noise energy $\sigma^2$ in \Cref{sec:moderate_signal_energy,sec:high_signal_energy}, and vice versa in \Cref{sec:low_noise_energy,sec:verylow_noise_energy}.

\subsubsection{Moderate signal-to-noise ratio}
\label{sec:moderate_signal_energy}

\begin{figure*}[!h]
    \centering
  \includegraphics[width=0.495\textwidth]{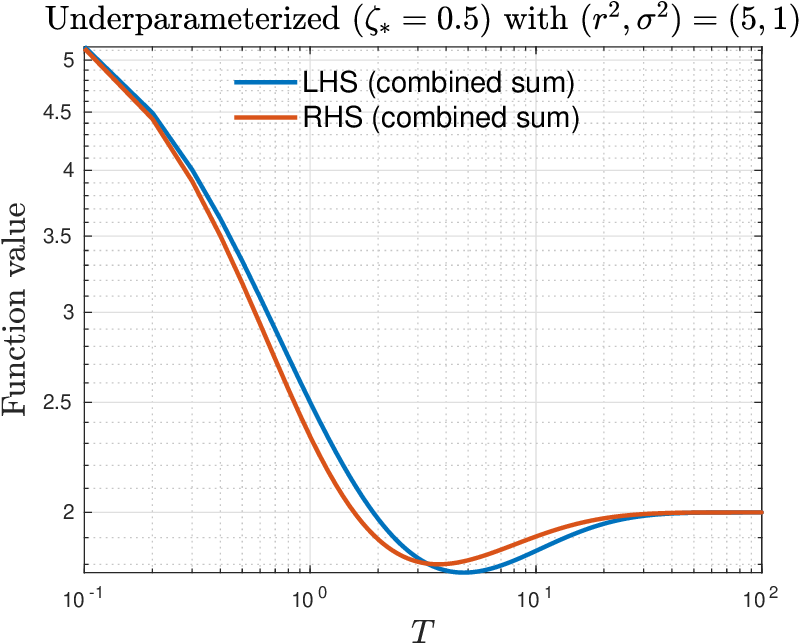}
  \includegraphics[width=0.485\textwidth]{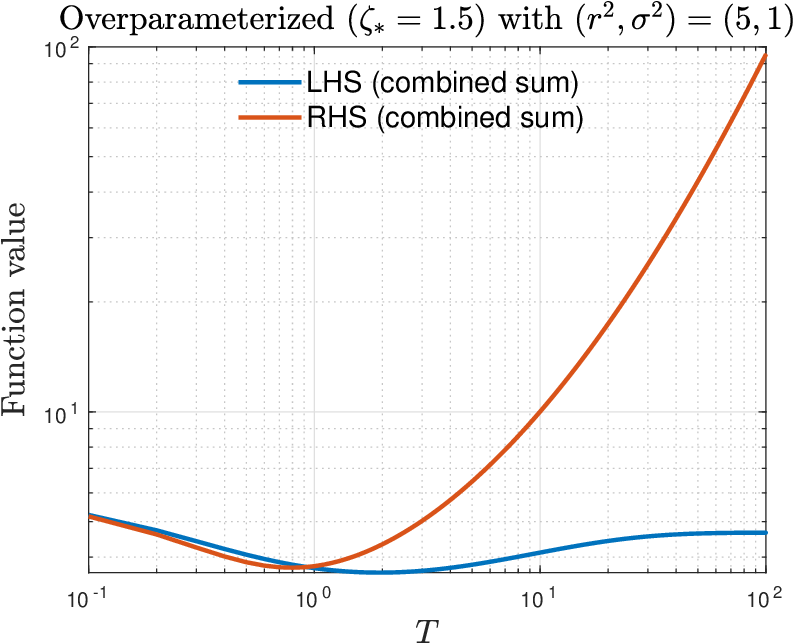}
  \caption{Comparison of the LHS and RHS in \eqref{eq:risk_gcv_asymp_mismatch} (combined sum) for the underparameterized (\emph{left}) and overparameterized (\emph{right}) regimes with moderate $\SNR = 5$.}
  \label{fig:gf_limit_mismatch_in_t_sum_sig_energy_5}
\end{figure*}

\begin{figure*}[!ht]
    \centering
    \includegraphics[width=0.8\textwidth]{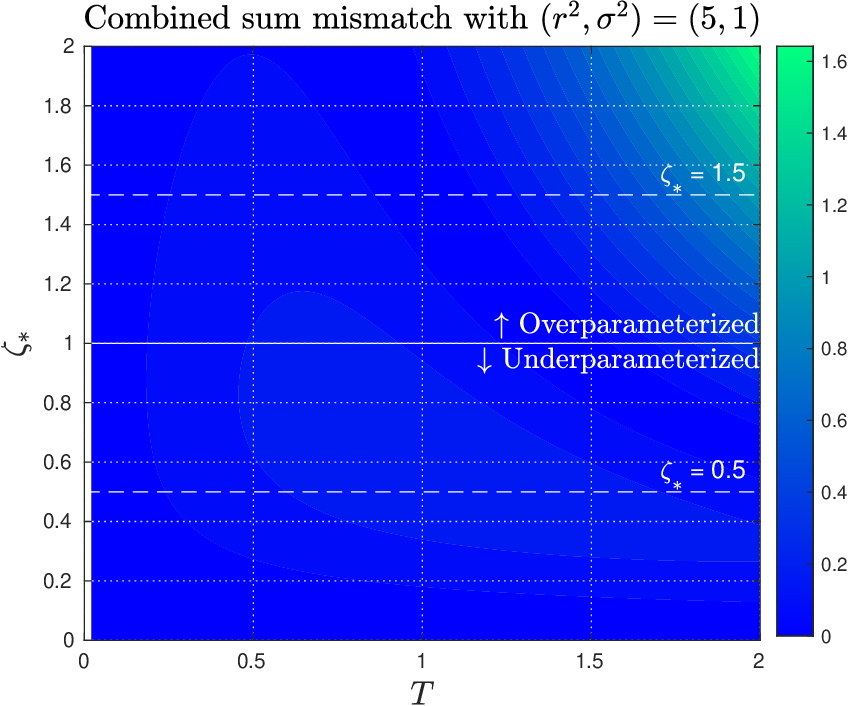}
    \caption{Contour plot of the absolute value of the difference between LHS and RHS of \eqref{eq:risk_gcv_asymp_mismatch} with $\SNR = 5$.
    The mismatch worsens with increasing signal energy, per our calculations in \Cref{sec:combined_sum_mismatch}.}
    \label{fig:gf_limit_mismatch_surface_sum_moderate_signal_energy}
\end{figure*}

\restoregeometry

\subsubsection{High signal-to-noise ratio}
\label{sec:high_signal_energy}

\begin{figure*}[!h]
    \centering
  \includegraphics[width=0.495\textwidth]{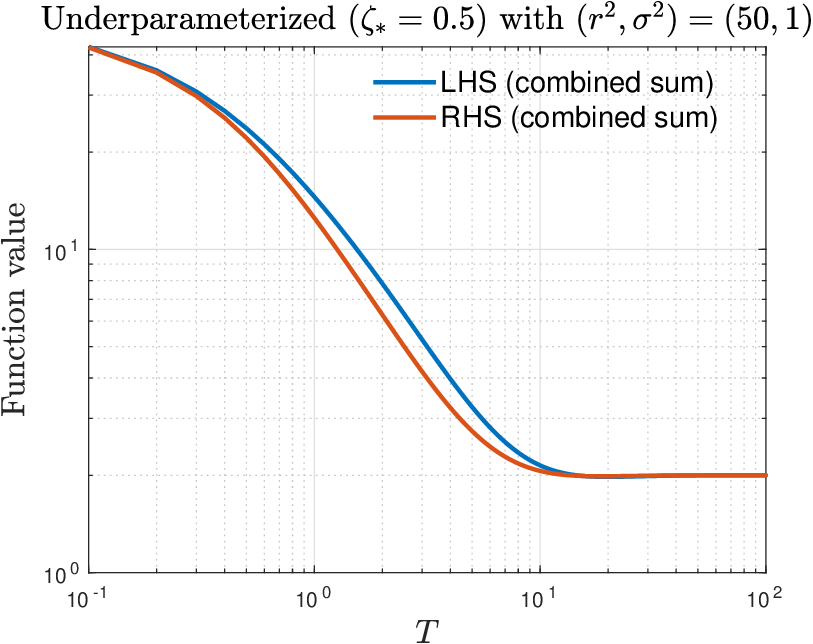}
  \includegraphics[width=0.485\textwidth]{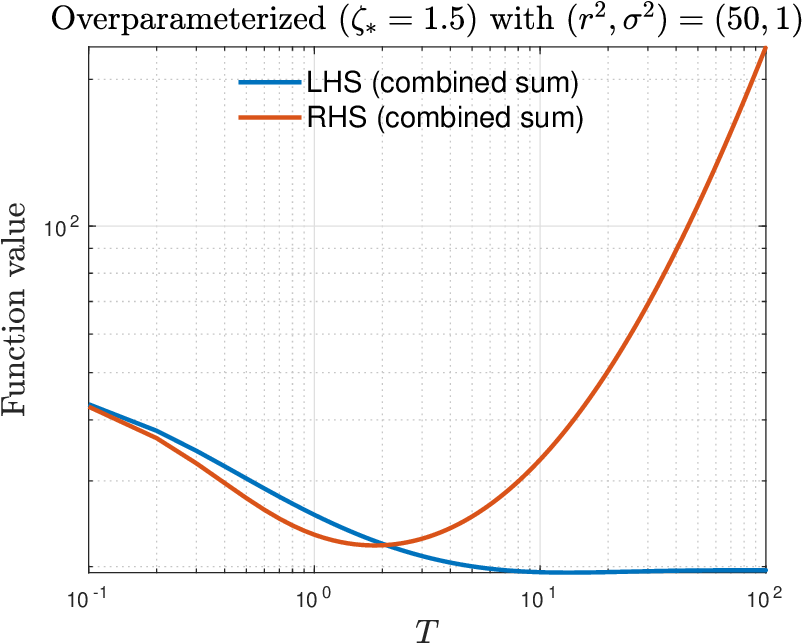}
  \caption{Comparison of the LHS and RHS in \eqref{eq:risk_gcv_asymp_mismatch} (combined sum) for the underparameterized (\emph{left}) and overparameterized (\emph{right}) regimes with high $\SNR = 50$.}
  \label{fig:gf_limit_mismatch_in_t_sum_sig_energy_50}
\end{figure*}

\bigskip

\begin{figure*}[!ht]
    \centering
    \includegraphics[width=0.8\textwidth]{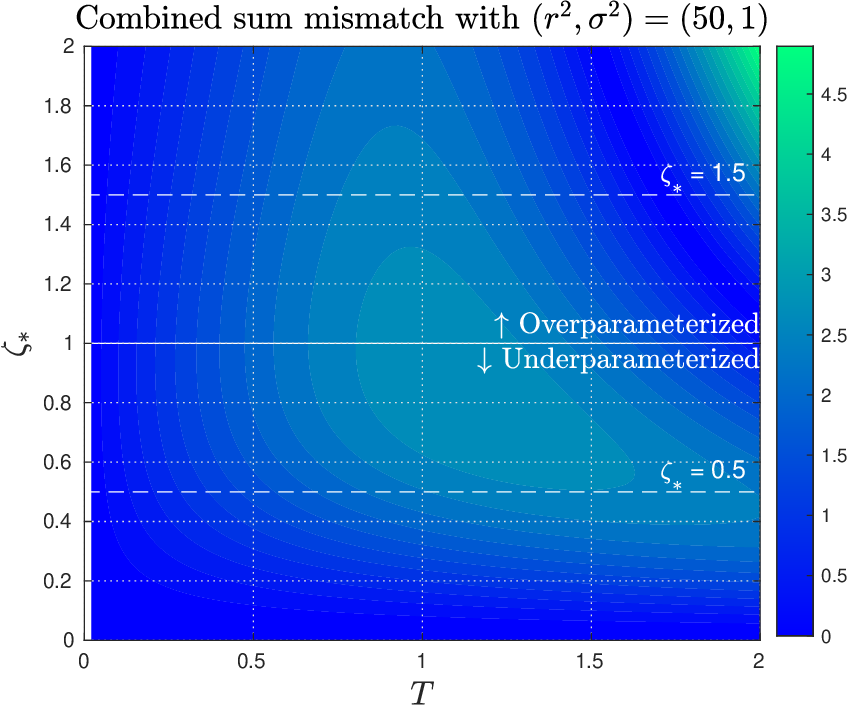}
    \caption{Contour plot of the absolute value of the difference between LHS and RHS of \eqref{eq:risk_gcv_asymp_mismatch} with $\SNR = 50$.
    It is visually apparent that the mismatch gets worse with increasing signal energy, per our calculations in \Cref{sec:combined_sum_mismatch}.}
    \label{fig:gf_limit_mismatch_surface_sum_high_signal_energy}
\end{figure*}

\clearpage
\subsubsection{Low signal-to-noise ratio}
\label{sec:low_noise_energy}

\begin{figure*}[!ht]
    \centering
  \includegraphics[width=0.495\textwidth]{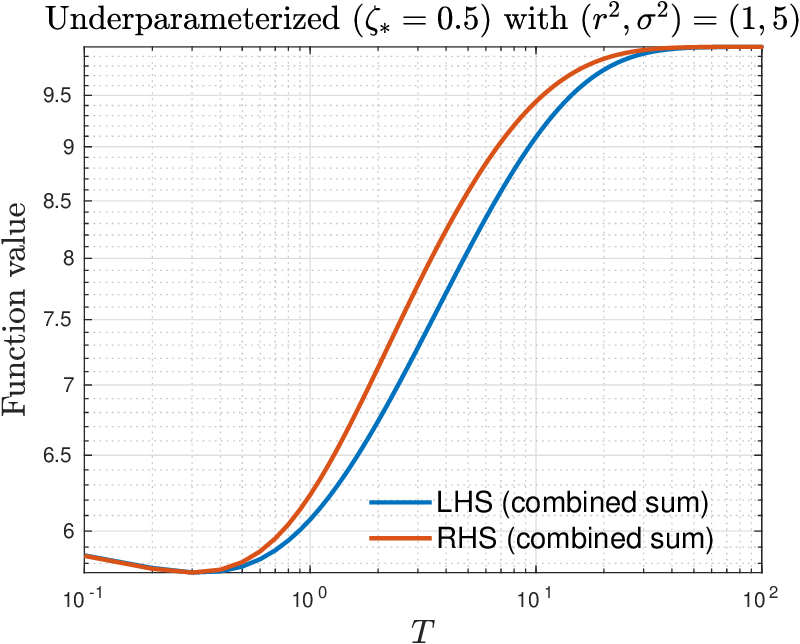}
  \includegraphics[width=0.485\textwidth]{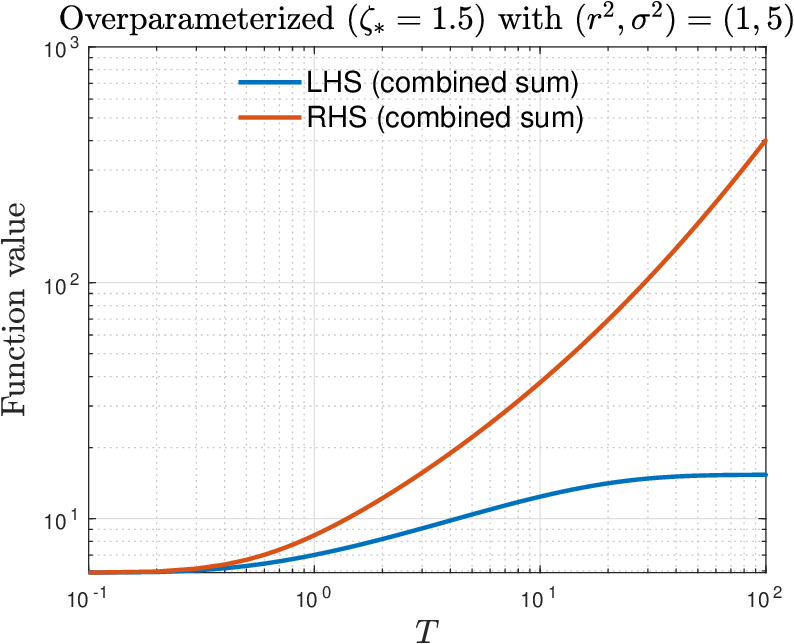}
  \caption{Comparison of the LHS and RHS in \eqref{eq:risk_gcv_asymp_mismatch} (combined sum) for the underparameterized (\emph{left}) and overparameterized (\emph{right}) regimes with low $\SNR = 0.2$.}
  \label{fig:gf_limit_mismatch_in_t_sum_noi_energy_5}
\end{figure*}

\bigskip

\begin{figure*}[!ht]
    \centering
    \includegraphics[width=0.8\textwidth]{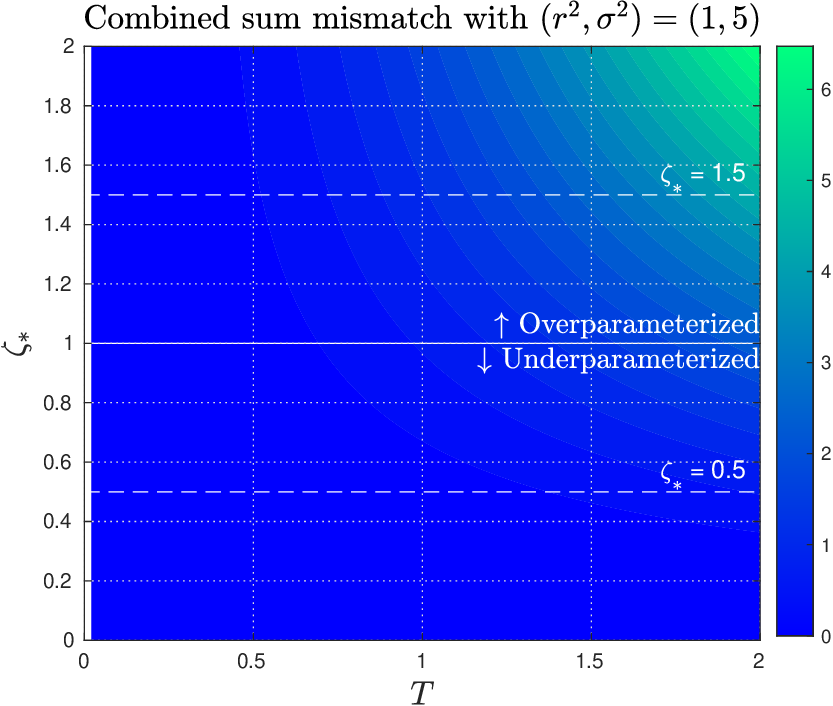}
    \caption{Contour plot of the absolute value of the difference between LHS and RHS of \eqref{eq:risk_gcv_asymp_mismatch} with $\SNR = 0.2$.
    We observe that the mismatch becomes worse with increasing noise energy.
    While the contours may look visually very similar, note that the range of values is higher in the right panel.
    The illustration is in line with our calculations in \Cref{sec:combined_sum_mismatch}.}
    \label{fig:gf_limit_mismatch_surface_sum_low_noise_energy}
\end{figure*}

\clearpage
\subsubsection{Very low signal-to-noise ratio}
\label{sec:verylow_noise_energy}

\begin{figure*}[!ht]
    \centering
  \includegraphics[width=0.495\textwidth]{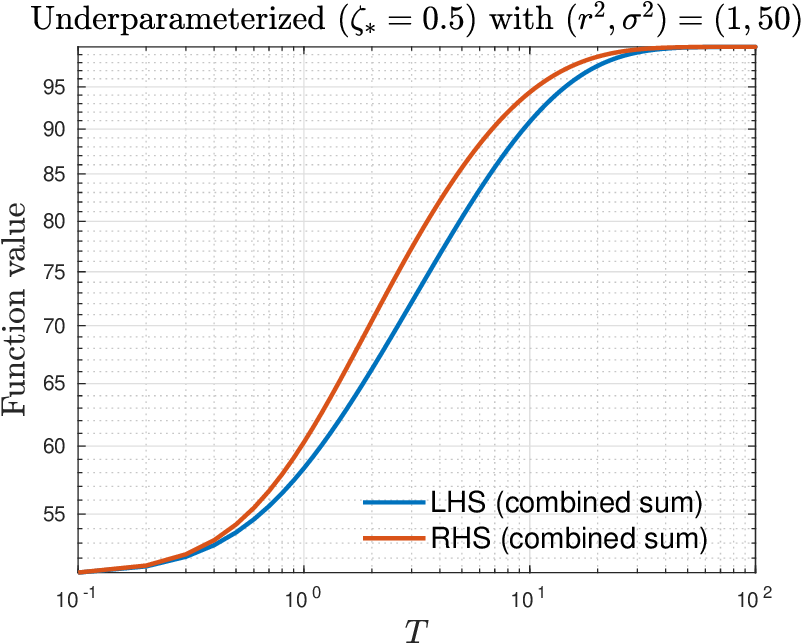}
  \includegraphics[width=0.485\textwidth]{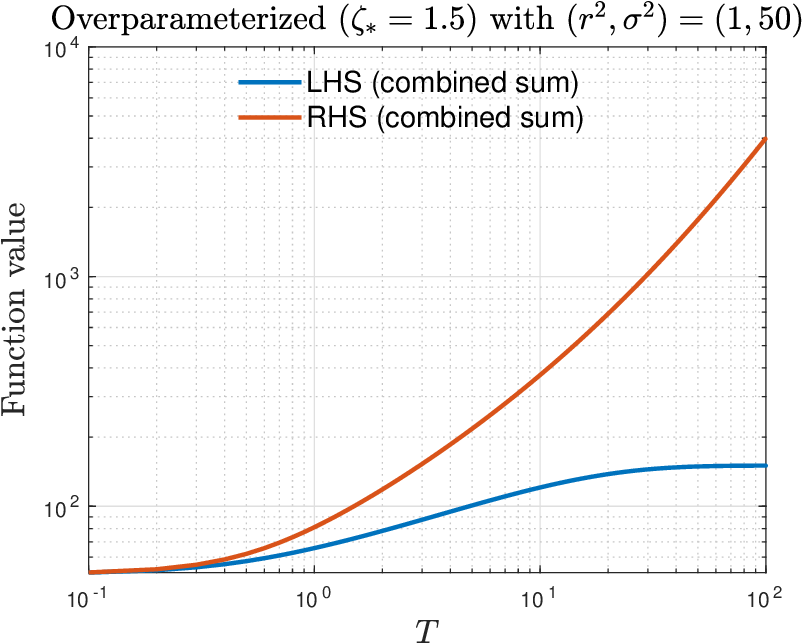}
  \caption{Comparison of the LHS and RHS in \eqref{eq:risk_gcv_asymp_mismatch} (combined sum) for the underparameterized (\emph{left}) and overparameterized (\emph{right}) regimes with very low $\SNR = 0.02$.}
  \label{fig:gf_limit_mismatch_in_t_sum_noi_energy_50}
\end{figure*}

\bigskip

\begin{figure*}[!ht]
    \centering
    \includegraphics[width=0.8\textwidth]{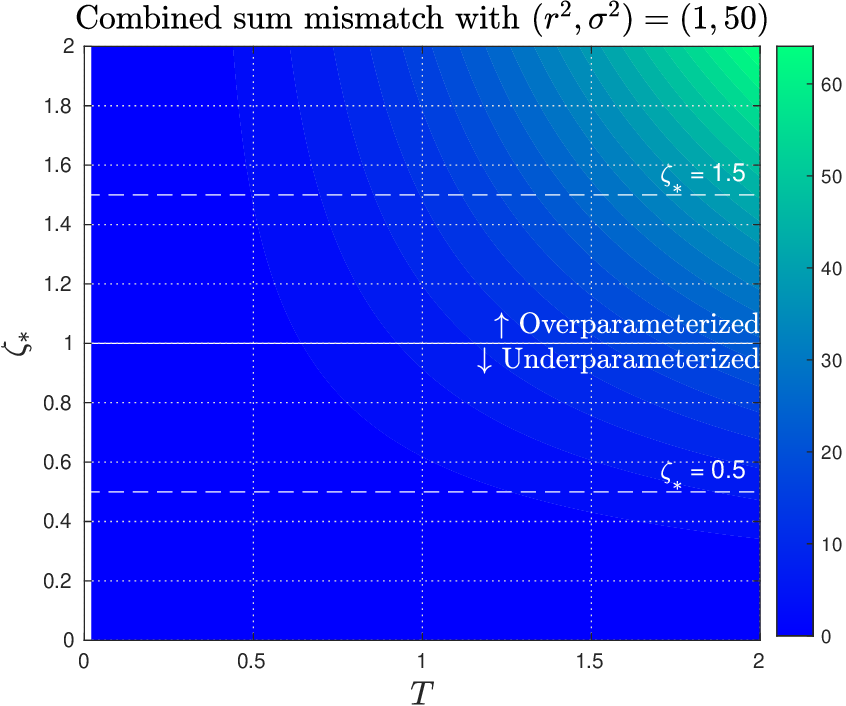}
    \caption{Contour plot of the absolute value of the difference between LHS and RHS of \eqref{eq:risk_gcv_asymp_mismatch} with $\SNR = 0.02$.
    We observe that the mismatch becomes worse with increasing noise energy.
    Although the contours may look visually very similar, note that the range of values is higher in the right panel.
    The illustration agrees with the calculations in \Cref{sec:combined_sum_mismatch}.}
    \label{fig:gf_limit_mismatch_surface_sum_verylow_noise_energy}
\end{figure*}

\clearpage
\subsection{Additional setup details}
\label{sec:setup-details}

\bigskip

\subsubsection{Setup details for \Cref{fig:gcv-inconsistency-with-loocv-n2500-main}}
\label{sec:fig:gcv-inconsistency-with-loocv-n2500-main}

\begin{itemize}
    \item Feature model:
    The feature vector $\bx_i \in \RR^{p}$ is generated according to  
      $\bx_i \sim \cN(\bm{0}, \id_p)$.
    \item Response model:
    Given feature vector $\bx_i$ for $i \in [n]$, the response variable $y_i \in \RR$ is generated according to
    $\by_i = \bx_i^{\top} \bbeta_0 + \eps_i$.
    where
    $\eps_i \sim \cN(0, \sigma^2)$
    with
    $\sigma^2 = 1$.
    \item Signal model:
    The signal vector is generated according to $\bbeta_0 \sim \cN(\bm{0}, r^2 p^{-1} \bI_p)$ with $r^2 = 5$.
\end{itemize}

\bigskip

\subsubsection{Setup details for \Cref{fig:pred-intervals}}
\label{sec:fig:pred-intervals}

\begin{itemize}
    \item Feature model:
    The feature $\bx_i \in \RR^{p}$ is generated according to
    \begin{equation}
        \label{eq:feature_model}
        \bx_i = \bSigma^{1/2} \bz_i,
    \end{equation}
    where $\bz_i \in \RR^{p}$ contains independently sampled entries
    from a common distribution,
    and $\bSigma \in \RR^{p \times p}$ is a positive semidefinite
    feature covariance matrix.
    We use an autoregressive covariance structure such that $\bSigma_{ij} = \rho^{|i-j|}$ for all $i, j$ with parameter $\rho = 0.25$.
    \item Response model:
    Given $\bx_i$, the response $y_i \in \RR$ is generated according to
    \begin{equation}
        \label{eq:response_model}
        y_i = \bbeta_0^\top \bx_i 
        + \bigl( \bx_i^\top \bA \bx_i - \tr[\bA \bSigma] \bigr) / p 
        + \eps_i,
    \end{equation}
    where $\bbeta_0 \in \RR^{p}$ is a fixed signal vector,
    $\bA \in \RR^{p \times p}$ is a fixed matrix,
    and $\eps_i \in \RR$ is a random noise variable.
    Note that we have subtracted the mean from the squared nonlinear component
    and scaled it to keep the variance of the nonlinear component
    at the same order as the noise variance
    (see \cite{mei_montanari_2022} for more details, for example).
    We again use Student's $t$ distribution for the random noise component, which is again standardized so that the mean is zero and the variance is one. 
    \item Signal model:
    We align the signal $\bbeta_0$ with the top eigenvector (corresponding to the largest eigenvalue) of the covariance matrix $\bSigma$.
    More precisely, suppose that $\bSigma = \bW \bR \bW^\top$ denotes the eigenvalue decomposition of the covariance matrix $\Sigma$, where $W \in \RR^{p \times p}$ is an orthogonal matrix whose columns $w_1, \dots, w_p$ are eigenvectors of $\Sigma$ and $R \in \RR^{p \times p}$ is a diagonal matrix whose entries $r_1 \ge \dots \ge r_p$ are eigenvalues of $\Sigma$ in descending order.
    We then let $\bbeta_0 = c \bw_1$, where $c$ controls the effective signal energy.
    We refer to the value of $\bbeta_0^\top \bSigma \bbeta_0$ as the effective signal energy, which is set at $50$.
    It is worth noting that even though the regression function above does not satisfy the assumptions of \Cref{asm:feat_dist_loocv}, it is easy to see that the function is approximately Lipschitz.
\end{itemize}

\clearpage
\subsection{Additional illustration for predictive intervals based on LOOCV}
\label{sec:prediction-intervals-linear-model}

See \Cref{fig:pred-intervals-linear-model} for an additional illustration of the prediction intervals based on LOOCV where the optimal stopping occurs at an intermediate iteration.
This is in contrast to \Cref{fig:pred-intervals} where optimal stopping occurs at a far enough iteration, due to the ``latent signal'' structure.
For \Cref{fig:pred-intervals-linear-model}, we use an isotropic setup under a linear model, similar to that of \Cref{fig:gcv-inconsistency-with-loocv-n2500-main}.

For the sake of completeness, the details are described below:
\begin{itemize}
    \item Feature model:
    The feature vector $\bx_i \in \RR^{p}$ is generated according to  
      $\bx_i \sim \cN(\bm{0}, \id_p)$.
    \item Response model:
    Given feature vector $\bx_i$ for $i \in [n]$, the response variable $y_i \in \RR$ is generated according to $\by_i = \bx_i^{\top} \bbeta_0 + \eps_i$,
    where
    $\eps_i \sim \cN(0, \sigma^2)$
    with
    $\sigma^2 = 1$.
    \item Signal model: 
    The signal vector is generated according to $\bbeta_0 \sim \cN(\bm{0}, r^2 p^{-1} \bI_p)$ with $r^2 = 5$.
\end{itemize}

\bigskip

\begin{figure}[!ht]
    \centering
  \includegraphics[width=0.99\textwidth]{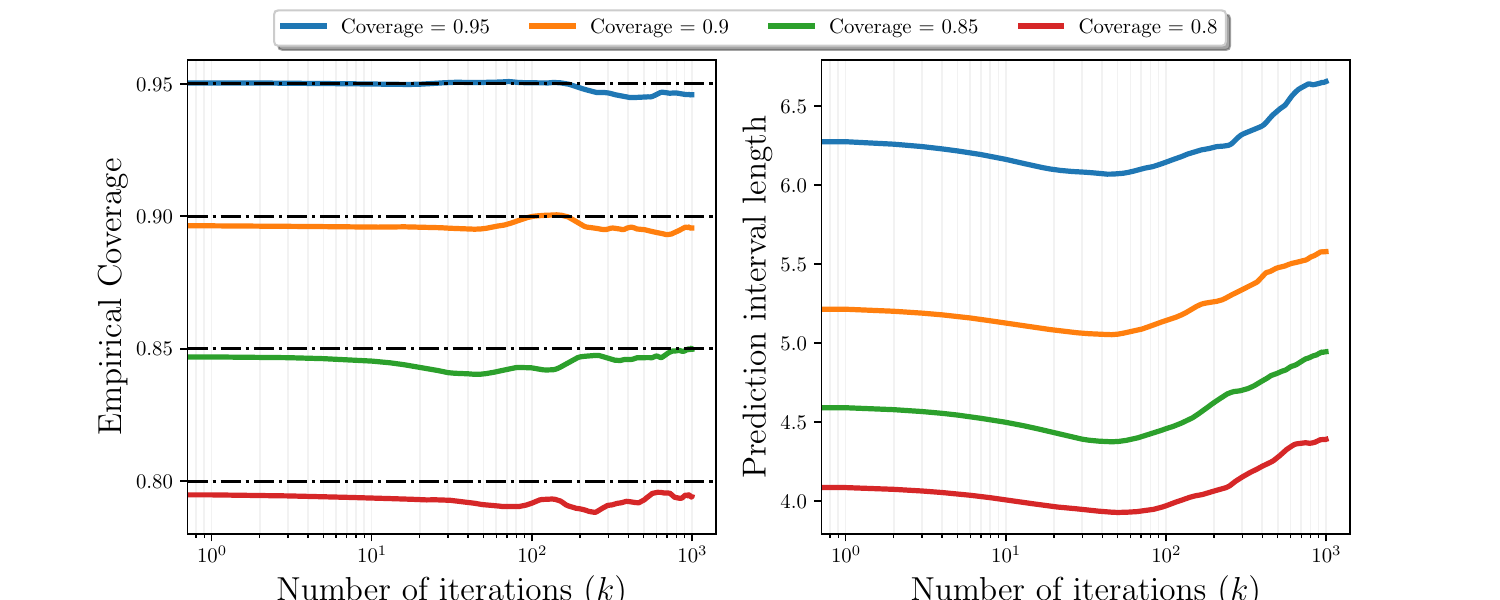}
  \vspace{1em}
  \caption{
    \textbf{LOOCV provides prediction intervals with near-consistent coverage in finite samples across different nominal coverage levels.}
    We consider an overparameterized regime, where the number of observations is $n=2500$ and the number of features is
    $p=5000$ (overparameterized).
    The non-intercept features are Gaussian with a $\rho$-autoregressive covariance
    $\Sigma$ (such that $\Sigma_{ij} = \rho^{|i - j |}$ for all $i,j$) with
    $\rho=0.25$.
    The response is generated from a linear model with a nonrandom
    signal vector $\bbeta_0$ that has unit Euclidean norm. 
    We initialize the GD process randomly and employ a universal step size $\delta = 0.01$. 
    In the \emph{left} panel, we plot the empirical coverage rates with various levels, and in the \emph{right} panel, we plot the length of the prediction intervals. 
    All simulation outcomes are based on one realization of $(\bX, \by)$.
}
  \label{fig:pred-intervals-linear-model}
\end{figure}

\clearpage

\subsection{Additional illustrations for \Cref{sec:extension-general-risk-functionals}}
\label{sec:additional-illustrations-distributional-closeness}

\subsubsection{Squared and absolute risk and LOOCV plug-in functionals}

\begin{figure*}[!h]
    \centering
    \includegraphics[width=0.55\textwidth]{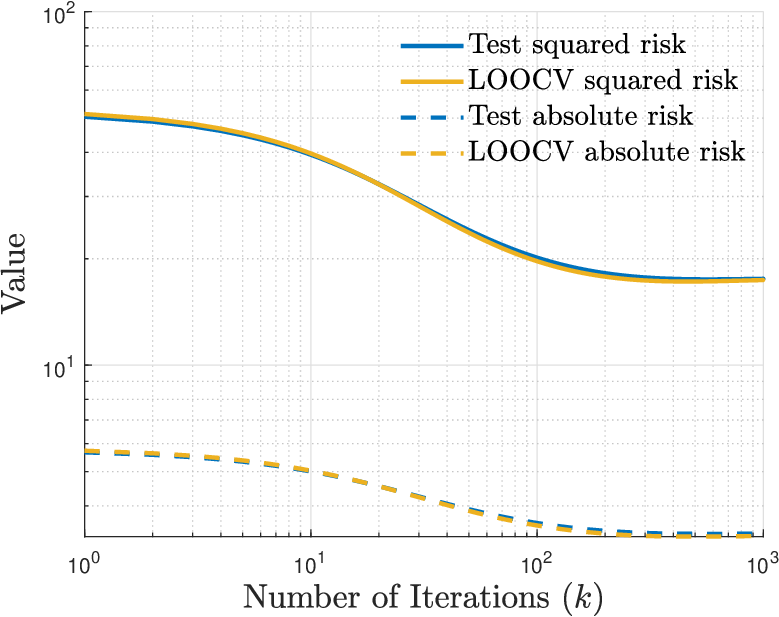}
    \caption{
    \textbf{LOOCV plug-in functionals are consistent for both squared and absolute error functionals.}
    We use the same setup as shown in Figure \ref{fig:test-loo-dist-comparison} to demonstrate the consistency for the squared error and absolute error functionals.
    }
    \label{fig:test-loo-dist-comparison-sq-vs-abs}
\end{figure*}

\subsubsection{Ridgeline plot of test error and LOOCV error distributions}

\begin{figure*}[!h]
    \centering
    \includegraphics[width=0.75\textwidth]{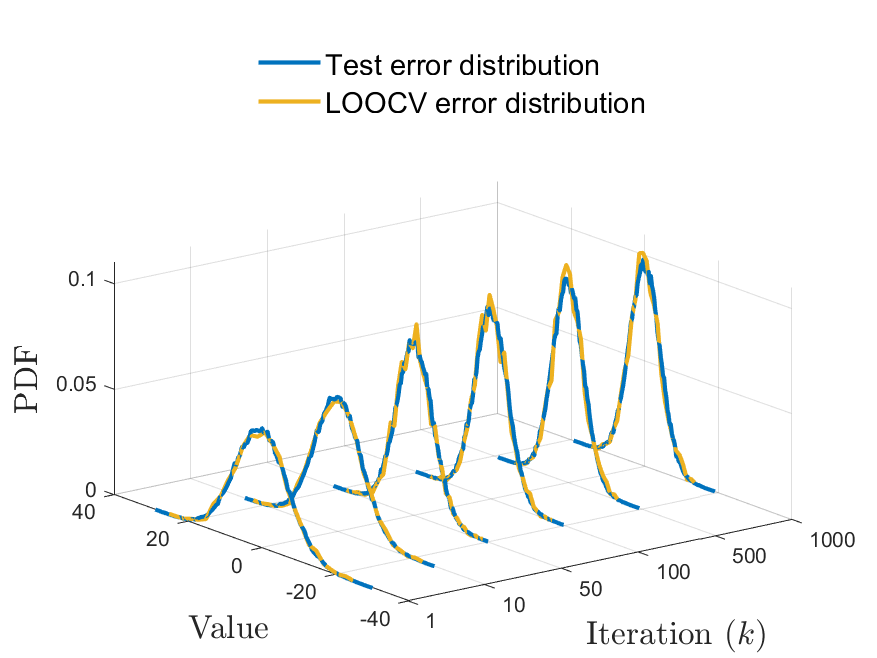}
    \caption{
    \textbf{Empirical distribution of LOOCV errors tracks the true test error distribution along the entire gradient descent path.}
    We use the same setup as in \Cref{fig:test-loo-dist-comparison}, but now visualize the evolution of the associated distribution in a single iteration-distribution ridgeline plot.
    }
    \label{fig:test-loo-dist-comparison-ridges}
\end{figure*}

\end{document}